\documentclass[letterpaper,10 pt,conference]{ieeeconf}
\IEEEoverridecommandlockouts

\usepackage{times}
\let\labelindent\relax
\usepackage{enumitem}

\usepackage[dvipsnames]{xcolor}

\usepackage{xspace}

\usepackage{etoolbox}

\usepackage{amsmath,amssymb,mathtools,nccmath, stmaryrd}
\allowdisplaybreaks

\usepackage{amsthm}
\usepackage{bm}
\usepackage{bbold,dsfont,bbm}

\usepackage{graphicx}
\usepackage{tikz,tikz-cd,pgfplots}
\usepackage[export]{adjustbox}
\pgfplotsset{compat=1.17}

\usepackage{multirow}
\usepackage{booktabs}

\usepackage{comment}

\usepackage{units}
\usepackage[per-mode=symbol]{siunitx}[=v2]
\DeclareSIUnit{\eur}{\euro}
\DeclareSIUnit{\usd}{USD}
\DeclareSIUnit{\mph}{mph}
\DeclareSIUnit{\month}{month}
\DeclareSIUnit{\year}{year}
\DeclareSIUnit{\million}{Mil}
\DeclareSIUnit{\mile}{mile}
\DeclareSIUnit{\car}{car}
\DeclareSIUnit{\train}{train}
\DeclareSIUnit{\mmveh}{\text{$\mu$}MV}
\DeclareSIUnit{\nounit}{-}
\usepackage[font=footnotesize]{caption}

\usepackage{subcaption}

\usepackage{array}
\newcommand{\PreserveBackslash}[1]{\let\temp=\\#1\let\\=\temp}
\newcolumntype{C}[1]{>{\PreserveBackslash\centering}p{#1}}
\newcolumntype{R}[1]{>{\PreserveBackslash\raggedleft}p{#1}}
\newcolumntype{L}[1]{>{\PreserveBackslash\raggedright}p{#1}}

\renewcommand{\baselinestretch}{0.961}

\theoremstyle{definition}
\newtheorem{assumption}{Assumption}
\newtheorem*{assumption*}{Assumption}
\newtheorem{theorem}{Theorem}

\newtheorem{lemma}{Lemma}

\newtheorem{definition}{Definition}
\newtheorem{counterexample}{Counterexample}
\theoremstyle{remark}
\newtheorem{remark}{Remark}
\newtheorem{example}{Example}

\usepackage[acronym]{glossaries-extra}
\setabbreviationstyle[acronym]{long-short}
\glssetcategoryattribute{acronym}{nohyperfirst}{true}

\makeatletter
\let\NAT@parse\undefined
\makeatother

\usepackage[hidelinks, citecolor=blue, urlcolor=cyan]{hyperref}
\usepackage[capitalise,noabbrev]{cleveref}
\crefname{counterexample}{counterexample}{counterexamples}
\Crefname{counterexample}{Counterexample}{Counterexamples}
\usepackage{cite}

\Crefname{figure}{Fig.}{Figures}
\crefname{equation}{Eq.}{Eqs.}

\makeatletter
    \if@cref@capitalise
        \crefname{subsection}{Section}{Sections}
        \crefname{subsubsection}{Section}{Sections}
        \crefname{assumption}{Assumption}{Assumptions}
        \crefname{problem}{Problem}{Problems}
    \else
        \crefname{subsection}{section}{sections}
        \crefname{subsubsection}{section}{sections}
        \crefname{assumption}{assumption}{assumptions}
        \crefname{problem}{problem}{problems}
    \fi
\makeatother

\usetikzlibrary{positioning,intersections, hobby, patterns, calc, decorations.pathmorphing, decorations.markings, shadows,shapes, cd,arrows.meta, fit,quotes}
\tikzset{
   tick/.style={postaction={
      decorate,
      decoration={markings, mark=at position 0.5 with {\draw[-] (0,.4ex) -- (0,-.4ex);}}}
   }
}
\tikzstyle{block} = [draw, rectangle, minimum height=2em, minimum width=3em,font=\bfseries,rounded corners,thick]
\tikzstyle{block1} = [draw, rectangle, minimum height=1.5em, minimum width=2.5em]
\tikzstyle{blockDyn} = [draw, rectangle, minimum height=2.5em, minimum width=3.5em, align=center, inner sep=10pt, thick, fill=white, copy shadow={draw=black,fill=black,opacity=1,shadow xshift=0.5ex,shadow yshift=-0.5ex}]
\tikzstyle{blockAlg} = [draw, rectangle, minimum height=1.5em, minimum width=2.5em, align=center, inner sep=10pt, thick]
\tikzstyle{sum} = [draw,circle]

\tikzstyle{blockfill} = [block,rounded corners=4,fill=white]

\tikzstyle{nodePre} = [circle, draw,inner sep=1pt,node contents={$\preceq$},thick]
\tikzstyle{nodePreEmpty} = [circle, draw,inner sep=1pt,thick]
\tikzstyle{nodePos} = [circle, draw,inner sep=1pt,node contents={$\posceq$},thick]
\tikzstyle{nodeProd} = [rectangle, draw,inner sep=4pt,node contents={$\times$},rounded corners,thick]
\tikzstyle{nodeSum} = [rectangle, draw,inner sep=4pt,node contents={$\mathbf{+}$},rounded corners,thick]

\definecolor{DPgreen}{rgb}{0.0, 0.5, 0.0}
\definecolor{red}{rgb}{0.75, 0.0, 0.0}

\newif\ifmargincomments 
\margincommentstrue

\newif\ifextendedversion

\ifmargincomments

\else

\fi

\definecolor{lightblue}{rgb}{0.60784,0.76078,0.90196}
\definecolor{darkblue}{rgb}{0.26667,0.44706,0.76863}
\definecolor{lightgreen}{rgb}{0.66275,0.81569,0.55686}
\definecolor{darkgreen}{rgb}{0.43922,0.67843,0.27843}
\definecolor{orange}{rgb}{0.92941,0.49020,0.19216}
\definecolor{yellow}{rgb}{1.00000,0.75294,0.00000}
\definecolor{grey}{rgb}{0.64706,0.64706,0.64706}
\definecolor{purple}{rgb}{0.51373,0.23529,0.04706}

\usetikzlibrary{positioning,intersections, hobby, patterns, calc, decorations.pathmorphing, decorations.markings, shadows,shapes, cd,arrows.meta, fit,quotes}

\tikzset{
   tick/.style={postaction={
      decorate,
      decoration={markings, mark=at position 0.5 with {\draw[-] (0,.4ex) -- (0,-.4ex);}}}
   }
}
\tikzstyle{block} = [draw, rectangle, minimum height=2em, minimum width=3em,font=\bfseries,rounded corners,thick]
\tikzstyle{block} = [draw, rectangle, minimum height=2em, minimum width=3em]
\tikzstyle{block1} = [draw, rectangle, minimum height=1.5em, minimum width=2.5em]
\tikzstyle{blockDyn} = [draw, rectangle, minimum height=2.5em, minimum width=3.5em, align=center, inner sep=10pt, thick, fill=white, copy shadow={draw=black,fill=black,opacity=1,shadow xshift=0.5ex,shadow yshift=-0.5ex}]
\tikzstyle{blockAlg} = [draw, rectangle, minimum height=1.5em, minimum width=2.5em, align=center, inner sep=10pt, thick]
\tikzstyle{sum} = [draw,circle]

\tikzstyle{nodePre} = [circle, draw,inner sep=1pt,node contents={$\preceq$},thick]
\tikzstyle{nodePreEmpty} = [circle, draw,inner sep=1pt,thick]
\tikzstyle{nodePos} = [circle, draw,inner sep=1pt,node contents={$\posceq$},thick]
\tikzstyle{nodeProd} = [rectangle, draw,inner sep=4pt,node contents={$\times$},rounded corners,thick]
\tikzstyle{nodeSum} = [rectangle, draw,inner sep=4pt,node contents={$\mathbf{+}$},rounded corners,thick]

\definecolor{red}{rgb}{0.75, 0.0, 0.0}

\tikzset{fcname/.store in =\fcname, fcname={}}
\tikzset{funame/.store in =\funame, funame={}}
\tikzset{rcname/.store in =\rcname, rcname={}}
\tikzset{runame/.store in =\runame, runame={}}
\tikzset{whereres/.store in =\whereres, whereres=0.5}
\tikzset{wherefun/.store in =\wherefun, wherefun=0.5}
\tikzset{relres/.store in =\relres, relres={above}}
\tikzset{relfun/.store in =\relfun, relfun={above}}
\tikzset{posres/.store in =\posres, posres=1}
\tikzset{posfun/.store in =\posfun, posfun=1}
\tikzset{loos/.store in =\loos, loos=2}
\tikzset{feedback/.store in =\feedback, feedback=0}
\tikzset{
   DP/.style={
      label/.style={
         font=\everymath\expandafter{\the\everymath\scriptstyle},
         inner sep=5pt,
         node distance=2pt and -2pt},
      semithick,
      node distance=1 and 1,
      rconn/.style={color=white,opacity=0.0,postaction={decorate}, shorten <=3.2pt, shorten >= 0.8,
      decoration={markings, 
      mark= at position 0 with {
               \coordinate (a);
      },
      mark=at position .5 with
      {
              \ifthenelse{\equal{\feedback}{1}}{\def\angleOut{90}\def\angleIn{90}}{\def\angleOut{0}\def\angleIn{180}}    
              \coordinate (b);
              \draw[dashed,dpred,opacity=1.0] (a) to[out=\angleOut,in=\angleIn,looseness=\loos] 
              node[pos=\posres,\relres=\whereres mm,dpred,opacity=1,fill=white,inner sep=1pt,outer sep=1pt]{\footnotesize{\rcname}} (b);
      },
      mark= at position 1 with 
      {
             \ifthenelse{\equal{\feedback}{1}}{\def\angleOut{0}\def\angleIn{0}}{\def\angleOut{180}\def\angleIn{0}} 
              \ifthenelse{\equal{\feedback}{1}}{\def\symbol{\succeq}}{\def\symbol{\preceq}} 
              \coordinate (c);
              \draw[dpgreen,opacity=1.0] (c) to[out=\angleOut,in=\angleIn,looseness=\loos]
              node[pos=\posfun,\relfun=\wherefun mm,dpgreen,opacity=1,fill=white,inner sep=1pt,outer sep=1pt]{\footnotesize{\fcname}} (b){}; 
              \node[draw,circle,inner sep=0.5pt,color=black,fill=white,opacity=1.0] at (b) (nodepreceq) {$\symbol$}; 
      }
      }},
      runconn/.style={color=dpred,dashed,postaction={decorate},
      decoration={markings,
      mark= at position 1 with {
              \coordinate (a);
              \draw[dpred,opacity=1.0,dashed] ($(a)+(0.05,0)$) --++ (0.5,0) node[\relres,pos=\posres]{\footnotesize{\runame}};}
      }
      },
      funconn/.style={color=white,postaction={decorate},
      decoration={markings,
      mark= at position 0 with {
      \coordinate (a);
      \draw[dpgreen] ($(a)+(-0.05,0)$) -- ($(a)+(-0.5,0)$) node[\relfun, pos=\posfun]{\footnotesize{\funame}};}
      }
      },
      execute at begin picture={\tikzset{
         x=\dpx, y=\dpy,
         every fit/.style={inner xsep=\dpx, inner ysep=\dpy}}}
      },
   dpx/.store in=\dpx,
   dpx = 1.5cm,
   dpy/.store in=\dpy,
   dpy = 1.5ex,
   dp port sep/.store in=\dpportsep,
   dp port sep=2,
   dp port length/.store in=\dpportlen,
   dp port length=4pt,
   dp min width/.store in=\dpminwidth,
   dp min width=0.5cm,
   dp rounded corners/.store in=\dpcorners,
   dp rounded corners=2pt,
   dp small/.style={dp port sep=1, dp port length=2.5pt, dpx=.4cm, dp min width=.4cm, dpy=.7ex},
   dp/.code 2 args={
      \pgfmathsetlengthmacro{\dpheight}{\dpportsep * (max(#1,#2)) * \dpy}
      \pgfkeysalso{draw,%
        minimum width=\dpminwidth,%
        minimum height=\dpheight,%
        font=\bfseries,
        outer sep=0pt,%
        inner sep=5pt,%
        rounded corners=\dpcorners,
        thick,
        prefix after command={\pgfextra{\let\fixname\tikzlastnode}},
        append after command={\pgfextra{\draw
            \ifnum #1=0{} \else foreach \i in {1,...,#1} { 
            ($(\fixname.north west)!{\i/(#1+1)}!(\fixname.south west)$) +(0,0) node[solid,left,circle,color=dpgreen,draw,fill=dpgreen,scale=0.3] {} coordinate (\fixname_fun\i) -- +(0,0) coordinate (\fixname_fun\i')}\fi 
            \ifnum #2=0{} \else foreach \i in {1,...,#2} {
            ($(\fixname.north east)!{\i/(#2+1)}!(\fixname.south east)$) +(0,0) coordinate (\fixname_res\i') -- +(0,0) node[solid,right,circle,color=dpred,draw,fill=dpred,scale=0.3] {} coordinate (\fixname_res\i)}\fi;
         }}}
         },
      dp name/.style={append after command={\pgfextra{\node[label=center,inner sep=2pt,fill=white] at (\fixname) {\textbf{#1}};}}}
   }
\newcommand{\newacronymwithcmds}[4]{%
  \newacronym{abk:#1}{#3}{#4}%
  \expandafter\newcommand\csname gls#2\endcsname{\gls{abk:#1}\xspace}
  \expandafter\newcommand\csname Gls#2\endcsname{\Gls{abk:#1}\xspace}
  \expandafter\newcommand\csname glspl#2\endcsname{\glspl{abk:#1}\xspace}
  \expandafter\newcommand\csname Glspl#2\endcsname{\Glspl{abk:#1}\xspace}
}

\newcommand{\glsdefhere}[2]{%
  \glsreset{abk:#2}%
  \hypertarget{glo:abk:#2}{}%
  #1{abk:#2}%
  \glsunset{abk:#2}%
}

\newacronymwithcmds{poset}{poset}{poset}{partially ordered set}

\newacronymwithcmds{cpo}{cpo}{CPO}{complete partial order}
\newacronymwithcmds{dcpo}{dcpo}{DCPO}{directed complete partial order}

\newacronymwithcmds{dp}{dp}{DP}{design problem}
\newacronymwithcmds{cps}{cps}{CPS}{cyber-physical systems}

\newacronymwithcmds{cdp}{cdp}{CDP}{co-design problem}

\newacronymwithcmds{cdpi}{cdpi}{CDPI}{co-design problem with implementation}

\newacronymwithcmds{dpi}{dpi}{DPI}{design problem with implementation}
\newacronymwithcmds{mdpi}{mdpi}{MDPI}{monotone design problem with implementation}
\newacronymwithcmds{mcdp}{mcdp}{MCDP}{Monotone Co-Design Problem}

\newacronymwithcmds{qms}{qms}{QMS}{Quasi Measurable Space}
\newacronymwithcmds{qus}{qus}{QUS}{Quasi Universal Space}

\newacronymwithcmds{ofs}{ofs}{OFS}{outcome feasible set}

\newacronymwithcmds{MarkovKer}{MarkovKer}{Markov kernel}{Markov kernel}

\newacronymwithcmds{mcmc}{mcmc}{MCMC}{Markov chain Monte-Carlo}

\newacronymwithcmds{uav}{uav}{UAV}{unmanned aerial vehicle}
\newacronymwithcmds{vi}{vi}{VI}{variational inference}

\newacronymwithcmds{agc}{agc}{A-G contract}{Assume-Guarantee contract}

\newboolean{proofinappendix}

\newcommand{\Pos}{\mathsf{Pos}}

\newcommand{\QMS}{\mathsf{QMS}}
\newcommand{\QUS}{\mathsf{QUS}}

\newcommand{\qusOf}[2]{\QUS\{#1, #2\}}

\newcommand{\Nats}{\mathbb{N}}
\newcommand{\Reals}{\mathbb{R}}
\newcommand{\RealsNonNeg}{\Reals_{\geq 0}}
\newcommand{\setOfBool}{\mathrm{Bool}}

\newcommand{\realelx}{x}
\newcommand{\boolelb}{b}

\newcommand{\realize}{\mathsf{realize}}

\definecolor{dpred}{rgb}{0.7, 0.0, 0.0}
\definecolor{dpgreen}{rgb}{0.0, 0.5, 0.0}
\definecolor{functorpurple}{rgb}{0.5, 0, 0.5}
\definecolor{imporange}{rgb}{1.0, 0.58, 0.063}
\definecolor{specificationcolor}{rgb}{0.258, 0.527, 0.957}

\newcommand{\imporangetext}[1]{\textcolor{imporange}{#1}}
\newcommand{\specificationtext}[1]{\textcolor{specificationcolor}{#1}}

\newcommand{\F}[1]{\textcolor{dpgreen}{#1}}
\newcommand{\FI}[1]{\F{\textit{#1}}}
\newcommand{\funPosetF}{{\F{F}}}

\newcommand{\R}[1]{\textcolor{dpred}{#1}}
\newcommand{\RI}[1]{\R{\textit{#1}}}
\newcommand{\resPosetR}{{\R{R}}}

\newcommand{\I}[1]{\imporangetext{#1}}
\newcommand{\II}[1]{\I{\textit{#1}}}
\newcommand{\impSetI}{\I{I}}

\newcommand{\impi}{{\I{i}}}

\newcommand{\Sp}[1]{\specificationtext{#1}}
\newcommand{\SpI}[1]{\Sp{\textit{#1}}}
\newcommand{\specSetS}{\Sp{S}}
\newcommand{\specs}{\Sp{s}}

\newcommand{\probquery}{\text{pr}}
\newcommand{\fixfunminres}{\text{Fix\F{Fun}Min\R{Res}}\xspace}

\ifPDFTeX

\else

\fi

\newcommand{\sigAlg}{\sigma}
\newcommand{\sigAlgOf}[1]{\sigAlg(#1)}

\newcommand{\defeq}{\mathrel{\raisebox{-0.3ex}{$\overset{\text{\tiny def}}{=}$}}}

\newcommand{\setWithArg}[2]{\{#1 \mid #2\}}
\newcommand{\setWithIndex}[2]{\{#1\}_{#2}}

\newcommand{\USet}{\mathtt{U}}
\newcommand{\USetOf}[1]{\USet(#1)}
\newcommand{\upperClosure}[1]{\mathord{\uparrow}#1}
\newcommand{\upperOf}[1]{\upperClosure{\set{#1}}}

\newcommand{\lowerClosure}[1]{\mathord{\downarrow}#1}
\newcommand{\lowerOf}[1]{\lowerClosure{\set{#1}}}

\newcommand{\lowerBoundfix}{\mathrm{L}}
\newcommand{\upperBoundfix}{\mathrm{U}}

\newcommand{\DP}{\mathsf{DP}}
\newcommand{\dpOf}[2]{\DP\{#1, #2\}}
\newcommand{\qusdpOf}[2]{\DP_{\text{A}}\{#1, #2\}}
\newcommand{\dprb}{{\mathrm{dp}}}
\newcommand{\dprbOf}[1]{\dprb_#1}
\newcommand{\dprba}{\dprbOf{a}}
\newcommand{\dprbb}{\dprbOf{b}}
\newcommand{\dprbL}{\dprbOf{\lowerBoundfix}}
\newcommand{\dprbU}{\dprbOf{\upperBoundfix}}

\newcommand{\setproduct}{\times}
\newcommand{\posetproduct}{\times}

\newcommand{\mapFollowing}{\mathbin{\circ}}

\newcommand{\mthen}{\fatsemi}
\newcommand{\mthenMathbin}{\mathbin{\mthen}}
\newcommand{\stack}{\otimes}
\newcommand{\stackMathbin}{\mathbin{\stack}}
\newcommand{\trace}{\text{Tr}}
\newcommand{\traceOf}[1]{\trace(#1)}

\newcommand{\union}{\vee}
\newcommand{\intersection}{\wedge}

\newcommand{\IntervalSty}[1]{{#1}_{\Interval}}

\newcommand{\MarkovArr}{\Rightarrow}
\newcommand{\MarkovArrOf}[1]{\xRightarrow{#1}}

\newcommand{\mapArr}{\to}
\newcommand{\mapArrOf}[1]{\xrightarrow{#1}}

\newcommand{\partialmapArr}{\rightharpoonup}
\newcommand{\partialmapArrOf}[1]{\xrightharpoonup{#1}}
\newcommand{\domOf}[2]{\left. #1 \right|_#2}

\newcommand{\op}{{\mathrm{op}}}
\newcommand{\inv}{{-1}}
\newcommand{\measurePushForward}[2]{#1{}_*#2}

\newcommand{\interval}[2]{[#1, #2]}

\newcommand{\subsetXP}{\subsetX_\posetP}

\newcommand{\Interval}{\mathcal{I}}

\newcommand{\QUSDistri}{\mathcal{P}}
\newcommand{\QUSBorel}{\mathcal{B}}

\newcommand{\IntervalOf}[1]{\Interval(#1)}

\newcommand{\QUSDistriOf}[1]{\QUSDistri(#1)}

\newcommand{\QUSDistriOfDP}[2]{\QUSDistriOf{\qusdpOf{#1}{#2}}}

\newcommand{\QUSMarkovFollowing}{\mathbin{\kernelFollowing_\QUSDistri}}

\newcommand{\feasibleSetF}{\SetF}
\newcommand{\feasibleSetFOf}[1]{\feasibleSetF_{#1}}
\newcommand{\qusObjDP}[3]{\tup{\qusdpOf{#1}{#2}, \qusSetExpInDP{#1}{#2}{#3}}}

\newcommand{\emptySet}{\varnothing}
\newcommand{\singletonEl}{\star}
\newcommand{\singletonSet}{\set{\singletonEl}}

\newcommand{\SetA}{A}
\newcommand{\SetB}{B}
\newcommand{\SetC}{C}
\newcommand{\SetD}{D}

\newcommand{\SetF}{F}
\newcommand{\SetI}{I}
\newcommand{\SetM}{M}

\newcommand{\SetO}{O}

\newcommand{\setelmA}{m_\SetA}
\newcommand{\setelmB}{m_\SetB}
\newcommand{\setelmC}{m_\SetC}

\newcommand{\mapf}{f}
\newcommand{\mapg}{g}
\newcommand{\maph}{h}
\newcommand{\mapr}{r}
\newcommand{\mapq}{q}
\newcommand{\mapVarphi}{\varphi}
\newcommand{\mapVarphitilde}{\tilde{\mapVarphi}}

\newcommand{\mapGraph}{\Gamma}
\newcommand{\mapGraphOf}[1]{\mapGraph_{#1}}

\newcommand{\SetExp}[2]{{#1}^{#2}}

\newcommand{\qusSetExpInDP}[3]{\SetExp{\qusdpOf{#1}{#2}}{#3}}

\newcommand{\qusDistriProduct}{\mathbin{\otimes_\QUSDistri}}

\newcommand{\ev}{\text{ev}}
\newcommand{\evOf}[1]{\ev_{#1}}
\newcommand{\evOfFunRes}[2]{\evOf{#1,#2}}

\newcommand{\confidenceInterval}{\text{conf}}

\newcommand{\idmap}{\text{id}}

\newcommand{\const}{\text{const}}

\newcommand{\constinj}{\rebuttal{\text{cInj}}}

\newcommand{\proj}{\text{proj}}
\newcommand{\perm}{\text{perm}}

\newcommand{\alignedArrowDef}[7]{
\begin{alignedat}{2}
    #1 \colon #2 & #3 && #4 \\
    #5 & #6 && #7
\end{alignedat}
}
\newcommand{\defmap}[5]{\alignedArrowDef{#1}{#2}{\to}{#3}{#4}{\mapsto}{#5}}
\newcommand{\alignedArrowDefOneLine}[7]{
\begin{alignedat}{2}
    #1 \colon #2 & #3 && #4, \,
    #5 & #6 && #7
\end{alignedat}
}
\newcommand{\defmapOneLine}[5]{\alignedArrowDefOneLine{#1}{#2}{\to}{#3}{#4}{{\mapsto}}{#5}}

\newcommand{\tup}[1]{\langle #1 \rangle}
\newcommand{\set}[1]{ \left\{ #1 \right\} }

\newcommand{\SigAlg}{\Sigma}
\newcommand{\SigAlgOf}[1]{\SigAlg_{#1}}
\newcommand{\SigAlgA}{\SigAlgOf{\SetA}}
\newcommand{\SigAlgB}{\SigAlgOf{\SetB}}

\newcommand{\measureSpaceA}{\tup{\SetA, \SigAlgA}}
\newcommand{\measureSpaceB}{\tup{\SetB, \SigAlgB}}

\newcommand{\completionSub}{\mathcal{G}}

\newcommand{\measureProduct}{\times}

\newcommand{\opensetsO}{\mathcal{O}}

\newcommand{\sampleSpaceOmega}{\Omega}
\newcommand{\outcomeomega}{\omega}

\newcommand{\randomVarAlpha}{\alpha}
\newcommand{\randomVarAlphaa}{\randomVarAlpha_a}
\newcommand{\randomVarAlphab}{\randomVarAlpha_b}
\newcommand{\randomVarBeta}{\beta}
\newcommand{\randomVarKappa}{\kappa}

\newcommand{\subsetX}{X}
\newcommand{\subsetY}{Y}
\newcommand{\subsetZ}{Z}
\newcommand{\subsetXA}{\subsetX_\SetA}
\newcommand{\subsetYB}{\subsetY_\SetB}
\newcommand{\subsetZC}{\subsetZ_\SetC}

\newcommand{\distmu}{\mu}
\newcommand{\distkappa}{\kappa}

\newcommand{\measureKernela}{a}
\newcommand{\measureKernelb}{b}

\newcommand{\measureKernelf}{f}

\newcommand{\measureKernelObs}{{\mathrm{obs}}}
\newcommand{\measureKernelPolicy}{\pi}
\newcommand{\measureKernelApply}[2]{\rebuttal{(#2, #1)}} 

\newcommand{\kernelFollowing}{\mathbin{\circ}}

\newcommand{\True}{\top}
\newcommand{\False}{\bot}

\newcommand{\posetP}{P}
\newcommand{\posetQ}{Q}
\newcommand{\posetR}{R}

\newcommand{\posetleq}{\preceq}
\newcommand{\posetgeq}{\succeq}
\newcommand{\posetleqOf}[1]{\posetleq_{#1}}

\newcommand{\poselxF}{x_{\funPosetF}}
\newcommand{\poselxFindex}[1]{x_{\funPosetF, #1}}

\newcommand{\poselxR}{x_{\resPosetR}}
\newcommand{\poselxRindex}[1]{x_{\resPosetR, #1}}

\newcommand{\poselxP}{x_{\posetP}}
\newcommand{\poselyP}{y_{\posetP}}
\newcommand{\poselxQ}{x_{\posetQ}}
\newcommand{\poselyQ}{y_{\posetQ}}
\newcommand{\poselxRnoRes}{x_{\posetR}}
\newcommand{\poselx}{x}

\newcommand{\subsetAfunF}{A_{\funPosetF}}

\definecolor{baiocchi}{RGB}{193,221,245}

\newcommand{\measureKernelUAV}{a_\text{UAV}}

\newcommand{\actuatorset}{\mathcal{A}}
\newcommand{\batteryset}{\mathcal{B}}

\newcommand{\Task}{T}

\newcommand{\Actuation}{A}
\newcommand{\Battery}{B}
\newcommand{\actuator}{a}
\newcommand{\actuatorOne}{{a_1}}
\newcommand{\actuatorTwo}{{a_2}}
\newcommand{\actuatorThree}{{a_3}}

\newcommand{\lift}{L}

\newcommand{\powerp}{P}

\newcommand{\rebuttal}[1]{{\color{blue}#1}}
\newcommand{\CTremark}{(CT)\xspace}

\renewcommand{\rebuttal}[1]{#1}

\renewcommand{\baselinestretch}{1}

\setboolean{proofinappendix}{true}
\renewcommand\footnotemark{}

\title{Distributional Uncertainty and Adaptive Decision-Making in System Co-design}

\author{Yujun Huang, Gioele Zardini}

\date{Laboratory for Information and Decision Systems\\ Massachusetts Institute of Technology}

\begin{document}

\title{Distributional Uncertainty and Adaptive Decision-Making in System Co-design}

\author{Yujun Huang, Gioele Zardini
\thanks{The authors are with the Laboratory for Information \& Decision Systems, 
    Massachusetts Institute of Technology, Cambridge, MA 02139 {\tt \{yujun233,gzardini\}@mit.edu}}
    \thanks{We thank M. Furter for introducing \glsxtrlongpl{abk:qms} and M. Alharbi for discussions.}
}

\maketitle

\begin{abstract}
Complex engineered systems require coordinated design choices across heterogeneous components under conflicting objectives and uncertain specifications.
\rebuttal{Monotone co-design provides a compositional framework for such problems.
Performance of each subsystem is modeled with a design problem: a relation specifying what resources suffice to provide each functionality.}
Existing uncertain co-design models rely on interval bounds, which support worst-case reasoning but cannot represent probabilistic risk or multi-stage adaptive decisions.
We develop a distributional extension of co-design that models uncertain design outcomes as distributions over design problems and supports adaptive decision processes through Markov-kernel re-parameterizations.
Using quasi-measurable and quasi-universal spaces, we show that the standard co-design compositions remain compositional under this richer uncertainty, and introduce queries and observations extracting probabilistic trade-offs, including feasibility probabilities, confidence bounds, and distributions of minimal required resources.
A task-driven \glsentrylong{abk:uav} case study shows how the framework captures risk-sensitive and information-dependent design choices that interval models cannot express.
\end{abstract}

\section{Introduction}\label{sec:introduction}
\rebuttal{Embodied intelligence and \glscps tightly couple heterogeneous hardware (sensors, actuators) with software (perception, planning) and must be evaluated against multiple, often incomparable, metrics, so designers face nontrivial \emph{trade-offs} across components, subsystems, and system-level objectives while coordinating stakeholders with diverse expertise.
For safety- and mission-critical applications, uncertainty cannot be treated purely adversarially: designers must quantify \emph{risk} (e.g., probability of meeting specifications) and reason about how robustness trades against performance.
Design is moreover rarely one-shot, as decisions are staged while information arrives, so later choices should be \emph{adaptive} to earlier commitments and intermediate observations.}

Monotone co-design provides a compositional framework for such design~\cite{censi2019,zardiniCoDesignComplexSystems2023,censi2022}.
\rebuttal{Performance of each component is modeled as a \gls{abk:dp}: a relation specifying which resources suffice to provide a functionality, both forming \glsxtrlongpl{abk:poset}.}
System performance is built by interconnecting \glspl{abk:dp} in series, parallel, and feedback, with traced symmetric monoidal structure~\cite{zardiniCoDesignComplexSystems2023}, enabling joint optimization of hardware\slash software architectures in robotics and control~\cite{zardiniecc21,zardiniTaskdrivenModularCodesign2022,milojevic2025codei}, transportation~\cite{zardini2022co}, and automotive~\cite{neumann2024co}.

Existing co-design tools, however, treat uncertainty via interval bounds on feasible sets~\cite{censi2017uncertainty, zardiniCoDesignComplexSystems2023}.
\rebuttal{Intervals support worst-case robustness but are too coarse for probabilistic risk (chance constraints, confidence levels, quantiles) or multi-stage adaptive policies: they cannot separate ``almost-always safe'' designs from uniformly conservative ones, nor express strategies reacting to intermediate observations.}

\rebuttal{This paper extends monotone co-design with \emph{distributional} uncertainty, thereby also enabling \emph{adaptive} decisions.}
We model uncertain \rebuttal{design} outcomes via probability measures (and parametric families) over spaces of \glspl{abk:dp}, \rebuttal{with quasi-measurable structures avoiding} measurability pathologies while preserving compositionality.
\rebuttal{Here \emph{distributional} means uncertainty is modeled with distributions supplied to the framework, rather than robust sets or intervals; the same term is sometimes used for ambiguity sets of distributions \cite{duttaDistributionalRobustVerification25}, which is not this paper's topic though we note it as a future direction.
This enables (i) distributionally uncertain trade-offs (\glsentrylongpl{abk:dp}) as design results, (ii) multi-stage adaptive design policies conditioning later decisions and observations on earlier realized outcomes, and (iii) queries returning probabilistic trade-offs, uncertain maps, and risk metrics (e.g., expected resource use, quantiles) rather than only deterministic feasible sets, all while preserving measurability.}
This is the first framework that simultaneously handles general distributional uncertainty over compositional design models, supports multi-stage adaptivity, and remains closed under the standard co-design compositions.

\subsection{Related work}\label{subsec:related-work}
\rebuttal{This work sits at the intersection of system-level design for \glscps, compositional specification formalisms, and robust and stochastic optimization.
The recurring gap: rich probabilistic and adaptive uncertainty models lack a modular, multiobjective language for heterogeneous subsystems; while compositional frameworks use descriptions too coarse for quantitative risk or staged policies.}

\subsubsection{System-level design in robotics and CPS}
\rebuttal{Co-optimizing hardware and software at the system level is long recognized as central to robotics and \glscps~\cite{alur2015principles,zardiniCoDesignComplexSystems2023,grosjeanFrameworkCommunicationCompute2025}, including multiobjective formulations for robot system design~\cite{merletOptimalDesignRobots2005} and cross-layer CPS methodologies spanning sensing, computation, and control~\cite{seshia2016design,bradleyOptimizationControlCyberPhysical2015}.}
Recent work also studies cost--performance trade-offs in probabilistic planning~\cite{saberifarChartingTradeoffDesign2022}.
\rebuttal{These efforts target specific problem classes; a framework combining heterogeneous components, multiobjective trade-offs, and distributional uncertainty is missing~\cite{LeeCPSfocusOnModels}.}

\subsubsection{Contract-based and compositional design}
\rebuttal{Assume--guarantee contracts enable modular reasoning by associating components with environment assumptions and behavioral guarantees~\cite{incerAlgebraContracts2022}, with applications spanning circuits and autonomous systems~\cite{benvenisteContractsSystemDesign2018,yuContractComponentSelectionBehavior}.}
They naturally support open-system composition~\cite{abadiComposingSpecifications1993}, but require substantial adaptation when subsystems interact through explicit interfaces~\cite{incerAlgebraContracts2022,PactiAssumeGuaranteeContracts2025}.
\rebuttal{Contracts moreover emphasize satisfying specifications rather than trade-offs, and optimization is typically limited to the weakest component contract for a given environment~\cite{incerAlgebraContracts2022}. \cite{xiaoEfficientExplorationCyberPhysical2024,xiaoContractBasedArchitectureExploration2026} combine mixed-integer programming with contracts to minimize cost subject to specifications but remain single-objective, and probabilistic extensions target satisfaction probabilities rather than compositional optimization over objectives and risk~\cite{deglurkarSystemLevelAnalysisModule2024}.}

\subsubsection{Robust, stochastic, and adaptive optimization}
\rebuttal{
Robust and stochastic optimization offer mature tools for uncertainty, including ambiguity sets, distributional formulations, and chance constraints~\cite{martiStochasticOptimizationMethods2024,bertsimasRobustAdaptiveOptimization2022,shapiroLecturesStochasticProgramming2021}, though many target a single scalar metric (e.g., expected cost~\cite{martiStochasticOptimizationMethods2024,dommelFoundationsMultistageStochastic2021}) rather than trade-offs across interconnected subsystems.
Chance constraints are addressed by discretization~\cite{shenApproximateMethodsSolving2024} and gradient methods~\cite{vanackooijGeneralizedGradientsProbabilistic2020a,bertholdAlgorithmicSolutionOptimization2022}, where randomized policies become essential~\cite{shenChanceConstrainedProbability2023}.
Multi-stage stochastic programming models \emph{here-and-now} versus \emph{wait-and-see} decisions via policies over evolving uncertainty~\cite{pflugMultistageStochasticOptimization2014,bertsimasRobustAdaptiveOptimization2022}; general formulations exist~\cite{dommelFoundationsMultistageStochastic2021,pflugMultistageStochasticOptimization2014}, but implementations rely on scenario trees~\cite{kayacikAdaptiveMultistageStochastic2024,kayacikPartiallyAdaptiveMultistage2025}, restrictive under continuous uncertainty and commonly abstract away how observations depend on prior design choices~\cite{martiStochasticOptimizationMethods2024,shapiroLecturesStochasticProgramming2021}.
We import these distributional and adaptive ideas into co-design.
}

\subsubsection{Monotone co-design and uncertainty}
\rebuttal{Monotone co-design models components as monotone relations between functionalities and resources, composed in series\slash parallel\slash feedback while preserving order and enabling optimization~\cite{zardiniCoDesignComplexSystems2023}
It is instantiated across autonomous systems, control, and beyond~\cite{zardiniecc21,zardiniTaskdrivenModularCodesign2022,milojevic2025codei,zardini2022co,neumann2024co}.
Uncertainty there is treated by interval models guaranteeing conservative feasibility~\cite{censi2017uncertainty}, which encode neither probabilistic risk nor adaptive policies.}

\subsection{Statement of Contribution}
\rebuttal{Our contributions are fourfold.
First, we introduce probability measures and parametric uncertainty over \glspl{abk:dp} using quasi-measurable spaces, a distributional uncertainty model closed under co-design compositions.
Second, we formalize multi-stage design processes whose implementations and specifications depend on intermediate observations and earlier choices, generalizing interval uncertainty compositionally.
Third, we define query operators from distributional uncertain design results to probabilistic performance characterizations, enabling risk-aware design.
Finally, we illustrate the framework on \gls{abk:uav} co-design with uncertain task profiles and specifications, revealing trade-offs that intervals cannot capture.}

\subsection{Structure of the paper}
\rebuttal{\Cref{subsec:math-preliminary,sec:background-co-design} review mathematical preliminaries, monotone co-design, and interval uncertainty
They are existing materials, except the formalization of \I{implementations} (design choices), adapted from the literature to better resemble the new constructions.
\Cref{sec:distributional-uncertainty-co-design} develops the distributional and adaptive extension (distributions over \glspl{abk:dp}, implementations\slash specifications, queries\slash observations) using \glsxtrlongpl{abk:qus}; within it we review \glsxtrlongpl{abk:qms}\slash \glsxtrlongpl{abk:qus} from~\cite{heunenConvenientCategoryHigherOrder2017,forreQuasiMeasurableSpaces2021}, while everything from \cref{def:qus-of-dp} onwards is new.
\Cref{sec:numerical-example} presents the \gls{abk:uav} case study.
Every imported definition and result carries a citation on its statement.}

\section{Mathematical preliminaries}\label{subsec:math-preliminary}

\rebuttal{
Some results in co-design and \glsxtrlongpl{abk:qus} are related to \emph{category theory}, but it is \emph{not} prerequisite here: every definition, statement, and proof in the main text uses ordinary set- and measure-theoretic language, and categorical observations appear only in remarks marked \CTremark, skippable without loss of continuity.
Standard references are~\cite{maclaneCategoriesWorkingMathematician1978,riehl2017category, fong2019}, and~\cite{censi2022} for an applied, engineering-oriented introduction; an extended version collects the categorical concepts in full~\cite{huang2026distributionalextended}.
}

\subsubsection{Sets and functions}\label{subsubsec:sets-functions}
\rebuttal{$\SetA \setproduct \SetB$ denotes the \emph{Cartesian product}, with elements the tuples $\tup{\setelmA, \setelmB}$, $\setelmA \in \SetA$, $\setelmB \in \SetB$.
Nested and unwrapped products are identified: $\SetA \setproduct (\SetB \setproduct \SetC)$ and $\SetA \setproduct \SetB \setproduct \SetC$, and element-wise $\tup{\setelmA, \tup{\setelmB, \setelmC}}$ and $\tup{\setelmA, \setelmB, \setelmC}$.}

\rebuttal{We write~$\mapf \colon \SetA \to \SetB$ for a function (or \emph{map}) with \emph{domain} $\SetA$ and \emph{co-domain} $\SetB$, and $\setelmA \mapsto \mapf(\setelmA)$ for its action on elements.
We denote the graph of $\mapf$ by $\mapGraphOf{\mapf} \defeq \set{\tup{\setelmA, \mapf(\setelmA)}} \subseteq \SetA \setproduct \SetB$, the image of $\subsetXA \subseteq \SetA$ by $\mapf(\subsetXA) \defeq \setWithArg{\setelmB \in \SetB}{\exists \setelmA \in \subsetXA, \mapf(\setelmA) = \setelmB}$, the pre-image of $\subsetYB \subseteq \SetB$ by $\mapf^\inv\left(\subsetYB\right) \defeq \setWithArg{\setelmA \in \SetA}{\mapf(\setelmA) \in \subsetYB}$, and the identity map by $\idmap_{\SetA} \colon \setelmA \mapsto \setelmA$.
The \emph{composite} of~$\mapf \colon \SetA \to \SetB$ and~$\mapg \colon \SetB \to \SetC$ is $\mapg \mapFollowing \mapf \colon \setelmA \mapsto \mapg(\mapf(\setelmA))$, drawn as~$\SetA \mapArrOf{\mapf} \SetB \mapArrOf{\mapg} \SetC$.
A \emph{partial map} $\mapf \colon \SetA \partialmapArr \SetB$ is a map $\domOf{\SetA}{\mapf} \mapArr \SetB$ defined on a subset $\domOf{\SetA}{\mapf} \subseteq \SetA$.
\emph{Projections} and \emph{permutations} are indicated by suffixes, e.g.~$\proj_{1,3} \colon \tup{\setelmA, \setelmB, \setelmC} \mapsto \tup{\setelmA, \setelmC}$ and~$\perm_{1 \leftrightarrow 2} \colon \tup{\setelmA, \setelmB, \setelmC} \mapsto \tup{\setelmB, \setelmA, \setelmC}$.}

\rebuttal{We distinguish the \emph{product} $\mapf \setproduct \mapg \colon \tup{\setelmA, \setelmA'} \mapsto \tup{\mapf(\setelmA), \mapg(\setelmA')}$ from the \emph{tuple} $\tup{\mapf, \mapg}\colon {\setelmA}\mapsto {\tup{\mapf(\setelmA), \mapg(\setelmA)}}$ of maps; both can be indexed, $\setproduct_{i \in I} \SetA_i$ and $\tup{\mapf_i}_{i \in I}$.
$\maph(-,\setelmB) \colon \setelmA \mapsto \maph(\setelmA, \setelmB)$ is the \emph{curry} of $\maph \colon \SetA \setproduct \SetB \to \SetC$.}

\subsubsection{Background on orders}\label{sec:app_order}

\rebuttal{
Orders on resources, functionalities, designs, etc., play a key role in co-design theory.

\begin{definition}[Poset and monotone maps]
A \emph{\glsdefhere{\gls}{poset}} is a tuple $\mathcal{P} =\tup{ P,\preceq_\mathcal{P}}$ with $P$ a set and~$\preceq_\mathcal{P}$ a partial order (reflexive, transitive, antisymmetric relation); when clear we write~$\posetP$ and~$\posetleq$ as the \glsposet and order.
A map $f\colon \posetP \to \posetQ$ is \emph{monotone} if $x\preceq_\mathcal{P} y\Rightarrow f(x) \preceq_\mathcal{Q} f(y)$.
\end{definition}

The \emph{opposite} of a \gls{abk:poset}~$\mathcal{P} = \tup{ P,\preceq_\mathcal{P}}$ is the poset $\mathcal{P}^\op \defeq\tup{ P,\posetleq_{\mathcal{P}}^\op }$ with the same elements and reversed ordering:
$
\poselxP \posetleq_{\mathcal{P}}^\op \poselyP \Leftrightarrow 
\poselyP \posetleq_{\mathcal{P}}\poselxP
$.
The \emph{product} of $\tup{P,\preceq_{\mathcal{P}}}$ and $\tup{Q,\preceq_{\mathcal{Q}}}$ is $\tup{P\times Q,\preceq_{\mathcal{P}\times \mathcal{Q}}}$, with
\begin{equation*}
    \tup{\poselxP,\poselxQ}\preceq_{\mathcal{P}\times \mathcal{Q}}\tup{\poselyP,\poselyQ} \Leftrightarrow (\poselxP \preceq_{\mathcal{P}} \poselyP) \wedge (\poselxQ \preceq_\mathcal{Q} \poselyQ).
\end{equation*}
Monotonicity is preserved by composition and products \cite{censi2022}, and intervals are $\interval{\poselx_{\posetP, \lowerBoundfix}}{\poselx_{\posetP, \upperBoundfix}} \defeq \setWithArg{\poselxP}{\poselx_{\posetP, \lowerBoundfix} \posetleq \poselxP \posetleq \poselx_{\posetP, \upperBoundfix}}$ for $\poselx_{\posetP, \lowerBoundfix} \posetleq \poselx_{\posetP, \upperBoundfix}$.

\begin{definition}[Upper closure and upper sets]
    \rebuttal{The \emph{upper closure} of $\subsetXP \subseteq \posetP$ is $\upperClosure{\subsetXP} \defeq \setWithArg{\poselxP \in \posetP}{\exists \poselyP \in \subsetXP : \poselyP \posetleq_\posetP \poselxP}$, and $\subsetXP$ is an \emph{upper set} if $\upperClosure{\subsetXP} = \subsetXP$.
    $\USetOf{\posetP}$ denotes its upper sets; \emph{lower closures} and \emph{lower sets} are dual.}
\end{definition}

}

\subsubsection{Measure theory and probability}\label{subsubsec:sigma-alg-intro}

We recall standard measure theory and probability\rebuttal{ \cite{bogachevMeasureTheory2007}}.
Additional structures (sigma algebras, topologies) are omitted when clear from context.
\rebuttal{

\begin{definition}[Measurable spaces and maps]
A \emph{sigma algebra}~$\SigAlgA$ on a set~$\SetA$ is a collection of subsets containing $A$ and closed under complements and countable unions; $\measureSpaceA$ is a \emph{measurable space} and sets in $\SigAlgA$ are \emph{measurable} (\emph{events}).
A map between measurable spaces is \emph{measurable} if pre-images of measurable sets are measurable.
\end{definition}

For any family of subsets $\mathcal{G}\subseteq 2^\SetA$, the sigma algebra \emph{generated} by $\mathcal{G}$, denoted~$\sigAlgOf{\mathcal{G}}$, is the smallest sigma algebra containing $\mathcal{G}$.
Completion maximally enlarges the events on which distributions evaluate, by adding subsets of events with zero probability.
}

\begin{definition}[Completion of a sigma algebra]
    Given a distribution $\distmu$ on measurable space $\tup{\SetA, \SigAlg}$, the Lebesgue completion $\SigAlg_\distmu$ is the smallest sigma algebra containing $\SigAlg$ with $\subsetY \in \SigAlg_\distmu$ whenever $\subsetY \subseteq \subsetX \in \SigAlg$ and $\distmu(\subsetX) = 0$.
    The universal completion is the intersection over all distributions \rebuttal{on} it:
    \begin{equation*}
        (\SigAlg)_\completionSub \defeq \cap \setWithArg{\SigAlg_\distmu}{\distmu \text{ is a Distribution on } \rebuttal{\tup{\SetA, \SigAlg}}}.
    \end{equation*}
\end{definition}

\rebuttal{
\emph{Markov kernels} come from categorical probability and support compositional, graphical languages~\cite{fritzSyntheticApproachMarkov2020}.

\begin{definition}[Markov kernels and their composition]\label{def:common-markov-kernel}
A \emph{Markov kernel} between measurable spaces, $\measureKernela \colon \SetA \MarkovArr \SetB$, is a map $\measureKernela \colon \SetA \setproduct \SigAlgOf{\SetB} \to \interval{0}{1}$ s.t.:
    \begin{enumerate}[nosep]
        \item $\measureKernela\measureKernelApply{-}{\setelmA}$ is a probability distribution on $\measureSpaceB$ for each $\setelmA \in \SetA$;
        \item $\measureKernela\measureKernelApply{\subsetYB}{-} \colon \SetA \to \interval{0}{1}$ is measurable for each $\subsetYB \in \SigAlgOf{\SetB}$.
    \end{enumerate}
    We write $\measureKernela(\setelmA)$ for the distribution$\measureKernela\measureKernelApply{-}{\setelmA}$.
    Kernels $\measureKernela \colon \SetA \MarkovArr \SetB$ and $\measureKernelb \colon \SetB \MarkovArr \SetC$ \emph{compose} into $\measureKernelb \kernelFollowing \measureKernela \colon \tup{\setelmA, \subsetZC} \mapsto \int_{\setelmB \in \SetB} \measureKernelb\measureKernelApply{\subsetZC}{\setelmB} \measureKernela\measureKernelApply{d \setelmB}{\setelmA}$, with diagrammatic notation $\SetA \MarkovArrOf{\measureKernela} \SetB \MarkovArrOf{\measureKernelb} \SetC$ in top-down and left-to-right form (\cref{fig:markov-kernel-diagrams}).
    A distribution~$\distmu$ on $\SetB$ is a kernel $\singletonSet \MarkovArrOf{\const_{\distmu}} \SetB$, with $\measureKernelb \kernelFollowing \distmu$ the pushforward.
\end{definition}
}

\begin{figure}[tb]
    \centering
    \newcommand{\subfigscale}{0.25}
    
    \begin{subfigure}{0.48\columnwidth}
        \centering
        \includegraphics[scale=\subfigscale,valign=c]{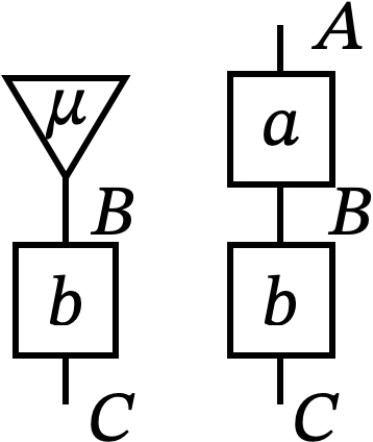}
        \subcaption{Top-down composition.}
    \end{subfigure}
    \begin{subfigure}{0.48\columnwidth}
        \centering
        \includegraphics[scale=\subfigscale,valign=c]{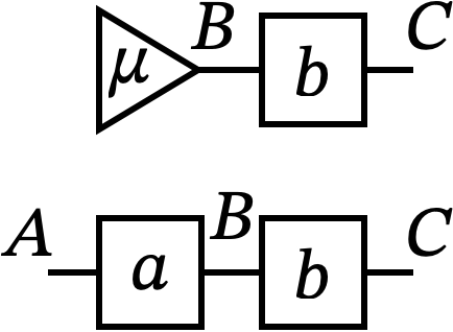}
        \subcaption{Left-to-right composition.}
        \begin{minipage}{.1cm}
        \vfill
        \end{minipage}
    \end{subfigure}

    \caption{Markov kernel diagrams $\measureKernelb \kernelFollowing \distmu$ and $\measureKernelb \kernelFollowing \measureKernela$ in two orientations.}
    \label{fig:markov-kernel-diagrams}
\end{figure}

Graphs, images, and pre-images of measurable maps require topological and descriptive-set structure.

\rebuttal{
\begin{definition}[Topological spaces, continuous maps, and Borel sigma algebras]
    A topological space is a tuple $\tup{\SetA, \opensetsO}$ with $\opensetsO$ the open sets of $\SetA$, containing the empty set and closed under arbitrary unions and finite intersections; \emph{neighborhoods} of $\setelmA \in \SetA$ are open sets containing it.
    A map between topological spaces is \emph{continuous} if pre-images of open sets are open.
    The Borel sigma algebra $\SigAlg_{\opensetsO}$, with its elements called Borel sets, is the smallest sigma algebra containing $\opensetsO$.
\end{definition}

$\setOfBool$ is the two-element set of logic true ($\True$) and false ($\False$) with the discrete topology.
\emph{Homeomorphisms} are continuous bijections whose domain and co-domain are homeomorphic; homeomorphic spaces are topologically equivalent.
Sets in $(\SigAlg_{\opensetsO})_\completionSub$ are called \emph{universally measurable}.
}
\rebuttal{
\begin{definition}[Hausdorff, Polish, Souslin spaces and analytic sets]
    A Hausdorff space is a topological space in which distinct points have disjoint neighborhoods, and a Polish space one \rebuttal{homeomorphic} to a closed subset of $\Reals^\Nats$ (e.g.\ $\Reals$, $\Reals^n$).
    An analytic set is a continuous image of a Polish space into a Hausdorff space; a Souslin space is an analytic Hausdorff space.
\end{definition}
    
}

\rebuttal{
    Every Borel set is analytic and every analytic set is universally measurable; countable intersections and unions of analytic sets are analytic, though complements need not be, and every continuous map is Borel measurable \cite{bogachevMeasureTheory2007}.
}
\rebuttal{The next two theorems restate\rebuttal{ \cite[Lemma 6.7.1, Theorem 6.7.3]{bogachevMeasureTheory2007}}.}

\begin{theorem}\label{thm:borel-maps-graphs-Souslin}
    Let $\SetA, \SetB$ be Souslin.
    \rebuttal{The graph $\mapGraphOf{\mapf}$ of a Borel measurable $\mapf \colon \SetA \to \SetB$ is Borel, hence analytic; conversely, if $\mapGraphOf{\mapf}$ is analytic then $\mapf$ is Borel measurable.}
\end{theorem}

\begin{theorem}\label{thm:measurable-image-pre-image-Souslin}
    For Borel measurable $\mapf \colon \SetA \to \SetB$ between Souslin spaces, images and pre-images of analytic sets are analytic.
\end{theorem}

\section{Background of monotone co-design}\label{sec:background-co-design}

\subsection{Monotone co-design theory}\label{subsec:co-design-intro}
\rebuttal{
Co-design compositionally models and analyzes complex engineering systems with heterogeneous components and tightly coupled design choices.
Its basic objects are \emph{\F{functionalities}} and \emph{\R{resources}}, each a \glsposet.
Notably, \F{functionalities} and \R{resources} are not properties of a \glsposet but \emph{roles} it plays at a component's interface: a component provides its functionalities while requiring its resources, and a \F{functionality} provided by one component can be the \R{resource} required by another (or by itself), as captured by the series and trace operations in \cref{def:dp-compositions}.
Order reflects quality: larger means better, i.e.\ harder to provide as a \F{functionality}, more desirable as a \R{resource}.
These need not be quantities that the designer controls (money, computation, etc.); they capture anything about performance, such as the weather or the task profile described in \cref{sec:numerical-example}.
}
As a running example, consider \gls{abk:uav} design, where perception is a key sub-system.
\rebuttal{Fixing all other subsystems, perception provides \FI{detection accuracy} at the cost of \RI{computation power} under given \RI{weather conditions}.
\RI{Detection accuracy} can be a \R{resource} required by the planner, and \FI{computation power} a \F{functionality} provided by the computer, forming two interconnections.
Accuracy and computation power are positive reals, while weather conditions form a poset with non-comparable elements:
a clear night and a foggy day pose qualitatively different challenges, and designs optimized for one need not perform well in the other.
We then define \glsentrylongpl{abk:dp} as descriptions of component performance relative to a given \F{functionality}\slash \R{resource} interface.
}

\begin{definition}[\glsdefhere{\Gls}{dp}]\label{def:dp}
    Given \glspl{abk:poset}~$\funPosetF$ and $\resPosetR$ of \F{functionalities} and \R{resources}, a \emph{\gls{abk:dp}} is an upper set of $\funPosetF^\op \posetproduct \resPosetR$.
    We denote the set of such \glspl{abk:dp} by $\dpOf{\funPosetF}{\resPosetR}$.
    Given a \gls{abk:dp} \rebuttal{denoted by} $\dprb$, a pair $\tup{\poselxF, \poselxR}$ of functionality $\poselxF$ and resource $\poselxR$ is \emph{feasible} if $\tup{\poselxF, \poselxR} \in \dprb$.
    \rebuttal{
    We order $\dpOf{\funPosetF}{\resPosetR}$ by inclusion: $\dprba \posetleq \dprbb \Leftrightarrow \dprba \subseteq \dprbb$.
    }
\end{definition}

\rebuttal{
The upper set condition captures a design intuition: if $\poselxR$ suffices to provide $\poselxF$ (i.e., $\tup{\poselxF,\poselxR} \in \dprb$), then a better resource~$\poselxR' \posetgeq \poselxR$ suffices for a worse functionality~$\poselxF' \posetleq \poselxF$.
Taking the opposite order on $\F{\funPosetF}$ but not on $\R{\resPosetR}$ captures the difference between the two interfaces: for \emph{this component} it is \emph{easier} to provide \emph{worse} \F{functionalities} with \emph{better} \R{resources}.
A \gls{abk:dp} therefore records \emph{feasibility}, not \emph{efficiency}: a design delivering \FI{\SI{2}{\kilogram}} of payload also appears as delivering \FI{\SI{1}{\kilogram}} at the same cost.
Optimality is recovered by infimals of query results (\cref{def:mdpi-queries}), not by inspection.
}

Systems performance is built from those of sub-systems as multi-graphs of \glspl{abk:dp}, called \emph{co-design problems}; the composition operations are in \cref{def:dp-compositions}, some shown in \cref{fig:dp_def}.

\begin{figure}[tb]

\begin{subfigure}{\columnwidth}
    \centering
    \begin{tikzpicture}[DP]
            \node[dp={2}{2}] (cnt) {DP};
            \draw[runconn, runame={resources}, relres=above,posres=0.9] (cnt_res1){};
            \draw[runconn, runame={}, relres=above,posres=0.9] (cnt_res2){};
            \draw[funconn, funame={functionalities},relfun=above,posfun=1.15] (cnt_fun1){};
            \draw[funconn, funame={},relfun=above,posfun=1.15] (cnt_fun2){};
\end{tikzpicture}
    \subcaption{A \gls{abk:dp} is a monotone relation between posets of \F{functionalities} and \R{resources}. \label{fig:mathcodesign}}

\end{subfigure}

\centering
\begin{subfigure}[b]{0.29\columnwidth}
  \centering
  \scalebox{0.8}{\begin{tikzpicture}[DP]
    \node[dp={1}{1}] (f) {$d$};
    \node[dp={1}{1}, right=1cm of f] (g) {$e$};
    \draw[rconn, rcname={}, fcname={}] (f_res1)  to (g_fun1);
    \draw[runconn, runame={}] (g_res1);
    \draw[funconn, funame={}] (f_fun1);
\end{tikzpicture}}
  \subcaption{Series.}
\end{subfigure}
\begin{subfigure}[b]{0.29\columnwidth}
  \centering
  \scalebox{0.8}{\begin{tikzpicture}[DP]
    \node[dp={1}{1}] (f) {$d$};
    \node[dp={1}{1}, below=0.3cm of f] (g) {$e$};
    \draw[runconn, runame={}] (f_res1){};
    \draw[runconn, runame={}] (g_res1){};
    \draw[funconn, funame={}] (f_fun1){};
    \draw[funconn, funame={}] (g_fun1){};
\end{tikzpicture}}
  \subcaption{Parallel.}
\end{subfigure}
\begin{subfigure}[b]{0.29\columnwidth}
  \centering
  \scalebox{0.8}{\begin{tikzpicture}[DP]
    \node[dp={2}{2}] (f) {$d$};
    \draw[runconn, runame={}] (f_res2){};
    \draw[funconn, funame={}] (f_fun2){};
    \draw[rconn,rcname={},fcname={},feedback=1,loos=3] (f_res1) -- ($(f)+(0,4)$) |- (f_fun1);
\end{tikzpicture}}
  \subcaption{Loop.}
\end{subfigure}

\label{fig:diagrams}
\caption{\Glspl{abk:dp} can be composed in different ways.}
\label{fig:dp_def}

\vspace{-5mm}
\end{figure}

\rebuttal{
\begin{definition}[Composition operations for \glspl{abk:dp}]\label{def:dp-compositions}
    The following operations construct new \glspl{abk:dp} from existing ones.
    
    \noindent \emph{Series}: for \glspl{abk:dp} $\dprba \in \dpOf{\posetP}{\posetQ}$, $\dprbb \in \dpOf{\posetQ}{\posetR}$, their series connection \rebuttal{$\dprba \mthen \dprbb$ is the \glsdp}
        \[            
        \{\tup{\poselxP, \poselxRnoRes} \mid
            \exists \poselxQ : \tup{\poselxP, \poselxQ} \in \dprba
            \text{ and }\tup{\poselxQ, \poselxRnoRes} \in \dprbb\}.
        \]
    Here $\dprba$ uses the functionalities of $\dprbb$ as its resources. \\
    \emph{Parallel}: for $\dprba \in \dpOf{\posetP}{\posetQ}$, $\dprba' \in \dpOf{\posetP'}{\posetQ'}$, their parallel connection \rebuttal{$\dprba \stackMathbin \dprba'$ is the \glsdp}
        \[
            \{\tup{\tup{\poselxP,\poselxP'}, \tup{\poselxQ,\poselxQ'}} \mid 
            \tup{\poselxP, \poselxQ} \in \dprba
            , \tup{\poselxP', \poselxQ'} \in \dprba' \}.
        \]
    This represents two non-interacting systems. \\
    \rebuttal{\emph{Feedback\slash Trace\slash Loop}}: for $\dprb \in \dpOf{\posetP \posetproduct \posetR}{\posetQ \posetproduct \posetR}$, its trace \rebuttal{$\traceOf{\dprb}$ is the \glsdp}
        \[
            \{ \tup{\poselxP, \poselxQ} \mid
            \exists \poselxRnoRes : \tup{\tup{\poselxP,\poselxRnoRes}, \tup{\poselxQ,\poselxRnoRes}} \in \dprb \}.
        \]
    Here $\dprb$ uses its own functionalities as resources. \\
    \emph{Union}\slash \emph{intersection}: for $\dprba,\dprbb \in \dpOf{\posetP}{\posetQ}$, the union $\dprba \union \dprbb$ represents a free choice between them:
        \[
            \{\tup{\poselxP,\poselxQ} \mid \\
            \tup{\poselxP, \poselxQ} \in \dprba
            \text{ or }\tup{\poselxP, \poselxQ} \in \dprbb \};
        \]
    while intersection~$\dprba \intersection \dprbb$ has feasibilities promised by both:
        \[
            \{\tup{\poselxP,\poselxQ} \mid
            \tup{\poselxP, \poselxQ} \in \dprba
            \text{ and }\tup{\poselxP, \poselxQ} \in \dprbb \}.
        \]
    Union and intersection extend to sets of \glspl{abk:dp}, e.g.\ $\union\setWithIndex{\dprb_i}{i \in I}$.
\end{definition}
}

\begin{remark}
\label{rmk:ct-dp-tsmc}
\rebuttal{\CTremark}
With \glsplposet as \emph{objects} \rebuttal{(interfaces of composition)} and \glspldp as \emph{morphisms} \rebuttal{(composable processes)}, the operations in \cref{def:dp-compositions} make the resulting category \emph{traced symmetric monoidal} \rebuttal{(it has a ``place side by side'' product modeling parallel composition, in which the order of parallel components is immaterial and loops can be closed by a feedback operator)}.
Moreover, for fixed $\funPosetF, \resPosetR$ the hom-poset $\dpOf{\funPosetF}{\resPosetR}$ \rebuttal{(all morphisms between a fixed pair of objects, here carrying an order)} is a \emph{complete lattice} \rebuttal{(every subset has a supremum and an infimum)}, and the category is enriched in $\Pos$ \rebuttal{(its hom-sets themselves carry an order)}, the category of \glsplposet and monotone functions~\cite{zardiniCoDesignComplexSystems2023}.

\end{remark}

\rebuttal{Designers care not only which functionality\slash resource pairs are achievable but which choices realize them, captured by \emph{\I{implementations}}; we color \I{implementations} and later defined \Sp{specifications} similar to \F{functionalities} and \R{resources}.
The definition below differs from \cite{censi2019,zardiniCoDesignComplexSystems2023}: it better resembles the new parameterized framework and is a more effective deterministic baseline.

\begin{definition}[\rebuttal{\glsdefhere{\Gls}{mdpi}\footnote{Note that \glspldp are monotone by definition. The term ``monotone'' here is inherited from monotone co-design \cite{censi2019}.}}]\label{def:mdpi}
    For \glspl{abk:poset}~$\funPosetF, \resPosetR$, an \gls{abk:mdpi} is a tuple~$\tup{\impSetI, \realize}$ with $\impSetI$ a set of \emph{implementations} (design choices) and $\realize \colon \impSetI \to \dpOf{\funPosetF}{\resPosetR}$.
    \rebuttal{We denote $\realize(\impi)$ by~$\dprb_\impi$, and call
    \begin{equation*}
        \dprb_{\tup{\impSetI, \realize}} \defeq \union\setWithIndex{\dprb_\impi}{\impi \in \impSetI} \in \dpOf{\funPosetF}{\resPosetR}
    \end{equation*}
    the \gls{abk:dp} \emph{induced} by $\tup{\impSetI, \realize}$.}
\end{definition}

Here $\realize(\impi)$ is the trade-off provided by $\impi$ and $\dprb_{\tup{\impSetI, \realize}}$ the \emph{free choice} among implementations: $\tup{\poselxF, \poselxR} \in \dprb_{\tup{\impSetI, \realize}}$ iff $\tup{\poselxF, \poselxR} \in \dprb_\impi$ for some $\impi$.
For \glsuav perception, hardware and software choices such as \II{LiDAR models} and \II{object detection algorithms} each yield a different trade-off between \FI{detection accuracy}, \RI{computation power}, and \RI{weather conditions}; every combination is an implementation $\impi$ with trade-off $\dprb_\impi$.
\glsplmdpi compose as \glspldp do, a composite's choice being a tuple of its components'; richer decision processes are a key contribution of this paper (\cref{subsec:imp-spec-adaptive-mdpi}).
}

\rebuttal{
\begin{definition}
\label{lem:lifted-compositions-imps}
    The composition operations in \cref{def:dp-compositions}, lifted to \glsplmdpi, are defined as follows.
    Each binary operation~$\diamond$ (one of $\mthen$, $\stack$, $\union$, or $\intersection$) is lifted to:
    \begin{multline*}
        \tup{\impSetI_1, \realize_1} \diamond \tup{\impSetI_2, \realize_2} \\
        \defeq
        \tup{\impSetI_1 \setproduct \impSetI_2, \impi_1 \setproduct \impi_2 \mapsto \realize_1(\impi_1) \diamond \realize_2(\impi_2)}.
    \end{multline*}
    The trace operation~$\trace$ is lifted to:
    \begin{align*}
        \traceOf{\tup{\impSetI, \realize}} \defeq \tup{\impSetI, \impi \mapsto \traceOf{\realize(\impi)}}.
    \end{align*}
    We reuse the notation of \cref{def:dp-compositions} for the lifted operations.
\end{definition}
These are well-defined: each output pairs a set with a map into $\dpOf{\funPosetF}{\resPosetR}$, since $\realize_1(\impi_1) \diamond \realize_2(\impi_2)$ and $\traceOf{\realize(\impi)}$ are \glspldp, so the output is again an \gls{abk:mdpi}.
No order need be preserved, as implementation sets carry none; hence this case is simpler than the interval case below.
}

\begin{remark}
\label{rmk:ct-mdpi-bicat}
\rebuttal{\CTremark}
With the lifted operations of \cref{lem:lifted-compositions-imps} one might hope \glsplposet and \glsplmdpi form a strict traced symmetric monoidal category.
With ordinary set products, however, the coherence diagrams commute only up to canonical $2$-isomorphisms \rebuttal{(invertible transformations between morphisms)}, giving a bicategorical structure.
The strictified set product of~\cite{censi2022} yields a (strict) category, but the order of implementation sets in \rebuttal{Cartesian} products still prevents strict functoriality of~$\stack$.
\end{remark}

\rebuttal{This framework supports practical questions about optimal design choices, formalized as \emph{queries}: they return all combinations of resources and design choices fulfilling the requested functionalities.}

\begin{definition}[Queries for \glsplmdpi]\label{def:mdpi-queries}
    \rebuttal{
    For an \glsmdpi~$\tup{\impSetI, \realize}$ with functionality and resource \glsplposet $\funPosetF$, $\resPosetR$, the \emph{Fix \F{functionalities} minimize \R{resources}} query is
    \begin{equation*}
    \begin{multlined}
        \fixfunminres_{\tup{\impSetI, \realize}} \colon \subsetAfunF \mapsto
        \\
        \setWithArg{\tup{\poselxR, \impi}}{\forall \poselxF \in \subsetAfunF, \tup{\poselxF, \poselxR} \in \realize(\impi)}.
    \end{multlined}
    \end{equation*}
    }
\end{definition}

\begin{remark}\label{rmk:queries-fix-any-opt-others}
\rebuttal{Designers care about \emph{optimal} resources and the choices realizing them, suggesting restriction to the \emph{minimal elements} of $\proj_1(\fixfunminres_{\tup{\impSetI, \realize}}(\subsetAfunF))$.
A feasible resource set however need not be the upper closure of its minimal elements (e.g. $(1,\infty)$ has none), so the framework computes all feasible pairs as in \cref{def:mdpi-queries}.
Real requirements are moreover often \F{functionality}\slash \R{resource} \emph{pairs}, e.g.\ \FI{90\%} accuracy on \RI{sunny days} and \FI{80\%} on \RI{foggy days} while minimizing \R{computation power}; these are allowed by taking the \emph{opposite}
of a \R{resource} \glsposet and treating it as a \F{functionality}.}
\end{remark}

\subsection{Interval uncertainty in co-design}\label{subsec:interval-uncertain-dp}
In~\cite{censi2017uncertainty, zardiniCoDesignComplexSystems2023}, uncertainty is modeled via \emph{intervals} of \glspl{abk:dp}, i.e., \emph{optimistic} (best-case) and \emph{pessimistic} (worst-case) designs; recall that $\dpOf{\funPosetF}{\resPosetR}$ is ordered by inclusion, so a ``better'' \rebuttal{(``larger'')} \gls{abk:dp} makes more~$\tup{\poselxF, \poselxR}$ pairs feasible.

\rebuttal{

\begin{definition}[Interval uncertainty of \glspl{abk:dp}]
\label{lem:interval-lifted-ops}
    Given \glspl{abk:poset}~$\funPosetF$ and $\resPosetR$, the set of interval \glspl{abk:dp}~$\IntervalOf{\dpOf{\funPosetF}{\resPosetR}}$ is defined as
    \begin{equation*}
            \setWithArg{\interval{\dprbL}{\dprbU}}{\dprbL, \, \dprbU \in \dpOf{\funPosetF}{\resPosetR}, \, \dprbL \posetleq \dprbU},
    \end{equation*}
    where the lower bound $\dprbL$ represents the pessimistic estimate, and $\dprbU$ represents the optimistic estimate.
    Compositional operations in \cref{def:dp-compositions} are lifted to intervals of \glspl{abk:dp}:
    \begin{align*}
        \IntervalSty{\trace}(\interval{\dprbL}{\dprbU}) &\defeq \interval{\traceOf{\dprbL}}{\traceOf{\dprbU}}, \\
        \interval{\dprbL}{\dprbU} \mathbin{\IntervalSty{\diamond}} \interval{\dprbL'}{\dprbU'} &\defeq \interval{\dprbL \diamond \dprbL'}{\dprbU \diamond \dprbU'},
    \end{align*}
    where~$\diamond$ is either $\mthen$, $\stack$, $\union$ or $\intersection$.
    Moreover, $\dpOf{\funPosetF}{\resPosetR}$ embeds into $\IntervalOf{\dpOf{\funPosetF}{\resPosetR}}$ by $\dprb \mapsto \interval{\dprb}{\dprb}$~\cite{censi2017uncertainty}.
\end{definition}
\rebuttal{For \glsuav sensing, a design combination might guarantee \FI{detection accuracy} between \FI{80\%} and \FI{85\%}: the \emph{lower bound} $\dprbL$ makes accuracies up to \FI{80\%} feasible, the \emph{upper bound} $\dprbU$ up to \FI{85\%}.}

\rebuttal{
\Cref{lem:interval-lifted-ops} carries a well-definedness obligation on which the whole interval framework rests: for $\IntervalSty{\diamond}$ to land in $\IntervalOf{\dpOf{\funPosetF}{\resPosetR}}$, the pair produced must satisfy $\dprbL \diamond \dprbL' \posetleq \dprbU \diamond \dprbU'$.
It does, because every operation in \cref{def:dp-compositions} is \emph{monotone} for the inclusion order, so $\dprbL \posetleq \dprbU$ and $\dprbL' \posetleq \dprbU'$ give the inequality, and likewise for $\trace$~\cite{censi2017uncertainty, zardiniCoDesignComplexSystems2023}.
Monotonicity is thus what makes worst- and best-case bounds composable; we need it again in \cref{lem:composing-inner-bounds}.
}

\Glsplmdpi lift to interval uncertainty too.
\begin{definition}[\Glspl{abk:mdpi} with interval uncertainty]\label{def:mdpi-interval}
    For \glspl{abk:poset}~$\funPosetF, \resPosetR$, an \emph{\glsmdpi with interval uncertainty} is a tuple $\tup{\impSetI, \realize}$ with $\impSetI$ a set of \I{implementations} and $\realize \colon \impSetI \to \IntervalOf{\dpOf{\funPosetF}{\resPosetR}}$.
    Compositions are formed by replacing each operation in \cref{lem:lifted-compositions-imps}, e.g.\ $\trace$, with its interval lift from \cref{lem:interval-lifted-ops}, e.g.\ $\IntervalOf{\trace}$.
\end{definition}
}

\rebuttal{These are well-defined for both reasons combined: a composite's implementation set is again a set, and the interval-valued map lands in $\IntervalOf{\dpOf{\funPosetF}{\resPosetR}}$ by monotonicity.
Queries then yield optimistic and pessimistic answers~\cite{censi2017uncertainty}: run \cref{def:mdpi-queries} on both bounds and read the resource sets as best- and worst-case designs.}

\section{Distributional uncertainty and adaptive decisions in co-design}\label{sec:distributional-uncertainty-co-design}
\rebuttal{Real design processes are more complex than one-shot \I{implementation} choices under best-- and worst--case uncertainty: \I{implementations} carry parameters inducing \Sp{specifications} uncertain in ways intervals cannot capture, and some decisions can wait until part of that uncertainty resolves.

Consider again the perception and actuator sub-systems of a \glsuav.
A sensor is chosen from a catalogue, each variant posing \Sp{specifications} such as field of view, noise characteristics, and failure modes. Manufacturers guarantee only \emph{distributions} over these parameters, and purchasing a sensor draws a sample from them.
Actuator \Sp{specifications} are likewise observable only after obtaining a concrete product, and the designer is often free to select one component \emph{after} observing another, for instance buying the actuator after testing the delivered sensor.
The design process is therefore inherently \emph{adaptive} and involves full probability distributions, not only interval bounds; reducing them to confidence intervals would discard exactly the quantitative trade-offs between \F{functionality}, \R{resources}, and risk that practitioners care about.}

To support such trade-offs compositionally we move from intervals to \emph{distributions over \glspldp}, and from static \glsplmdpi to \glsplmdpi with \emph{distributional uncertainty and adaptive decisions}.
\rebuttal{
As in stochastic programming we take the distributions as given, agnostic to what they represent; in system design they arise from at least four sources:
\begin{enumerate}[nosep,label=(\roman*)]
    \item \emph{Manufacturing variability}: unit-to-unit spread of nominally identical components (motor constants, battery capacity), often a datasheet tolerance.
    \item \emph{Environmental variability}: conditions outside control of the designer such as weather, temperature, or the \FI{task profile} of \cref{sec:numerical-example}, typically estimated from logged operational data.
    \item \emph{Algorithmic randomness}: stochastic performance of sampling- or learning-based components, characterized by benchmarking over a test distribution.
    \item \emph{Epistemic uncertainty}: limited knowledge of a component, e.g.\ a surrogate fitted to few measurements, the distribution being a posteriori knowledge.
\end{enumerate}
The first three are \emph{aleatoric}, the last \emph{epistemic}; our constructions do not distinguish them, since both yield a distribution over \glspldp and the composition results apply verbatim.
What \emph{does} matter is independence: \cref{lem:composing-inner-bounds} assumes the component sample spaces factor, true for independent physical origins (motor and battery tolerances) but not under a common cause (a shared ambient temperature), in which case the shared quantity should be the common parameter fed to both components.
Estimating these distributions, updating by conditioning, and exploration--exploitation are future work.
}

\subsection{Distributions over \glsentrylongpl{abk:dp} as uncertain design results}\label{subsec:distri-dps}

In deterministic co-design the design result for fixed \I{implementations} is a \glsdp encoding \F{functionality}\slash \R{resource} trade-offs (\cref{def:dp}).
To model distributional uncertainty we need \emph{random} \glspldp, i.e.\ \emph{distributions over \glspldp}, that are:
\rebuttal{
\begin{enumerate}[label=(\roman*),nosep]
    \item \emph{Compositional}: the series, parallel, and feedback operations in \cref{def:dp-compositions} remain measurable;
    \item \emph{Semantically expressive}: each evaluation $\evOfFunRes{\poselxF}{\poselxR} \colon \dpOf{\funPosetF}{\resPosetR} \to \setOfBool$, $\dprb \mapsto \big(\tup{\poselxF,\poselxR} \in \dprb \big)$, is measurable, so we can ask the probability that a \F{functionality}\slash \R{resource} pair (or countable family) is feasible;
    \item \emph{Stochasticity of \F{functionalities}\slash \R{resources}}: the \emph{constant injection} ${\constinj} \colon {\funPosetF \posetproduct \resPosetR}\to {\dpOf{\funPosetF}{\resPosetR}}$, ${\tup{\poselxF, \poselxR}} \mapsto {\lowerOf{\poselxF} \posetproduct \upperOf{\poselxR}}$, is measurable, turning distributions over~$\funPosetF \posetproduct \resPosetR$ into distributions over \glspldp;
    \item \emph{Compatible with co-design}: deterministic \glspldp embed as degenerate (delta) distributions.
\end{enumerate}
A naive approach puts a sigma algebra on $\dpOf{\funPosetF}{\resPosetR}$ and takes ordinary probability measures; \Cref{cex:impossibility-sigma-algebra-dps} shows the first three requirements already fail, whereas this section's construction satisfies all four.
}

\begin{counterexample}[\rebuttal{Impossibility of sigma algebras on \glspldp}]\label{cex:impossibility-sigma-algebra-dps}
Equip $\Reals$ with the discrete partial order, $x \posetleq y \iff x = y$.
\rebuttal{Each measurable map~$\mapf \colon \Reals \to \setOfBool$ then defines a \glsdp by:
$
     \dprb_{\mapf}
    \defeq
    \setWithArg{\tup{\realelx, \boolelb}}{\mapf(\realelx) \posetleqOf{\setOfBool} \boolelb}
    \in \dpOf{\F{\Reals}}{\R{\setOfBool}}.
$}
Write~$\SetExp{\setOfBool}{\Reals}$ as the set of measurable maps~$\Reals \to \setOfBool$, and $\varphi \colon \SetExp{\setOfBool}{\Reals} \to \dpOf{\F{\Reals}}{\R{\setOfBool}}, \
    \mapf \mapsto \dprb_{\mapf}$.
Suppose for contradiction a sigma algebra~$\sigAlgOf{\dpOf{\F{\Reals}}{\R{\setOfBool}}}$ makes (i) all operations of \cref{def:dp-compositions}, (ii) the evaluations $\evOfFunRes{\poselxF}{\poselxR}$, and (iii) the constant injection $\constinj$ measurable.
\rebuttal{Equip~$\SetExp{\setOfBool}{\Reals}$ with the pullback sigma algebra $\sigAlgOf{\SetExp{\setOfBool}{\Reals}} \defeq \setWithArg{\SetA \subseteq \SetExp{\setOfBool}{\Reals}}{\varphi(\SetA) \in \sigAlgOf{\dpOf{\F{\Reals}}{\R{\setOfBool}}}}$.}
The composite of three measurable maps
\begin{multline*}
    (\setOfBool \setproduct \Reals) \setproduct \SetExp{\setOfBool}{\Reals}
    \mapArrOf{\constinj \setproduct \varphi}
    \dpOf{\F{\setOfBool}}{\R{\Reals}} \setproduct \dpOf{\F{\Reals}}{\R{\setOfBool}}
    \\
    \mapArrOf{\mthen}
    \dpOf{\F{\setOfBool}}{\R{\setOfBool}}
    \mapArrOf{\evOfFunRes{\False}{\False}}
    \setOfBool
\end{multline*}
is measurable.
\rebuttal{Fixing the first argument to~$\False$, the induced map sends $\tup{\tup{\False, \realelx}, \mapf}$ to $\set{\tup{\False, \realelx}} \setproduct \setWithArg{\tup{\realelx', \boolelb}}{\mapf(\realelx') \posetleqOf{\setOfBool} \boolelb}$, then to $\setWithArg{\tup{\False, \boolelb}}{\mapf(\realelx) \posetleqOf{\setOfBool} \boolelb}$, then to $\neg \mapf(\realelx)$, whose measurability is equivalent to that of} the usual evaluation map~$\SetExp{\setOfBool}{\Reals} \times \Reals \to \setOfBool,\ \tup{\mapf,\realelx} \mapsto \mapf(\realelx).$
No sigma algebra on~$\SetExp{\setOfBool}{\Reals}$ makes this evaluation measurable~\cite{aumannBorelStructuresFunction1961}, forming a contradiction.
\end{counterexample}

\rebuttal{
\begin{remark}[Why a new formalism is needed]
It is worth pausing on why the distributional case demands machinery the interval case does not, since the two are often expected to be siblings.
In dynamical systems a disturbance and a signal are objects of the same kind, both time-indexed trajectories, so robust and stochastic formulations share one language and differ only in how the disturbance is quantified.
Here the object carrying the uncertainty is the \gls{abk:dp} itself: a \emph{relation} between \glsplposet, i.e., a higher-order object.
Intervals cope with this because \glspldp are ordered and the operations of \cref{def:dp-compositions} are monotone, so bounds compose (\cref{lem:interval-lifted-ops}).
Distributions do not: \cref{cex:impossibility-sigma-algebra-dps} shows that \emph{no} sigma algebra on $\dpOf{\funPosetF}{\resPosetR}$ supports even the three minimal requirements (i)--(iii), so the obstruction is a property of ordinary measure theory rather than an artifact of a particular construction.
What is required is a setting in which map spaces are first-class objects and evaluation is itself a morphism.
This is the same obstruction that blocks higher-order probability on function spaces, and \glsxtrlongpl{abk:qms}\slash \glsxtrlongpl{abk:qus} are exactly such a setting~\cite{heunenConvenientCategoryHigherOrder2017,forreQuasiMeasurableSpaces2021}: the framework below is not a generalization for its own sake, but close to the least structure under which ``a distribution over a design problem'' is definable at all while the co-design operations survive.
The dividend is that everything downstream is elementary: adaptivity is Markov kernels (\cref{def:repara-distri-mdpi}), queries are maps out of the distribution space (\cref{def:distri-mdpi-queries}), and the designer writes components and a diagram as in the deterministic case, with the sources of randomness made explicit.
\end{remark}

We review \glsxtrlongpl{abk:qms}\slash \glsxtrlongpl{abk:qus}~\cite{heunenConvenientCategoryHigherOrder2017,forreQuasiMeasurableSpaces2021} before constructing distributions on \glspldp.
}

\subsubsection{\glsentrylongpl{abk:qms} and \glsentrylongpl{abk:qus}}
Probability theory fixes a sample space $\sampleSpaceOmega$ with a sigma algebra $\SigAlgOf{\sampleSpaceOmega}$ of events, then considers \emph{random variables} from it into other \rebuttal{\emph{measurable spaces}}.
\glsxtrlong{abk:qms} invert this: random variables are first-class and sigma algebras derived from them, resolving \rebuttal{the impossibility above}.
\rebuttal{Introduced as a categorical construction for \emph{higher-order probability} over maps \cite{heunenConvenientCategoryHigherOrder2017}, many of their results are proved categorically; category theory \rebuttal{stays non-prerequisite here, as we restate every result we need in set- and measure-theoretic language.}
\rebuttal{Fix $\tup{\sampleSpaceOmega, \SigAlgOf{\sampleSpaceOmega}}$ and write} $\SetExp{\sampleSpaceOmega}{\sampleSpaceOmega}$ for the measurable maps $\sampleSpaceOmega \to \sampleSpaceOmega$.}

\begin{definition}[Quasi-measurable spaces, quasi-measurable maps, and induced sigma algebras]\label{def:qms-and-maps-and-sigalgs}
A \emph{\glsdefhere{\gls}{qms}} is a pair~$\tup{\SetA, \SetExp{\SetA}{\sampleSpaceOmega}}$ with~$\SetA$ a set and~$\SetExp{\SetA}{\sampleSpaceOmega}$ a distinguished set of \emph{random variables}~$\sampleSpaceOmega\to \SetA$ such that:
    \begin{enumerate}[nosep]
        \item constant maps~$\const_{\setelmA} \colon \outcomeomega \mapsto \setelmA$ for all $\setelmA\in \SetA$, lie in $\SetExp{\SetA}{\sampleSpaceOmega}$;
        \item $\randomVarAlpha \mapFollowing \mapVarphi \in \SetExp{\SetA}{\sampleSpaceOmega}$ for every~$\randomVarAlpha \in \SetExp{\SetA}{\sampleSpaceOmega}$ and measurable $\mapVarphi \in \SetExp{\sampleSpaceOmega}{\sampleSpaceOmega}$.
    \end{enumerate}
    \rebuttal{When the random variables are clear we write ``\glsqms $\SetA$''; the sample space with all measurable maps is itself a \glsqms.
    A map $\mapf \colon \SetA \to \SetB$ of \glsplqms is \emph{quasi-measurable} if $\mapf \mapFollowing \randomVarAlpha \in \SetExp{\SetB}{\sampleSpaceOmega}$ for all~$\randomVarAlpha \in \SetExp{\SetA}{\sampleSpaceOmega}$; $\SetExp{\SetB}{\SetA}$ denotes these.}
    \rebuttal{Every \glsqms induces a sigma algebra $\QUSBorel(\SetA) \defeq \setWithArg{\subsetXA \subseteq \SetA}{\forall \randomVarAlpha \in \SetExp{\SetA}{\sampleSpaceOmega}, \randomVarAlpha^\inv(\subsetXA) \in \SigAlgOf{\sampleSpaceOmega}}$.}

\end{definition}

\rebuttal{When $\SetA$ carries a natural sigma algebra, e.g.\ the Borel one of a Polish space, taking all measurable maps as random variables reconnects to ordinary probability: if $\sampleSpaceOmega$ is also Polish, the induced sigma algebra is the completion of the Borel one.
Taking \emph{all} maps $\sampleSpaceOmega \to \SetA$ as random variables would leave very few maps \emph{from} $\SetA$ quasi-measurable, just as the power set as sigma algebra leaves very few maps \emph{into} $\SetA$ measurable.}

\begin{definition}[Product of \glsplqms]\label{def:product-qms}
    \rebuttal{The \emph{product} of \glsplqms $\tup{\SetA, \SetExp{\SetA}{\sampleSpaceOmega}}$, $\tup{\SetB, \SetExp{\SetB}{\sampleSpaceOmega}}$ is $\tup{\SetA \setproduct \SetB, \SetExp{\left( \SetA \setproduct \SetB \right)}{\sampleSpaceOmega}}$, a random variable on~$\SetA \setproduct \SetB$ being a tuple of one per factor: $\SetExp{\left( \SetA \setproduct \SetB \right)}{\sampleSpaceOmega} \defeq \setWithArg{\tup{\randomVarAlphaa, \randomVarAlphab}}{\randomVarAlphaa \in \SetExp{\SetA}{\sampleSpaceOmega}, \randomVarAlphab \in \SetExp{\SetB}{\sampleSpaceOmega}}$.}
\end{definition}

\rebuttal{Function spaces now carry natural sets of random variables.}

\begin{definition}[\glsqms of quasi-measurable maps and evaluation]\label{def:qms-maps}
    \rebuttal{As random variables for $\SetExp{\SetB}{\SetA}$ we take those maps $\randomVarBeta \colon \sampleSpaceOmega \to \SetExp{\SetB}{\SetA}$ for which ${\outcomeomega}\mapsto {\randomVarBeta\big(\mapVarphi(\outcomeomega)\big)\big(\randomVarAlpha(\outcomeomega)\big)}$ is a random variable in $\SetExp{\SetB}{\sampleSpaceOmega}$ for every $\randomVarAlpha \in \SetExp{\SetA}{\sampleSpaceOmega}$ and $\mapVarphi \in \SetExp{\sampleSpaceOmega}{\sampleSpaceOmega}$.}
\end{definition}

\rebuttal{Crucially, evaluation ${\ev} \colon {\SetA \setproduct \SetExp{\SetB}{\SetA}}\to {\SetB}$, ${\tup{\setelmA, \mapf}}\mapsto {\mapf(\setelmA)}$, is then quasi-measurable \cite{forreQuasiMeasurableSpaces2021}.}

\rebuttal{
\begin{remark}
\label{rmk:ct-qms-ccc}
\CTremark
With the terminal \glsqms (a singleton), products, and the \glsqms of quasi-measurable maps, \glsplqms and quasi-measurable maps form a \emph{\rebuttal{Cartesian} closed category}~\cite{forreQuasiMeasurableSpaces2021}; the next lemma is a consequence.
\end{remark}
}

\begin{lemma}\label{lem:qms-curry}
    \rebuttal{The curry (partial evaluation) ${\bar{\mapf}} \colon {\SetA}\to {\SetExp{\SetC}{\SetB}}$, ${\setelmA}\mapsto{\mapf(\setelmA, -)}$, of a quasi-measurable $\mapf \colon \SetA \setproduct \SetB \to \SetC$ is quasi-measurable, as are the pre- and post-composition maps ${(-) \mapFollowing \mapf} \colon {\mapg}\mapsto {\mapg \mapFollowing \mapf}$ and ${\mapf \mapFollowing (-)} \colon {\maph}\mapsto {\mapf \mapFollowing \maph}$, of types ${\SetExp{\SetC}{\SetB}}\to {\SetExp{\SetC}{\SetA}}$ and ${\SetExp{\SetA}{\SetC}}\to {\SetExp{\SetB}{\SetC}}$, for quasi-measurable $\mapf \colon \SetA \to \SetB$.}
\end{lemma}

To connect \glsplqms with standard probability and Markov kernels compositionally, we restrict to ``good'' sample spaces.

\begin{definition}[Quasi-universal spaces]\label{def:qus}
A \glsqms~$\tup{\SetA, \SetExp{\SetA}{\sampleSpaceOmega}}$ is a \emph{\glsdefhere{\gls}{qus}} if the sample space~$\sampleSpaceOmega$ is \rebuttal{Polish, typically~$\sampleSpaceOmega=\mathbb{R}^n$ with the usual topology, and completion of its Borel sigma-algebra.}
\end{definition}

\rebuttal{This is not computationally limiting: randomness that is computable or arises from countably many experiments originates from such sample spaces, and all \glsqms constructions restrict straightforwardly to \glsplqus.}
We adopt push-forward of distributions over \glsqus \cite{forreQuasiMeasurableSpaces2021}.

\begin{definition}[Distributions over \glsplqus]\label{def:distri-qus}
Let~$\tup{\SetA, \SetExp{\SetA}{\sampleSpaceOmega}}$ be a \glsqus.
\rebuttal{A \emph{distribution over}~$\SetA$ is an equivalence class~$\set{\tup{\distmu, \randomVarAlpha}}/\simeq$ with~$\distmu$ a probability distribution on $\sampleSpaceOmega$ and $\randomVarAlpha \in \SetExp{\SetA}{\sampleSpaceOmega}$ a random variable.
Two pairs are equivalent if $\distmu(\randomVarAlpha^\inv(\subsetXA)) = \distmu'(\randomVarAlpha'^\inv(\subsetXA))$ for all $\subsetXA \in \QUSBorel(\SetA)$, i.e.\ if they induce the same measure on the induced sigma algebra.}
We write $\QUSDistriOf{\SetA}$ for the set of classes and $\tup{\distmu, \randomVarAlpha}$ for a class when unambiguous.
\rebuttal{To make $\QUSDistriOf{\SetA}$ itself a \glsqus, we take as random variables precisely the maps $\measurePushForward{\randomVarAlpha}{\distkappa} \colon {\sampleSpaceOmega}\to{\QUSDistriOf{\SetA}}$, ${\outcomeomega}\mapsto{\tup{\randomVarKappa(\outcomeomega), \randomVarAlpha}}$,}
    where~$\randomVarAlpha\in \SetExp{\SetA}{\sampleSpaceOmega}$ and~$\randomVarKappa \in \SetExp{\QUSDistriOf{\sampleSpaceOmega}}{\sampleSpaceOmega}$ is a classical Markov kernel assigning each~$\outcomeomega$ a distribution~$\randomVarKappa(\outcomeomega)$ on~$\sampleSpaceOmega$.
    One can check that $\QUSDistriOf{\sampleSpaceOmega}$ (with the induced sigma algebra $\QUSBorel(\sampleSpaceOmega)$) is the usual space of probability measures on~$\sampleSpaceOmega$ \rebuttal{\cite{forreQuasiMeasurableSpaces2021}}.
\end{definition}

\rebuttal{
As in usual probability, (quasi) measurable deterministic maps induce maps between spaces of distributions.
\begin{lemma}\label{lem:qus-maps-lift-dists}
    Each quasi-measurable $\mapf \colon \SetA \to \SetB$ induces $\QUSDistriOf{\mapf} \colon \QUSDistriOf{\SetA} \to \QUSDistriOf{\SetB},\, \tup{\distmu, \randomVarAlpha} \mapsto \tup{\distmu, \mapf \mapFollowing \randomVarAlpha}$, transforming distributions by sampling an outcome and applying the map; $\QUSDistriOf{\mapf}$ is quasi-measurable \cite{forreQuasiMeasurableSpaces2021}.
\end{lemma}
}

\rebuttal{
\emph{Markov kernels} and their compositions carry over from classical measure theory.

\begin{definition}[Markov kernels for \glsplqus and their compositions]\label{def:markov-kernels-qus}
For \glsplqus $\SetA, \SetB$, a \emph{Markov kernel}~$\SetA \MarkovArr \SetB$ is a quasi-measurable map~$\measureKernela \colon \SetA \to \QUSDistriOf{\SetB}$.
The \emph{Kleisli composition}~$\measureKernela \QUSMarkovFollowing \measureKernelb\colon \SetA \MarkovArr \SetC$ of~$\measureKernela \colon \SetA \MarkovArr \SetB$ and~$\measureKernelb \colon \SetB \MarkovArr \SetC$ is defined as follows.
For~$\setelmA\in \SetA$ write~$\measureKernela(\setelmA)=\tup{\distmu, \randomVarAlpha}$; as~$\measureKernelb$ is quasi-measurable and~$\randomVarAlpha\in \SetExp{\SetB}{\sampleSpaceOmega}$, the composite~$\randomVarAlpha\mapFollowing\measureKernelb$ lies in~$\SetExp{\QUSDistriOf{\SetC}}{\sampleSpaceOmega}$, so there are~$\randomVarKappa \in \SetExp{\QUSDistriOf{\sampleSpaceOmega}}{\sampleSpaceOmega}$, $\randomVarBeta \in \SetExp{\SetC}{\sampleSpaceOmega}$ with~$\randomVarAlpha \mapFollowing \measureKernelb = \measurePushForward{\randomVarBeta}{\randomVarKappa}$.
We set~$\left(\measureKernela \QUSMarkovFollowing \measureKernelb\right)(\setelmA)\defeq\tup{\distmu \kernelFollowing \randomVarKappa, \randomVarBeta}\in \QUSDistriOf{\SetC}.$
\end{definition}
}

\subsubsection{Sample space and assumptions on posets}
\rebuttal{We fix a sample space~$\sampleSpaceOmega$ (a \glsqus) aggregating all sources of foregoing randomness (\cref{subsec:distri-dps}) which is Polish; random variables into~$\SetA$ are maps $\sampleSpaceOmega \to \SetA$ in the \glsqms\slash \glsqus sense (not classical probability theory).
We ask \emph{which random variables into $\dpOf{\funPosetF}{\resPosetR}$ are compositional, computable, and rich enough for applications}, under mild regularity on the \glsplposet involved.}

\begin{assumption}\label{asm:posets-souslin-order-borel}
All \glsplposet considered are \emph{Souslin spaces} and the order relation~$\setWithArg{\tup{\poselxP,\poselxP'}}{\poselxP \posetleqOf{\posetP} \poselxP'}
            \subseteq \posetP \setproduct \posetP$
is a Borel set.
\end{assumption}

\rebuttal{It holds for finite \glsplposet with discrete topology, $\Reals^n$ with the usual topology and product or discrete order, and complete separable metric spaces with closed order relations.}

\begin{lemma}\label{lem:posets-souslin-product}
If \glsplposet $\posetP, \posetQ$ satisfy \cref{asm:posets-souslin-order-borel}, so do $\posetP^\op$ and $\posetP \posetproduct \posetQ$.
\end{lemma}
\begin{proof}
    \rebuttal{We show the sets are Souslin and the order relations Borel.}
    Since $\setWithArg{\tup{\poselxP, \poselxP'}}{\poselxP \posetgeq \poselxP'} = \perm_{1 \leftrightarrow 2}(\setWithArg{\tup{\poselxP', \poselxP}}{\poselxP' \posetleq \poselxP})$ and permutations are \rebuttal{homeomorphisms}, \rebuttal{\cref{thm:measurable-image-pre-image-Souslin} applies and the opposite order set $\setWithArg{\tup{\poselxP, \poselxP'}}{\poselxP \posetgeq \poselxP'}$ is a Borel set.
    For the product \glsposet, product of two Souslin Spaces is Souslin
    , so we only need to prove that the set $\setWithArg{\tup{\poselxP, \poselxQ, \poselxP', \poselxQ'}}{\poselxP \posetleq \poselxP', \poselxQ \posetleq \poselxQ'}$ is a Borel set.}
    Observe that
    \begin{align*}
        & \setWithArg{\tup{\poselxP, \poselxQ, \poselxP', \poselxQ'}}{\poselxP \posetleq \poselxP', \poselxQ \posetleq \poselxQ'}
        \\
        =
        &
        \perm_{2 \leftrightarrow 3}( \setWithArg{\tup{\poselxP, \poselxP', \poselxQ, \poselxQ'}}{\poselxP \posetleq \poselxP', \poselxQ \posetleq \poselxQ'}) \\
        = 
        &
        \begin{multlined}[t][\dimexpr\linewidth-2em\relax]
            \perm_{2 \leftrightarrow 3}( \set{\poselxP \posetleq \poselxP'} 
            \setproduct \set{\poselxQ \posetleq \poselxQ'} ).
        \end{multlined}
    \end{align*}
    \rebuttal{Products of Borel sets are Borel and permutations are homeomorphisms, so the statements follow.}
\end{proof}

\begin{lemma}\label{lem:upper-sets-analytic}
Let~$\posetP$ satisfy \cref{asm:posets-souslin-order-borel}. If~$\SetA\subseteq \posetP$ is analytic, so are its closures~$\upperClosure{\SetA}$ and~$\lowerClosure{\SetA}$.
\end{lemma}
\begin{proof}
Let~$\SetO = \setWithArg{\tup{\poselxP,\poselxP'}}{\poselxP \posetleq \poselxP'} \subseteq \posetP \setproduct \posetP$.
Then~$\upperClosure{\SetA} = \proj_2(\proj_1^\inv(\SetA) \cap \SetO)$ is analytic \rebuttal{by \cref{thm:measurable-image-pre-image-Souslin}, continuity of projections, and closure of analytic sets under finite intersection}; the same argument on the opposite order gives~$\lowerClosure{\SetA}$.
\end{proof}

\subsubsection{Quasi-universal space of design problems}

\rebuttal{We first define \glsplqus of \glspldp by choosing a family of random variables, the analogue of a sigma algebra, so distributions make sense and \cref{lem:qus-maps-lift-dists} applies.}

\begin{definition}[\Glsqus of \glspldp]\label{def:qus-of-dp}
Let \glsplposet~$\funPosetF$ and $\resPosetR$ satisfy \cref{asm:posets-souslin-order-borel}.
The \emph{\glsqus of \glspldp} from~$\funPosetF$ to~$\resPosetR$ is $\qusObjDP{\funPosetF}{\resPosetR}{\sampleSpaceOmega}$, where (i)~$\qusdpOf{\funPosetF}{\resPosetR}$ is the set of \glspldp~$\dprb \in \dpOf{\funPosetF}{\resPosetR}$ expressible as an upper closure~$\dprb = \upperClosure{\SetA}$ of an analytic~$\SetA \subseteq \funPosetF^\op \setproduct \resPosetR$;
(ii) a random variable~$\randomVarAlpha \colon \sampleSpaceOmega \to \qusdpOf{\funPosetF}{\resPosetR}$ lies in~$\qusSetExpInDP{\funPosetF}{\resPosetR}{\sampleSpaceOmega}$ if its \emph{\glsdefhere{\gls}{ofs}}
\begin{equation*}
\feasibleSetFOf{\randomVarAlpha} \defeq
\setWithArg{\tup{\outcomeomega,\poselxF,\poselxR}}{
                \tup{\poselxF,\poselxR} \in \randomVarAlpha(\outcomeomega)}
            \subseteq \sampleSpaceOmega \setproduct \funPosetF \setproduct \resPosetR
\end{equation*}
is analytic.
\end{definition}

\begin{lemma}
    $\qusObjDP{\funPosetF}{\resPosetR}{\sampleSpaceOmega}$ is a \glsqus.
\end{lemma}
\begin{proof}
As~$\sampleSpaceOmega$ is Polish, it suffices to check the two \glsqms axioms: (i) constant maps are random variables; (ii) random variables are closed under precomposition with measurable maps on~$\sampleSpaceOmega$.
For (i), let~$\dprb = \upperClosure{\SetA}$ with~$\SetA$ analytic in~$\funPosetF^\op \setproduct \resPosetR$.
By \cref{lem:upper-sets-analytic},~$\dprb$ is analytic, hence~$\feasibleSetFOf{\randomVarAlpha}
        =
        \setWithArg{\tup{\outcomeomega,\poselxF,\poselxR}}{\tup{\poselxF,\poselxR} \in \dprb}
        =
        \sampleSpaceOmega \setproduct \dprb$
is analytic.
Thus the constant map~$\randomVarAlpha(\outcomeomega) = \dprb$ lies in $\qusSetExpInDP{\funPosetF}{\resPosetR}{\sampleSpaceOmega}$.
For (ii), let~$\randomVarAlpha \in \qusSetExpInDP{\funPosetF}{\resPosetR}{\sampleSpaceOmega}$ and~$\mapVarphi \in \SetExp{\sampleSpaceOmega}{\sampleSpaceOmega}$ be measurable.
Define the measurable map~$\mapVarphitilde \colon
        \sampleSpaceOmega \setproduct \funPosetF \setproduct \resPosetR
        \to
        \sampleSpaceOmega \setproduct \funPosetF \setproduct \resPosetR,$~$
        \tup{\outcomeomega,\poselxF,\poselxR}
        \mapsto
        \tup{\mapVarphi(\outcomeomega),\poselxF,\poselxR}.$
Then:
\begin{equation*}
\begin{aligned}
        \feasibleSetFOf{\mapVarphi \mthenMathbin \randomVarAlpha}
        & =\setWithArg{\tup{\outcomeomega, \poselxF, \poselxR}}{\tup{\poselxF, \poselxR} \in \mapVarphi \mthenMathbin \randomVarAlpha(\outcomeomega)} \\
        &= \setWithArg{\tup{\outcomeomega, \poselxF, \poselxR}}{\tup{\poselxF, \poselxR} \in \randomVarAlpha(\mapVarphi(\outcomeomega))} \\
        &= \setWithArg{\tup{\outcomeomega, \poselxF, \poselxR}}{\tup{\mapVarphi(\outcomeomega), \poselxF, \poselxR} \in \feasibleSetFOf{\randomVarAlpha}} =\mapVarphitilde^\inv(\feasibleSetFOf{\randomVarAlpha}),
\end{aligned}
\end{equation*}
which is analytic as the pre-image of an analytic set under a measurable map\rebuttal{, according to \cref{thm:measurable-image-pre-image-Souslin}}.
\end{proof}

\rebuttal{
\begin{remark}\label{rmk:dist-f-r-as-dps}
    Stochasticity on \F{functionalities} such as \FI{demands}, or \R{resources} such as \RI{weather conditions}, is encoded by a \glsdp component carrying only the corresponding interface; the Task \glsdp of \cref{fig:uav-adaptive-dp-battery-tech-after-actuation} is an example.
\end{remark}
}
By construction,~$\qusObjDP{\funPosetF}{\resPosetR}{\sampleSpaceOmega}$ contains all constant \glspldp via constant random variables.
\rebuttal{\Cref{def:qus-of-dp} imposes the fewest conditions on \glsplposet and random variables: the former maximizes expressiveness, the latter unifies heterogeneous descriptions of components sharing an interface into one space.
For an actuator's \F{rotary speed}\slash \F{torque} trade-off, one manufacturer may report torque constant and winding resistance, another a parameterized \F{torque}--\F{speed} curve; each reduces to an analytic set as in \cref{def:qus-of-dp}, which is what makes decisions \emph{among} them expressible (\cref{subsec:imp-spec-adaptive-mdpi}).
}

\subsubsection{Compositionality of the QUS of design problems}

All co-design operations on \glspldp turn out to be quasi-measurable on~$\qusdpOf{\funPosetF}{\resPosetR}$, hence lift to distributions via \rebuttal{\cref{lem:qus-maps-lift-dists}}.

\begin{theorem}\label{thm:qus-dp-compositional}
    All composition operations of \cref{def:dp-compositions} (series, parallel, feedback, union, intersection) are quasi-measurable as maps between the corresponding \glsplqus of \glspldp.
    Consequently,
    these operations lift to distributions over \glspldp in $\QUS$.
    The lifted union~$\QUSDistriOf{\union}$ is the design result when one may choose the outcome to implement after observing the randomness; the lifted intersection $\QUSDistriOf{\intersection}$ applies when one must be robust to it.
\end{theorem}
\begin{proof}
\rebuttal{
We show that post-composing any operation of \cref{def:dp-compositions} with random variables whose \glsplofs are analytic yields a random variable with analytic \glsofs.\\
\textbf{Series}: \rebuttal{Let~$\randomVarAlphaa \in \qusSetExpInDP{\posetP}{\posetQ}{\sampleSpaceOmega}$ and~$\randomVarAlphab \in \qusSetExpInDP{\posetQ}{\posetR}{\sampleSpaceOmega}$ with feasible sets~$
\feasibleSetFOf{\randomVarAlphaa}
$ and~$
\feasibleSetFOf{\randomVarAlphab}
$.
Post-composing with $\mthen$ gives}
    \begin{equation*}
        \defmapOneLine
        {\mthen \mapFollowing \tup{\randomVarAlphaa,\randomVarAlphab}}{\sampleSpaceOmega}{\qusdpOf{\posetP}{\posetR}}
        {\outcomeomega}
        {\randomVarAlphaa(\outcomeomega) \mthenMathbin \randomVarAlphab(\outcomeomega)}
    \end{equation*}
    with the \glsofs (with~$\star= \tup{\outcomeomega, \poselxP, \poselxRnoRes}$)
    \begin{align*}
        &
        \setWithArg{\star}{ \tup{\poselxP, \poselxRnoRes} \in \randomVarAlphaa(\outcomeomega) \mthenMathbin \randomVarAlphab(\outcomeomega)} \\
        =
        &
            \setWithArg{\star}{\exists \poselxQ \in \posetQ,
            \tup{\poselxP, \poselxQ} \in \randomVarAlphaa(\outcomeomega), \tup{\poselxQ, \poselxRnoRes} \in \randomVarAlphab(\outcomeomega)} \\
        =
        &
            \setWithArg{\star}{\exists \poselxQ \in \posetQ,
            \tup{\outcomeomega, \poselxP, \poselxQ} \in \feasibleSetFOf{\randomVarAlphaa}, \tup{\outcomeomega, \poselxQ, \poselxRnoRes} \in \feasibleSetFOf{\randomVarAlphab}}\\
        =
        &
        \begin{multlined}[t][\dimexpr\linewidth-2em\relax]
            \setWithArg{\star}{ \exists \poselxQ \in \posetQ,
            \tup{\outcomeomega, \poselxP, \poselxQ, \poselxRnoRes}
            \\
            \in \proj_{1,2,3}^\inv(\feasibleSetFOf{\randomVarAlphaa}),
            \tup{\outcomeomega, \poselxP, \poselxQ, \poselxRnoRes} \in \proj_{1,3,4}^\inv(\feasibleSetFOf{\randomVarAlphab})}
        \end{multlined} \\
        =
        &
        \proj_{1,2,4}\big( \proj_{1,2,3}^\inv(\feasibleSetFOf{\randomVarAlphaa}) \cap
        \proj_{1,3,4}^\inv(\feasibleSetFOf{\randomVarAlphab}) \big),
    \end{align*}
    \rebuttal{which is analytic since \cref{thm:measurable-image-pre-image-Souslin} applies (projections are continuous) and set intersection preserves analytic.}\\
    \textbf{Parallel}:
    \rebuttal{Let $\randomVarAlphaa \in \qusSetExpInDP{\posetP}{\posetQ}{\sampleSpaceOmega}$ and $\randomVarAlphaa' \in \qusSetExpInDP{\posetP'}{\posetQ'}{\sampleSpaceOmega}$, we similarly have the \glsofs of $\stack \mapFollowing \tup{\randomVarAlphaa,\randomVarAlphaa'}$ being analytic, with $\singletonEl = \tup{\outcomeomega, \poselxP, \poselxP', \poselxQ, \poselxQ'}$ and projection maps as above}:
    \begin{align*}
    &
    \begin{multlined}[t][\dimexpr\linewidth-2em\relax]
        \setWithArg{ \singletonEl }{ \tup{\tup{\poselxP, \poselxP'}, \tup{\poselxQ, \poselxQ'}} 
        \in \randomVarAlphaa(\outcomeomega) \stackMathbin \randomVarAlphaa'(\outcomeomega)}
    \end{multlined}
    \\
        =
        &
        \begin{multlined}[t][\dimexpr\linewidth-2em\relax]
            \setWithArg{ \singletonEl  }{ \tup{\outcomeomega, \poselxP, \poselxQ} \in \feasibleSetFOf{\randomVarAlphaa},
            \tup{\outcomeomega, \poselxP', \poselxQ'} \in \feasibleSetFOf{\randomVarAlphaa')} }
        \end{multlined}
        \\
        =
        &
        \proj_{1,2,4}^\inv(\feasibleSetFOf{\randomVarAlphaa}) \cap \proj_{1,3,5}^\inv(\feasibleSetFOf{\randomVarAlphaa'}).
    \end{align*}
    \rebuttal{\textbf{Feedback\slash Trace\slash Loop}:
    Take $\randomVarAlpha \in \qusSetExpInDP{\posetP \posetproduct \posetR}{\posetQ \posetproduct \posetR}{\sampleSpaceOmega}$ and the diagonal set $\SetD \defeq \setWithArg{\tup{\tilde\poselxRnoRes,\tilde\poselxRnoRes'}}{\tilde\poselxRnoRes = \tilde\poselxRnoRes'}$, $\trace \mapFollowing \randomVarAlpha$ has \glsofs}
    \begin{align*}
    & 
    \setWithArg{ \tup{\outcomeomega, \poselxP, \poselxQ} }{ \tup{\poselxP, \poselxQ} \in \trace(\randomVarAlpha(\outcomeomega)) } \\
    =
    &
    \setWithArg{\tup{\outcomeomega, \poselxP, \poselxQ}}{\exists \poselxRnoRes, \tup{\outcomeomega, \tup{\poselxP,\poselxRnoRes}, \tup{\poselxQ,\poselxRnoRes}} \in \feasibleSetFOf{\randomVarAlpha}} \\
    =
    &
    \begin{multlined}[t][\dimexpr\linewidth-2em\relax]
        \setWithArg{\tup{\outcomeomega, \poselxP, \poselxQ}}{\exists \poselxRnoRes, \poselxRnoRes',
        \tup{\outcomeomega, \poselxP,\poselxRnoRes, \poselxQ,\poselxRnoRes'} 
        \in \feasibleSetFOf{\randomVarAlpha},
        \\
        \tup{\outcomeomega, \poselxP,\poselxRnoRes, \poselxQ,\poselxRnoRes'}
        \in \setWithArg{\tup{\tilde\outcomeomega, \tilde\poselxP,\tilde\poselxRnoRes, \tilde\poselxQ,\tilde\poselxRnoRes'}}{\tilde\poselxRnoRes = \tilde\poselxRnoRes'} }
    \end{multlined} \\
    =
    &
    \proj_{1,2,4}(\feasibleSetFOf{\randomVarAlpha} \cap \proj_{3,5}^\inv(\SetD)).
    \end{align*}
    It is an analytic set, since the diagonal set $\SetD \subseteq \posetR \posetproduct \posetR$ is a Borel set (thus analytic) \rebuttal{and similar arguments as above apply}.\\
    \textbf{Union and Intersection}:
    For two random variables $\randomVarAlphaa, \randomVarAlphab \in \qusSetExpInDP{\posetP}{\posetQ}{\sampleSpaceOmega}$, post-composing union and intersection gives \glsplofs (with~$\star=\tup{\outcomeomega, \poselxP, \poselxQ}$)
    \begin{align*}
        &
            \setWithArg{\star}{
            \tup{\poselxP, \poselxQ} \in \randomVarAlphaa(\outcomeomega) \text{ or/and } \tup{\poselxP, \poselxQ} \in \randomVarAlphab(\outcomeomega)}\\
        =
        &
        \setWithArg{\star}{ \tup{\outcomeomega, \poselxP, \poselxQ} \in \feasibleSetFOf{\randomVarAlphaa} \cup / \cap \feasibleSetFOf{\randomVarAlphab}} \\
        =
        &
        \feasibleSetFOf{\randomVarAlphaa} \cup / \cap \feasibleSetFOf{\randomVarAlphab},
    \end{align*}
    both being analytic.
    Their semantics of choosing after observing the outcome, and being robust against two different outcomes, are self-explanatory in the expressions.
}
\end{proof}

\begin{remark}\label{rmk:qus-dp-enriched-category}
\rebuttal{\CTremark}
    With \glsplposet satisfying \cref{asm:posets-souslin-order-borel} as objects and \glsplqus of \glspldp as morphisms, \cref{thm:qus-dp-compositional} yields a symmetric monoidal category enriched in $\QUS$.
    \rebuttal{
    With additional structure from~\cite{forreQuasiMeasurableSpaces2021} it is even enriched in the Kleisli category of $\QUSDistri$; this abstraction is not needed here, but exemplifies \cite{furter2025composableuncertaintysymmetricmonoidal}, which systematically adds uncertainties to general categories.
    }
\end{remark}

\begin{remark}\label{rmk:composing-distri-dps}
    \rebuttal{\Cref{thm:qus-dp-compositional,lem:qus-maps-lift-dists} guarantee that the lifting can be performed, but are not constructive: they do not say how to compose two distributions, because of how product distributions are built in \glsplqus \cite{forreQuasiMeasurableSpaces2021}.}
    Composing $\tup{\distmu_a, \randomVarAlphaa} \in \QUSDistriOfDP{\posetP}{\posetQ}$ in series with $\tup{\distmu_b, \randomVarAlphab} \in \QUSDistriOfDP{\posetQ}{\posetR}$ requires
    \begin{multline*}
        \QUSDistriOfDP{\posetP}{\posetQ} \measureProduct \QUSDistriOfDP{\posetQ}{\posetR}
        \\
        \mapArrOf{\qusDistriProduct}
        \QUSDistriOf{\qusdpOf{\posetP}{\posetQ} \measureProduct \qusdpOf{\posetQ}{\posetR}}
        \\
        \mapArrOf{\QUSDistriOf{\mthen}}
        \QUSDistriOfDP{\posetP}{\posetR},
    \end{multline*}
    \rebuttal{i.e., $\tup{\distmu_a, \randomVarAlphaa} \measureProduct \tup{\distmu_b, \randomVarAlphab} \mapsto \tup{\distmu, \randomVarAlpha}$ and then $\tup{\distmu, \randomVarAlpha} \mapsto \tup{\distmu, \randomVarAlpha \mapFollowing \mthen}$,}
    \rebuttal{where $\tup{\distmu, \randomVarAlpha}$ induces on $\qusdpOf{\posetP}{\posetQ} \measureProduct \qusdpOf{\posetQ}{\posetR}$ the product of the two factors' measures; it exists by \cite[Theorem 5.29]{forreQuasiMeasurableSpaces2021} and resembles the usual product distribution.
    If the sample space decomposes as $\sampleSpaceOmega = \sampleSpaceOmega_a \measureProduct \sampleSpaceOmega_b$ with $\randomVarAlphaa(\tup{\outcomeomega_a, \outcomeomega_b}) = \randomVarAlpha_{a,a}(\outcomeomega_a)$ and $\randomVarAlphab(\tup{\outcomeomega_a, \outcomeomega_b}) = \randomVarAlpha_{b,b}(\outcomeomega_b)$, however, then $\tup{\distmu, \randomVarAlpha} = \tup{\distmu_{a,a} \measureProduct \distmu_{b,b}, \randomVarAlpha_{a,a} \measureProduct \randomVarAlpha_{b,b}}$ with $\distmu_{a,a}, \distmu_{b,b}$ the marginals.
    As $\sampleSpaceOmega$ is usually a product of all randomness sources, such decompositions are common.}
\end{remark}

\subsubsection{Feasibility evaluation and constant injection}
With all \I{implementations} fixed, the design result is a distribution $\tup{\distmu, \randomVarAlpha} \in \QUSDistriOfDP{\funPosetF}{\resPosetR}$, and one asks \emph{what practical questions it answers}; the feasibility evaluation turns out to be quasi-measurable.

\begin{theorem}\label{thm:countable-design-requirements-measurable}
    The evaluation map
    \begin{equation*}
        \defmap{\ev}
        {\qusdpOf{\funPosetF}{\resPosetR} \setproduct \SetExp{\funPosetF}{\Nats} \setproduct \SetExp{\resPosetR}{\Nats}}
        {\setOfBool}
        {\tup{\dprb, \set{\poselxFindex{i}}_{i \in \Nats}, \set{\poselxRindex{i}}_{i \in \Nats}}}
        {\left(
            \forall i \in \Nats, \tup{\poselxFindex{i}, \poselxRindex{i}} \in \dprb
        \right)}
    \end{equation*}
    is quasi-measurable.
\end{theorem}
\begin{proof}
\rebuttal{
    Consider a random variable $\randomVarAlpha$ for $\qusdpOf{\funPosetF}{\resPosetR} \setproduct \SetExp{\funPosetF}{\Nats} \setproduct \SetExp{\resPosetR}{\Nats}$, with $\randomVarAlpha_1 \in \SetExp{\qusdpOf{\funPosetF}{\resPosetR}}{\sampleSpaceOmega}$, $\randomVarAlpha_2 \in \SetExp{\left( \SetExp{\funPosetF}{\Nats} \right)}{\sampleSpaceOmega}$, and $\randomVarAlpha_3 \in \SetExp{\left( \SetExp{\resPosetR}{\Nats} \right)}{\sampleSpaceOmega}$ as its components.
    Post-composing with $\ev$ gives
    \begin{equation*}
        \defmapOneLine{\ev \mapFollowing \randomVarAlpha}
        {\sampleSpaceOmega}{\setOfBool}
        {\outcomeomega}{\left(
            \forall i, \tup{\outcomeomega, \randomVarAlpha_2(\outcomeomega)_i, \randomVarAlpha_3(\outcomeomega)_i} \in \feasibleSetFOf{\randomVarAlpha_1}
        \right).}
    \end{equation*}
    We need to prove this map lives in $\SetExp{\setOfBool}{\sampleSpaceOmega}$. Namely, the pre-image of $\True \in \setOfBool$, $\left(\ev \mapFollowing \randomVarAlpha\right)^\inv(\True)$, is universally measurable.
    For each index $i$, we know the map
    $
        \bar{\randomVarAlpha}_{2,3,i} \colon {\outcomeomega} \mapsto {\tup{\outcomeomega, \randomVarAlpha_2(\outcomeomega)_i, \randomVarAlpha_3(\outcomeomega)_i}}
    $
    is measurable, so the corresponding pre-image
    $
        \bar{\randomVarAlpha}_{2,3,i}^\inv(\feasibleSetFOf{\randomVarAlpha_1})
    $
    is universally measurable.
    $
        \left(\ev \mapFollowing \randomVarAlpha\right)^\inv(\True)
        =
        \bigcap_{i \in \Nats}\bar{\randomVarAlpha}_{2,3,i}^\inv(\feasibleSetFOf{\randomVarAlpha_1})
    $
    is then universally measurable since it is a countable intersection of universally measurable sets.
}
\end{proof}

\rebuttal{\Cref{thm:countable-design-requirements-measurable} thus lets us compute, from a distribution on \glspldp together with distributions over \F{functionalities} and \R{resources}, the probability that a countable family of functionality\slash resource pairs is feasible for a sampled \glsdp; constant distributions recover evaluation at fixed levels.}

\rebuttal{Finally, we show that constant maps are quasi-measurable, so distributions on \F{functionalities}\slash \R{resources} forms distributions on \glspldp, further supporting \cref{rmk:dist-f-r-as-dps}.}

\begin{lemma}\label{lem:sample-space-discrete-order}
The sample space~$\sampleSpaceOmega$, equipped with the discrete order $\outcomeomega \posetleq \outcomeomega' \iff \outcomeomega = \outcomeomega'$ satisfies \cref{asm:posets-souslin-order-borel}.
    We always assign this discrete order to $\sampleSpaceOmega$ without further notice.
\end{lemma}
\begin{proof}
    By definition of \glsplqus, $\sampleSpaceOmega$ is homeomorphic to $\Reals$, hence Souslin.
    The diagonal set~$\set{\tup{\outcomeomega,\outcomeomega}} \subseteq \sampleSpaceOmega \setproduct \sampleSpaceOmega$ is closed, thus analytic and Borel, and coincides with the order relation under the discrete order.
\end{proof}

\begin{theorem}\label{thm:qus-dp-constant-injection}
    The constant injection map
    \begin{equation*}
        \defmapOneLine{\constinj}
        {\funPosetF \setproduct \resPosetR}{\qusdpOf{\funPosetF}{\resPosetR}}
        {\tup{\poselxF, \poselxR}}{\lowerOf{\poselxF} \posetproduct \upperOf{\poselxR}}
    \end{equation*}
    is quasi-measurable.
\end{theorem}
\begin{proof}
    $\set{\tup{\poselxF, \poselxR}}$ is an analytic set in $\funPosetF^\op \posetproduct \resPosetR$, so \rebuttal{$\constinj$ is well-defined since $\lowerOf{\poselxF} \posetproduct \upperOf{\poselxR} = \upperClosure{\set{\tup{\poselxF, \poselxR}}}$ in $\funPosetF^\op \posetproduct \resPosetR$ and \cref{lem:upper-sets-analytic} applies}.
    Pick two random variables $\randomVarAlphaa \in \SetExp{\funPosetF}{\sampleSpaceOmega}$ and $\randomVarAlphab \in \SetExp{\resPosetR}{\sampleSpaceOmega}$.
    We can write the \glsofs $\feasibleSetFOf{\constinj \mapFollowing \tup{\randomVarAlphaa, \randomVarAlphab}}$ as
    \begin{align*}
        &
        \setWithArg{\tup{\outcomeomega, \poselxF, \poselxR}}{
            \poselxF \posetleq \randomVarAlphaa(\outcomeomega), \randomVarAlphab(\outcomeomega) \posetleq \poselxR
        }
        \\
        =
        &
        \begin{multlined}[t]
        \setWithArg{\tup{\outcomeomega, \poselxF, \poselxR}}{
            \tup{\outcomeomega, \poselxF} \in \lowerClosure{\mapGraphOf{\randomVarAlphaa}}, \tup{\outcomeomega, \poselxR} \in \upperClosure{\mapGraphOf{\randomVarAlphab}}
        }
        \end{multlined}
        \\
        =
        &
        \proj_{1,2}^\inv(\lowerClosure{\mapGraphOf{\randomVarAlphaa}}) \cap \proj_{1,3}^\inv(\upperClosure{\mapGraphOf{\randomVarAlphab}})
    \end{align*}
    which is analytic \rebuttal{by \cref{thm:borel-maps-graphs-Souslin,thm:measurable-image-pre-image-Souslin} and set intersection preserving analytic}.
\end{proof}

\subsubsection{Confidence bounds on distributions over design problems}
Distributions over \glspldp \rebuttal{\xspace induce} \emph{confidence bounds}, generalizing interval uncertainty.
As \glspldp are rarely totally ordered, we distinguish levels for sampling \emph{inside} a bound from \emph{outside}.

\begin{definition}[Inner confidence bounds]\label{def:inner-confident-bounds}
\rebuttal{Let~$\tup{\distmu,\randomVarAlpha} \in \QUSDistriOfDP{\funPosetF}{\resPosetR}$ and~$p \in \interval{0}{1}$.
An interval~$\interval{\dprbL}{\dprbU}$ of \glspldp is an \emph{inner confidence bound} of level~$p$ if some universally measurable $\subsetX \subseteq \sampleSpaceOmega$ satisfies $\subsetX \subseteq \randomVarAlpha^\inv(\interval{\dprbL}{\dprbU})$ and $\distmu(\subsetX) \geq p$.
If both the set inclusion and the inequality hold with equality}, we call the bound \emph{strict}.
\end{definition}

\begin{definition}[Outer confidence bounds]\label{def:outer-confident-bounds}
Let~$\tup{\distmu,\randomVarAlpha} \in \QUSDistriOfDP{\funPosetF}{\resPosetR}$ and~$p,q \in \interval{0}{1}$.
\rebuttal{An interval~$\interval{\dprbL}{\dprbU}$ of \glspldp is an \emph{outer confidence bound} of level~$p,q$ if some universally measurable $\subsetX, \subsetY \subseteq \sampleSpaceOmega$ satisfy
$\subsetX \supseteq \randomVarAlpha^\inv(\interval{\emptySet}{\dprbL})$ and $\distmu(\subsetX) \leq p$; $\subsetY \supseteq \randomVarAlpha^\inv(\interval{\dprbU}{\funPosetF^\op \posetproduct \resPosetR})$ and $\distmu(\subsetY) \leq p$.
If all inclusions and inequalities are equalities}, the bound is \emph{strict}.
\end{definition}

\rebuttal{Informally, an inner bound keeps the design result between the bounds with probability \emph{at least} $p$, an outer bound keeps the probabilities of results worse or better than the boundaries \emph{at most} $p$ and $q$; both are strict when exact.}

\begin{remark}
\rebuttal{Confidence bounds need not exist for a given~$\tup{\distmu,\randomVarAlpha}$ and, when they do, are generally not unique even among strict ones; under extra structure, e.g.\ ordered~$\sampleSpaceOmega$ and monotone~$\randomVarAlpha$, they can often be built from subsets of~$\sampleSpaceOmega$ such as intervals.
Although intervals of \glspldp always compose by composing their bounds (\cref{lem:interval-lifted-ops}), the result need not be a valid confidence bound, since the lifted operations of \cref{thm:qus-dp-compositional} are non-constructive; when it is valid, monotonicity implies the level can only deteriorate.
Inner bounds compose whenever the decomposition of \cref{rmk:composing-distri-dps} holds; outer bounds are future work.}
\end{remark}

\rebuttal{
\begin{lemma}\label{lem:adjoint-map-and-inv}
    For each map $\mapf \colon \SetA \mapArr \SetB$ and subsets $\subsetXA \subseteq \SetA$, $\subsetYB \subseteq \SetB$, we have $\mapf(\mapf^\inv(\subsetYB)) \subseteq \subsetYB$ and $\subsetXA \subseteq \mapf^\inv(\mapf(\subsetXA))$.
\end{lemma}
\begin{proof}
    For each $\setelmB$ in $\mapf(\mapf^\inv(\subsetYB))$, we know there exist $\setelmA$ in $\mapf^\inv(\subsetYB)$ so that $\mapf(\setelmA) = \setelmB$, directly giving $\setelmB \in \subsetYB$ by the definition of $\mapf^\inv(\subsetYB)$.
    For each $\setelmA$ in $\subsetXA$, by the definition of $\mapf(\subsetXA)$ we have: $\mapf(\setelmA) \in \mapf(\subsetXA)$.
    Then by the definition of $\mapf^\inv(\mapf(\subsetXA))$, $\setelmA \in \mapf^\inv(\mapf(\subsetXA))$.
\end{proof}
}

\begin{lemma}\label{lem:composing-inner-bounds}
    Suppose we have a diagram of \glspldp \xspace \rebuttal{(for instance \cref{fig:uav-dp-composition} in \cref{sec:numerical-example})}, with $\mapf$ being the deterministic map taking the component \glspldp and returning the composed \glsdp.
    Denote the distributions over \glspldp for each component as $\tup{\distmu_i, \randomVarAlpha_i}$, indexed by a finite set $\SetI$.
    Suppose we have a decomposition of the sample space, namely: $\sampleSpaceOmega = \measureProduct_{i \in I} \sampleSpaceOmega_i$, and $\randomVarAlpha_i(\tup{\outcomeomega_i}_{i \in \SetI}) = \randomVarAlpha_{i,i}(\outcomeomega_i)$.
    Given an inner confident bound $\interval{\dprb_{\lowerBoundfix, i}}{\dprb_{\upperBoundfix, i}}$ for each distribution, together with $\subsetX_i \subseteq \sampleSpaceOmega$ and $p_i$ satisfying conditions in \cref{def:inner-confident-bounds}, the interval composition
    \begin{equation*}
        \IntervalOf{\mapf}(\tup{\interval{\dprb_{\lowerBoundfix,i}}{\dprb_{\upperBoundfix,i}}}_{i \in \SetI}) = \interval{\mapf(\tup{\dprb_{\lowerBoundfix,i}})}{\mapf(\tup{\dprb_{\upperBoundfix,i}})}
    \end{equation*}
    is an inner confidence bound of level \rebuttal{$\Pi_{i \in \SetI} p_i$} for the composed distribution.
\end{lemma}
\begin{proof}
\rebuttal{
    The composed distribution reads: $\tup{\measureProduct_i \distmu_{i,i}, \mapf \mapFollowing \measureProduct_i \randomVarAlpha_{i,i}}$. 
    Let $\subsetX_{i,i} \defeq \proj_i(\subsetX_i)$.
    We will prove the set $\measureProduct_i \subsetX_{i,i} \subseteq \sampleSpaceOmega$ satisfies conditions in \cref{def:inner-confident-bounds}.
    \rebuttal{We have $\proj_i^\inv(\subsetX_{i,i}) \supseteq \subsetX_i$ by \cref{lem:adjoint-map-and-inv}}.
    First observe: 
    $$
        \randomVarAlpha_i^\inv(\interval{\dprb_{\lowerBoundfix, i}}{\dprb_{\upperBoundfix, i}}) = \proj_i^\inv(\randomVarAlpha_{i,i}^\inv(\interval{\dprb_{\lowerBoundfix, i}}{\dprb_{\upperBoundfix, i}}))$$
    and since $\subsetX_i \subseteq \randomVarAlpha_i^\inv(\interval{\dprb_{\lowerBoundfix, i}}{\dprb_{\upperBoundfix, i}})$,
    \rebuttal{
    \begin{multline*}
        \subsetX_{i,i}
        = \proj_i(\subsetX_i)
        \subseteq \proj_i(\randomVarAlpha_i^\inv(\interval{\dprb_{\lowerBoundfix, i}}{\dprb_{\upperBoundfix, i}}))
        \\
        = \proj_i(\proj_i^\inv(\randomVarAlpha_{i,i}^\inv(\interval{\dprb_{\lowerBoundfix, i}}{\dprb_{\upperBoundfix, i}}))).
    \end{multline*}
    With \cref{lem:adjoint-map-and-inv} we further have
    $$\proj_i(\proj_i^\inv(\randomVarAlpha_{i,i}^\inv(\interval{\dprb_{\lowerBoundfix, i}}{\dprb_{\upperBoundfix, i}})))
    \subseteq
    \randomVarAlpha_{i,i}^\inv(\interval{\dprb_{\lowerBoundfix, i}}{\dprb_{\upperBoundfix, i}}),
    $$
    thus
    $\subsetX_{i,i} \subseteq \randomVarAlpha_{i,i}^\inv(\interval{\dprb_{\lowerBoundfix, i}}{\dprb_{\upperBoundfix, i}})$.}
    
    Since all the compositional operations are monotone, we have
    \begin{equation*}
        \mapf(\setproduct_i \interval{\dprb_{\lowerBoundfix,i}}{\dprb_{\upperBoundfix,i}}) \subseteq \interval{\mapf(\tup{\dprb_{\lowerBoundfix,i}})}{\mapf(\tup{\dprb_{\upperBoundfix,i}})},
    \end{equation*}
    thus (first applying \cref{lem:adjoint-map-and-inv})
    \begin{multline*}
        \setproduct_i \interval{\dprb_{\lowerBoundfix,i}}{\dprb_{\upperBoundfix,i}} \subseteq \mapf^\inv(\mapf(\setproduct_i \interval{\dprb_{\lowerBoundfix,i}}{\dprb_{\upperBoundfix,i}})) \\
        \subseteq \mapf^\inv(\interval{\mapf(\tup{\dprb_{\lowerBoundfix,i}})}{\mapf(\tup{\dprb_{\upperBoundfix,i}})}),
    \end{multline*}
    and then
    \begin{multline*}
        \setproduct_i X_{i,i}
        \subseteq \setproduct_i \randomVarAlpha_{i,i}^\inv(\interval{\dprb_{\lowerBoundfix,i}}{\dprb_{\upperBoundfix,i}}) \\
        = \left(\measureProduct_i \randomVarAlpha_{i,i}\right)^\inv(\setproduct_i \interval{\dprb_{\lowerBoundfix,i}}{\dprb_{\upperBoundfix,i}}) \\
        \subseteq \left(\measureProduct_i \randomVarAlpha_{i,i}\right)^\inv\left( \mapf^\inv(\interval{\mapf(\tup{\dprb_{\lowerBoundfix,i}})}{\mapf(\tup{\dprb_{\upperBoundfix,i}})}) \right) \\
        = \left(\mapf \mapFollowing \measureProduct_i \randomVarAlpha_{i,i}\right)^\inv\left( \interval{\mapf(\tup{\dprb_{\lowerBoundfix,i}})}{\mapf(\tup{\dprb_{\upperBoundfix,i}})} \right).
    \end{multline*}
    With \rebuttal{
    \begin{align*}
        \distmu_{i,i}(\subsetX_{i,i}) 
        &
        = \distmu_i(\proj_i^\inv(\subsetX_{i,i})) \geq \distmu_i(\subsetX_i) \geq p_i,
        \\
        \measureProduct_i \distmu_{i,i}(\setproduct_i \subsetX_{i,i})
        &
        = \distmu_{1,1}(\subsetX_{1,1}) \times \distmu_{2,2}(\subsetX_{2,2}) \cdots \geq \Pi_{i\in I} p_i,
    \end{align*}
    }
    $\interval{\mapf(\tup{\dprb_{\lowerBoundfix,i}})}{\mapf(\tup{\dprb_{\upperBoundfix,i}})}$ is an inner bound of level $\Pi_{i\in I} p_i$.
}
\end{proof}

\subsection{Implementations, specifications, and adaptive decisions}\label{subsec:imp-spec-adaptive-mdpi}

In deterministic co-design \glsplmdpi are design processes: choosing \I{implementations} yields deterministic \glspldp (\cref{def:mdpi}).
\rebuttal{
\Cref{def:qus-of-dp} lets us extend this to \emph{\glsplmdpi with distributional uncertainty}, where design choices yield distributions on \glspldp and adaptivity is re-parameterization by Markov kernels in \glsqus (\cref{def:markov-kernels-qus}).
}

\begin{definition}[\Glsplmdpi with distributional uncertainty]\label{def:distri-mdpi}
Let $\funPosetF, \resPosetR$ satisfy \cref{asm:posets-souslin-order-borel}.
An \emph{\glsmdpi with distributional uncertainty} is a tuple~$\tup{\specSetS,\impSetI,\measureKernela}$ with~$\specSetS, \impSetI$ \glsplqus of system \Sp{specifications} and \I{implementations} and~$\measureKernela$ a Markov kernel~$\measureKernela \colon \specSetS \measureProduct \impSetI \MarkovArr \qusdpOf{\funPosetF}{\resPosetR}$, i.e.\ a quasi-measurable map~$\specSetS \measureProduct \impSetI \to \QUSDistriOf{\qusdpOf{\funPosetF}{\resPosetR}}$.
\end{definition}

\rebuttal{For a one-shot (non-adaptive) \glsuav design, $\impSetI$ holds all design decisions and $\specSetS$ the specifications known \emph{before} deciding: $\impi$ is e.g.\ a combination of sensors, actuators, body, and algorithms, $\Sp{s}$ e.g.\ sensing distances and mean\slash variance of actuator efficiency, and $\measureKernela(\impi, \Sp{s})$ the induced distribution on design results.}
Although all uncertainty could be folded into~$\sampleSpaceOmega$, keeping \Sp{specifications} explicit is what lets us \emph{re-parameterize} design processes and express multi-stage adaptive decisions.

\begin{definition}[Re-parameterization of distributional uncertain \glsplmdpi]\label{def:repara-distri-mdpi}
Given a distributional uncertain \glsmdpi $\tup{\specSetS, \impSetI, \measureKernela}$, re-parameterizing it with another Markov kernel $\measureKernelf \colon \specSetS' \measureProduct \impSetI' \MarkovArr \specSetS \measureProduct \impSetI$ gives distributional uncertain \glsmdpi $\tup{\specSetS', \impSetI', \measureKernela \QUSMarkovFollowing \measureKernelf}$.
\end{definition}

\begin{figure}[tb]
    \centering
    \includegraphics[width=1.0\linewidth]{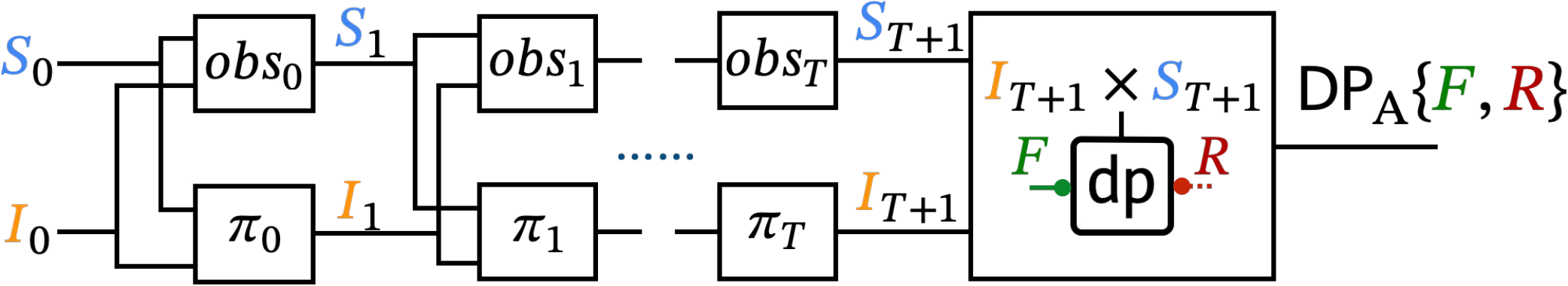}
    \caption{Diagram of $\measureKernela \QUSMarkovFollowing \tup{\measureKernelObs_{T}, \measureKernelPolicy_{T}} \QUSMarkovFollowing \cdots \QUSMarkovFollowing \tup{\measureKernelObs_{1}, \measureKernelPolicy_{1}} \QUSMarkovFollowing \tup{\measureKernelObs_{0}, \measureKernelPolicy_{0}} \colon \specSetS_0 \measureProduct \impSetI_0 \MarkovArr \qusdpOf{\funPosetF}{\resPosetR}$, a $T$-stage adaptive decision process.
    The right-most kernel, $\measureKernela \colon \specSetS_{T+1} \measureProduct \impSetI_{T+1} \MarkovArr \qusdpOf{\funPosetF}{\resPosetR}$, is illustrated in a similar way to those in \cref{fig:uav-adaptive-dp-battery-tech-after-actuation}.}
    \label{fig:T-stage-decision-process}
\end{figure}

Re-parameterization captures \emph{multi-stage adaptive decision processes}.
At each stage $k$, specifications $\specSetS_k$ and implementations $\impSetI_k$ are determined from the previous stage via a Markov kernel~$\tup{\measureKernelObs_{k-1}, \measureKernelPolicy_{k-1}} \colon \specSetS_{k-1} \measureProduct \impSetI_{k-1} \MarkovArr \specSetS_k \measureProduct \impSetI_k$ (\cref{fig:T-stage-decision-process}).
It has two components: (a) an \emph{observation model}~$\measureKernelObs_{k-1}
        \colon
        \specSetS_{k-1} \measureProduct \impSetI_{k-1} \MarkovArr \specSetS_k,$ describing how the environment and specifications evolve \rebuttal{(e.g. observing performance of the chosen actuator)}; (b) a \emph{decision policy}~$\measureKernelPolicy_{k-1}
        \colon
        \specSetS_{k-1} \measureProduct \impSetI_{k-1} \MarkovArr \impSetI_k,$ choosing the next implementation (possibly randomly) based on past specifications and decisions \rebuttal{(e.g. choosing a sensing algorithm based on performance of the chosen sensor)}.

\begin{remark}
\rebuttal{The space of Markov kernels~$\qusOf{\SetA}{\QUSDistriOf{\SetB}}$ is itself a \glsqus with quasi-measurable evaluation \cite{forreQuasiMeasurableSpaces2021}, so policies can be \I{implementations}; however parameterized policies are inpractice more convenient.}
\end{remark}

\rebuttal{

\begin{remark}\label{rmk:two-category-distri-mdpi}
The lifted operations of \cref{thm:qus-dp-compositional} and the re-parameterizations of \cref{def:repara-distri-mdpi} give two layers of adaptivity: (1) the \emph{policy level}, via the kernels $\measureKernelPolicy_k$, and (2) the \emph{co-design level}, since composition happens after all uncertainties relevant to the composite are observed.
In interval co-design (\cref{subsec:interval-uncertain-dp}) the optimistic and pessimistic \glspldp are solved independently, so choices may differ between the two and are all made at the \emph{co-design level}, after observing every relevant uncertainty.
\end{remark}

Re-parameterizations can be represented by arbitrary string diagrams, as in \cref{fig:uav-diagrams-different-adaptive-levels}.

\begin{remark}
\label{rmk:ct-monad-kleisli}
\rebuttal{\CTremark}
With structural quasi-measurable maps (duplication, deletion), $\QUSDistri$ extends to a strong monad \rebuttal{(adding structure compatible with series\slash parallel composition)} on~$\QMS$\rebuttal{ \cite{forreQuasiMeasurableSpaces2021}}, and restricting to \glsplqus makes it commutative, so its Kleisli category \rebuttal{(morphisms carrying the added structure)} is symmetric monoidal and arbitrary diagrams of Markov kernels are describable\rebuttal{ \cite{heunenConvenientCategoryHigherOrder2017,forreQuasiMeasurableSpaces2021}}.
Categorical coherence guarantees any diagram is a well-defined composite kernel \cite{selingerSurveyGraphicalLanguages2011,fritzSyntheticApproachMarkov2020}, as co-design diagrams like \cref{fig:uav-dp-composition} (bottom) define a proper \glsdp \cite{zardiniCoDesignComplexSystems2023}; both structures organize into a 2-category with uncertain co-design problems as 1-cells and re-parameterizations as 2-cells~\cite{furter2025composableuncertaintysymmetricmonoidal}.
\end{remark}

Conditionals and disintegrations of Markov kernels are not part of our basic structure \cite{forreQuasiMeasurableSpaces2021} but are powerful for learning and solving (hidden) Markov models \cite{fritzHiddenMarkovModels2025b}; with the literature on finite-horizon, varying-state-space decision processes \cite{feinbergHandbookMarkovDecision2002,liuSolutionTimeVaryingMarkov2018,bloor2025gaussian} they are a natural basis for efficient algorithms, left as future work.
Compositionality and re-parameterization thus give a quantitative, adaptive language: uncertain \glsplmdpi are defined per component, then composed and re-parameterized by kernels representing the design process (\cref{fig:uav-adaptive-dp-battery-tech-after-actuation}).}

\subsection{Queries and observations in co-design with distributional uncertainty}\label{subsec:queries-distri-mdpi}

\rebuttal{We now formalize what one can ask of distributional uncertain \glsplmdpi: \emph{queries} aggregate information about distributions over \glspldp, while \emph{observations} assign a quantity to each individual \glsdp and lift to queries via $\QUSDistri$; map spaces that \glsplqus provide greatly increase their flexibility (\cref{ex:fix-func-min-res-distri}).}

\subsubsection{Queries}
\begin{definition}[Queries for distributional uncertain \glsplmdpi]\label{def:distri-mdpi-queries}
\rebuttal{For a distributional uncertain \glsmdpi~$\tup{\specSetS,\impSetI,\measureKernela}$ with \glsplposet $\funPosetF, \resPosetR$, a \emph{query} is a map~$\maph \colon \QUSDistriOfDP{\funPosetF}{\resPosetR} \to \SetM$ for some set $\SetM$.}
\end{definition}

\rebuttal{We do not require $\maph$ quasi-measurable: querying is the final modeling step and measurability does not affect expressiveness.
When analyzing or optimizing $\specSetS \setproduct \impSetI \mapArrOf{\measureKernela} \QUSDistriOfDP{\funPosetF}{\resPosetR} \mapArrOf{\maph} \SetM$ one may impose measurability, continuity, or convexity on $\maph$ or $\maph \mapFollowing \measureKernela$, depending on the task.}

\begin{example}[Probability of feasible free choice]
\rebuttal{Currying the evaluation of \cref{thm:countable-design-requirements-measurable} gives a query for the probability of providing a free choice among countably many \F{functionality}\slash \R{resource} pairs: $\probquery \colon \QUSDistriOfDP{\funPosetF}{\resPosetR} \to \left( \SetExp{\left(\funPosetF \setproduct \resPosetR\right)}{\Nats} \to \interval{0}{1} \right)$,
    $
        \probquery(\tup{\distmu, \randomVarAlpha})\big(\set{\tup{\poselx_{\funPosetF, i}, \poselx_{\resPosetR, i}}}_{i \in \Nats}\big)
        =
        \distmu(\setWithArg{\outcomeomega}{\forall i, \tup{\poselx_{\funPosetF, i}, \poselx_{\resPosetR, i}} \in \randomVarAlpha(\outcomeomega)}).
    $
}
\end{example}

\begin{example}[Fix functionalities, minimize resources (probabilistic)]\label{ex:fix-func-min-res-distri}
\rebuttal{The probabilistic \emph{fix \F{functionalities}, minimize \R{resources}} query returns, for each countable family of required \F{functionalities}, the probability that each \R{resource} is sufficient:
    \begin{multline*}
    \fixfunminres \colon
        {\QUSDistriOfDP{\funPosetF}{\resPosetR}}
        \to
        {\left( \SetExp{\funPosetF}{\Nats} \to \left(\resPosetR \to \interval{0}{1}\right) \right)},\\
        \tup{\distmu, \randomVarAlpha}
        \mapsto
        \left(
        \begin{multlined}[c]
            \set{\poselx_{\funPosetF, i}}_{i \in \Nats}
            \mapsto
            \\
            \left(
            \begin{multlined}
                \poselxR
                \mapsto
                \distmu(\setWithArg{\outcomeomega}{\forall i, \tup{\poselx_{\funPosetF, i}, \poselx_{\resPosetR}} \in \randomVarAlpha(\outcomeomega)})
            \end{multlined}
            \right)
        \end{multlined}
        \right).
    \end{multline*}
}
\end{example}

\rebuttal{Some questions are defined only for certain distributions, giving partial queries.}

\begin{definition}[Partial queries]\label{def:partial-queries}
    A \emph{partial query} is a partial map $\maph \colon \QUSDistriOfDP{\funPosetF}{\resPosetR} \partialmapArr \SetM$.
\end{definition}

\begin{example}[Confidence intervals and point distributions]
\rebuttal{Where confidence intervals
(\cref{def:inner-confident-bounds,def:outer-confident-bounds})
can be computed,
the map $\confidenceInterval \colon \QUSDistriOfDP{\funPosetF}{\resPosetR} \partialmapArr \IntervalOf{\dpOf{\funPosetF}{\resPosetR}}$ is a partial query; post-composing it with a deterministic map $\mapq \colon \qusdpOf{\funPosetF}{\resPosetR} \mapArr \SetM$ yields another partial query $\QUSDistriOfDP{\funPosetF}{\resPosetR} \partialmapArrOf{\confidenceInterval} \IntervalOf{\dpOf{\funPosetF}{\resPosetR}} \mapArrOf{\IntervalOf{\mapq}} \IntervalOf{\SetM}$.
A second is $\const_\QUSDistri^\inv \colon \QUSDistriOf{\const}\left( \QUSDistriOf{\funPosetF \posetproduct \resPosetR} \right) \mapArr \QUSDistriOf{\funPosetF \posetproduct \resPosetR}$, satisfying $\const_\QUSDistri^\inv \mapFollowing \QUSDistriOf{\const} = \idmap$ and sending each distribution on \glspldp induced by one on $\funPosetF \posetproduct \resPosetR$ back to it.}
\end{example}

\begin{remark}
A partial query $\maph$ applied to a specific uncertain \glsmdpi~$\tup{\specSetS,\impSetI,\measureKernela}$ may still give a total map~$\maph \mapFollowing \measureKernela \colon \specSetS \setproduct \impSetI \to \SetM$, provided all design results lie in its domain,~$\measureKernela(\specSetS \setproduct \impSetI) \subseteq \domOf{\QUSDistriOf{\qusdpOf{\funPosetF}{\resPosetR}}}{\maph}$.
\end{remark}

\subsubsection{Observations}
\begin{definition}[Observations for realizations of \glsplmdpi]\label{def:observartion-distri-mdpi}
For a distributional uncertain \glsmdpi with \glsplposet $\funPosetF, \resPosetR$, an \emph{observation} is a quasi-measurable $\mapr \colon \qusdpOf{\funPosetF}{\resPosetR} \to \SetO$ into a \glsqus $\SetO$.
\end{definition}

\begin{remark}
\rebuttal{Quasi-measurability enables learning and Bayesian updating (datasets, priors and posteriors, an aleatoric\slash epistemic split), and lifts every observation to a query $\QUSDistriOf{\mapr} \colon \QUSDistriOfDP{\funPosetF}{\resPosetR} \to \QUSDistriOf{\SetO}$ from distributions on \glspldp to distributions on outcomes.
For a \glsposet $\SetO$ and monotone $\mapr$, an inner bound $\interval{\dprbL}{\dprbU}$ induces $\interval{\mapr(\dprbL)}{\mapr(\dprbU)}$ at the same level.}
\end{remark}

Under mild conditions on $\resPosetR$ one can define an observation returning the \emph{infimal resource requirement} for a set of functionalities.

\begin{lemma}\label{lem:sequence-poset-injection-qus}
\rebuttal{For every \glsqms $\posetP$, the constant sequence injection~$\constinj_{\Nats}:
    \posetP\to \SetExp{\posetP}{\Nats},$~$ 
    \poselxP\mapsto \set{\poselxP}_{i \in \Nats}$
is quasi-measurable by the definition of countable products of \glsplqms \cite{forreQuasiMeasurableSpaces2021}.}
\end{lemma}

\begin{lemma}\label{lem:total-order-inf-measurable}
Let $\posetP$ be a \glsposet satisfying \cref{asm:posets-souslin-order-borel}.
If:
\begin{enumerate}[nosep]
    \item \rebuttal{it is} totally ordered.
    \item each measurable set has an infimum.
    We denote $-\infty \defeq \inf(\posetP)$ and $\infty \defeq \inf(\emptySet)$.
    \item the Borel sigma algebra is generated by open and unbounded open intervals, i.e., $(a, b)$, $[-\infty, b)$, and $(a, \infty]$.
\end{enumerate}
Then the infimum map~$\inf:
    \QUSBorel(\posetP)\to \posetP,$~${\SetA}\mapsto {\inf(\SetA)}$
is quasi-measurable, where the random variables $\SetExp{\QUSBorel(\posetP)}{\sampleSpaceOmega}$ are defined as in \cite[Definition 2.47]{forreQuasiMeasurableSpaces2021}.
\end{lemma}
\begin{proof}
\rebuttal{
    Let $\randomVarBeta$ be a random variable in $\SetExp{\QUSBorel(\posetP)}{\sampleSpaceOmega}$.
    We need to prove $\inf \mapFollowing \randomVarBeta \in \SetExp{\posetP}{\sampleSpaceOmega}$.
    Because of how the Borel sigma algebra is generated, we only need to prove pre-images of two types of intervals: $[-\infty,b]$ and $[\infty,b)$, are universally measurable.
    From \cite[Definition 2.47]{forreQuasiMeasurableSpaces2021}, we know that there exists $\SetD \in \QUSBorel(\sampleSpaceOmega \setproduct \posetP)$, making $\poselxP \in \randomVarBeta(\outcomeomega) \iff \tup{\outcomeomega, \poselxP} \in \SetD$.
    We can compute the pre-image of $[-\infty,b]$ by
    \begin{align*}
        &
        \left( \inf \mapFollowing \randomVarBeta \right)^\inv([-\infty, b])
        \\
        =
        &
        \setWithArg{\outcomeomega}{\inf(\randomVarBeta(\outcomeomega)) \posetleq b}
        \\
        =
        &
        \setWithArg{\outcomeomega}{\exists \poselxP \in \randomVarBeta(\outcomeomega), \poselxP \posetleq b}
        \\
        =
        &
        \setWithArg{\outcomeomega}{\exists \poselxP, \tup{\outcomeomega, \poselxP} \in \SetD, \poselxP \posetleq b}
        \\
        =
        &
        \setWithArg{\outcomeomega}{\exists \poselxP, \tup{\outcomeomega, \poselxP} \in \SetD, \tup{\outcomeomega, \poselxP} \in \proj_{2}^\inv([\infty, b])}
        \\
        =
        &
        \proj_{1}(\SetD \cap \proj_{2}^\inv([\infty, b])).
    \end{align*}
    For the interval $[\infty, b)$, we only need to take the strict inequalities and have $\left( \inf \mapFollowing \randomVarBeta \right)^\inv([-\infty, b)) = \proj_{1}(\SetD \cap \proj_{2}^\inv([\infty, b))$.
    Both of them are universally measurable \rebuttal{by \cref{thm:measurable-image-pre-image-Souslin} and set intersection preserving analytic}.
}
\end{proof}

\rebuttal{Its conditions hold for any second countable totally ordered set with order-generated topology containing its limits, e.g.\ $\Reals$ with $\pm\infty$ or closed intervals of $\Reals$.}

\begin{example}[Distribution of minimal resource]\label{ex:distri-minimal-resource}
\rebuttal{Suppose $\resPosetR$ satisfies the conditions of \cref{lem:total-order-inf-measurable} and fix an at most countable family of functionalities $\set{\poselxFindex{i}}_{i \in \Nats}$.
We construct an observation~$r_{\min} \colon \qusdpOf{\funPosetF}{\resPosetR} \to \resPosetR$ returning the infimal resource providing all of them.
Currying $\ev$ of \cref{thm:countable-design-requirements-measurable} gives a quasi-measurable $\bar{\ev} \colon \qusdpOf{\funPosetF}{\resPosetR} \to \SetExp{\setOfBool}{\SetExp{\funPosetF}{\Nats} \setproduct \SetExp{\resPosetR}{\Nats}}$ (\cref{lem:qms-curry}); pre-composing with~$\const_{\{\poselxFindex{i}\}_{i \in \Nats}} \setproduct \idmap$ and then $\const_{\Nats}$ (\cref{lem:sequence-poset-injection-qus}) lands in $\SetExp{\setOfBool}{\resPosetR}$; with the quasi-measurable bijection $\chi$ of~\cite[Lemma~2.51]{forreQuasiMeasurableSpaces2021} between $\SetExp{\setOfBool}{\resPosetR}$ and $\QUSBorel(\resPosetR)$ and the infimum map of \cref{lem:total-order-inf-measurable}, we obtain}
\begin{multline*}
        \qusdpOf{\funPosetF}{\resPosetR}
        \mapArrOf{\bar{\ev}}
        \SetExp{\setOfBool}{\SetExp{\funPosetF}{\Nats} \setproduct \SetExp{\resPosetR}{\Nats}}
        \\
        \mapArrOf{(-) \mapFollowing \left(\const_{\{\poselxFindex{i}\}_{i \in \Nats}} \setproduct \idmap\right)}
        \SetExp{\setOfBool}{\SetExp{\resPosetR}{\Nats}}
        \\
        \mapArrOf{(-) \mapFollowing \constinj_\Nats}
        \SetExp{\setOfBool}{\resPosetR}
        \mapArrOf{\chi}
        \QUSBorel(\resPosetR)
        \mapArrOf{\inf}
        \resPosetR.
    \end{multline*}
    \rebuttal{On $\dprb \in \qusdpOf{\funPosetF}{\resPosetR}$ this returns}
    \[
        \mapr_{\min}(\dprb)
        =
        \inf\Big(
            \setWithArg{\poselxR}{
                \forall i,\;
                \tup{\poselxFindex{i},\poselxR} \in \dprb}
        \Big),
    \]
    \rebuttal{the infimal resource satisfying all required functionalities.
    Lifting $\mapr_{\min}$ via $\QUSDistri$ gives a query $\QUSDistriOf{r_{\min}} \colon \QUSDistriOfDP{\funPosetF}{\resPosetR} \to \QUSDistriOf{\resPosetR}$ from distributions on \glspldp to induced distributions on minimal resources.
    An inner confidence bound $\interval{\dprbL}{\dprbU}$ yields $\interval{\mapr_{\min}(\dprbL)}{\mapr_{\min}(\dprbU)}$, a bound on the minimal resource at the same probability.
    Currying $\ev$ over $\resPosetR^\Nats$ further lets $\mapr_{\min}$ take the required \F{functionalities} as extra inputs and stay quasi-measurable.}
\end{example}

\section{Numerical Example: uncertain task-driven co-design of \glsentrylongpl{abk:uav}}\label{sec:numerical-example}
We instantiate the framework on a task-driven \glsuav delivery problem adapted from~\cite{censi2017uncertainty, huang2025composable}.
\rebuttal{The goal is to expose in one benchmark the gap between (i) deterministic co-design, (ii) interval uncertainty, and (iii) distributional uncertainty with adaptive decision-making.
\Cref{subsec:task-driven-uav} gives the co-design model and decomposition; \cref{subsec:deterministic-interval-uav} solves the deterministic and interval baselines and clarifies how to read the optimistic\slash pessimistic envelope; \cref{subsec:distri-uav-max-adaptive} adds distributional uncertainty in the task profile and component specifications, computing distributions of optimal cost by Monte Carlo; \cref{subsec:design-results-with-adaptive-levels-uav} compares procedures differing in \emph{which} choices may depend on \emph{which} observations, in the language of \cref{subsec:imp-spec-adaptive-mdpi}.}

\subsection{Task-driven \glsentrylong{abk:uav} co-design}\label{subsec:task-driven-uav}

\begin{figure}[tb]
    \centering
    \includegraphics[width=\columnwidth]{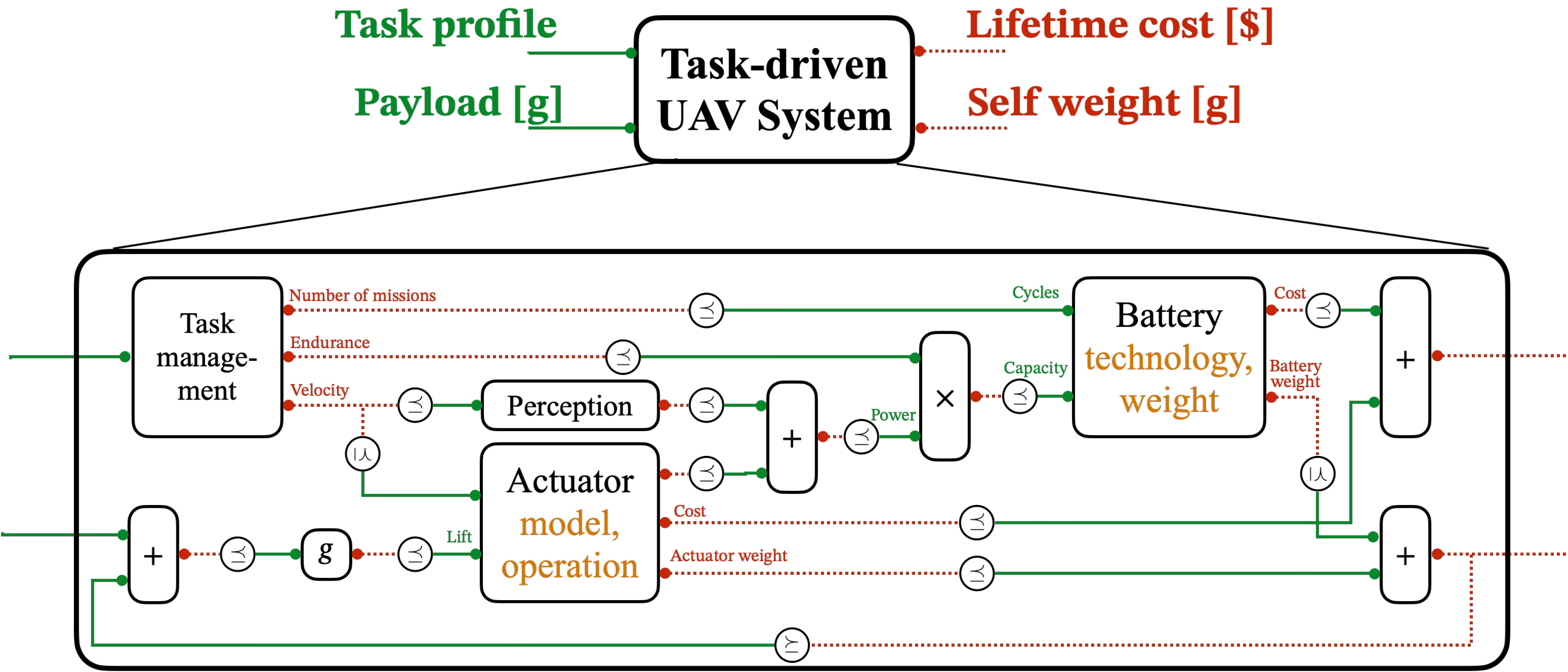}
    \caption{Co-design diagram for a task-driven \gls{abk:uav}, decomposed into functional components; designs carry \I{implementation} color.
    \rebuttal{Red and green lines leaving the lower diagram are the four system-level \F{functionalities}\slash \R{resources} of the interface above.}}
    \label{fig:uav-dp-interface}
    \label{fig:uav-dp-composition}
\end{figure}

\rebuttal{An \gls{abk:uav} must complete delivery missions over its operational lifetime.
A \FI{task profile} comprises \FI{number of delivery missions}, \FI{distance coverage}, and \FI{mission frequency}; given one, the objective is to maximize \FI{payload} while minimizing \RI{lifetime cost} and \RI{self weight}.
In co-design this is a system \gls{abk:dp} providing \FI{task profile} and \FI{payload} and requiring \RI{lifetime cost} and \RI{self weight} (\cref{fig:uav-dp-composition}, top), decomposed as in \cref{fig:uav-dp-composition} (bottom) following~\cite{zardiniCoDesignComplexSystems2023}.
For interpretability, explicit \I{implementations} appear only in \emph{Actuation} (model, operating point) and \emph{Battery} (technology, weight); co-design extends to richer architectures with minimal overhead~\cite{zardiniecc21, milojevic2025codei}.}

\rebuttal{Recall that colors record the \emph{role} a quantity plays at the interface where it appears, not an intrinsic property:
\FI{velocity} is a \F{functionality} the actuation component provides, hence green at the actuation interface, while it is a \R{resource} required by the task-management block, hence red there \cref{fig:uav-dp-composition}.
This role reversal is performed by series composition.}

\paragraph*{Task management}
\rebuttal{For each \FI{task profile}, this block derives intermediate requirements: \RI{number of missions}, \RI{endurance}, \RI{velocity}.}

\paragraph*{Perception}
The perception stack (sensor and software) is fixed, hence has no \I{implementation} choice; \rebuttal{its \RI{power} consumption increases with \FI{velocity}}~\cite{karaman2012high}.

\paragraph*{Actuation}
\rebuttal{Three actuator \I{models} are considered(\cref{tab:param-actuator-batteries}), each with deterministic \RI{weight}, \RI{cost}, and \FI{maximum velocity}.
An \I{operation} decision allocates \FI{power}~$\powerp$ to \F{lift}~$\lift$ via $\powerp \ge p_{0} + p_{1}\lift^2$, with $p_0, p_1$ \Sp{specifications} of this \glsdp.
Here the \emph{parameters} of design choices (some uncertain later) coincide with the \glsdp \Sp{specifications}; in general they may differ, as long as each choice's parameters yield a component \Sp{specification}.}

\paragraph*{Battery}
\rebuttal{A battery design is a \I{technology} choice (\cref{tab:param-actuator-batteries}) plus a \I{battery weight} decision.
Its \glsdp is specified by \SpI{energy density}, \SpI{unit energy per cost}, and \SpI{cycle life} (also the technology parameters), fixing the trade-off between provided \FI{capacity}\slash \FI{cycles} and required \RI{battery weight}\slash \RI{total cost}, maintenance and replacement included.}

\begin{table}[tb]
    \centering
    \footnotesize\setlength{\tabcolsep}{3pt}
    \begin{tabular}{l C{0.72cm} C{0.72cm} C{0.72cm} C{0.72cm} C{0.72cm}}
        \toprule
        {Actuator} & {Mass} [\si{g}] & {Cost} \newline [\si{\$}] & {Velocity} \newline [\si{m/s}] & $p_0$ \newline [\si{W}] & $p_1$ \newline [\si{W/N^2}] \\
        \hline 
         $\actuatorOne$ & 50.0 & 50.0 & 3.0 & 1.0 & 2.0 \\
         $\actuatorTwo$ & 100.0 & 100.0 & 3.0 & 2.0 & 1.5 \\
         $\actuatorThree$ & 150.0 & 150.0 & 3.0 & 3.0 & 1.5 \\
    \end{tabular}
    \vspace{1pt}

    \footnotesize\setlength{\tabcolsep}{3pt}
    \begin{tabular}{L{1.5cm} C{1.5cm} C{1.5cm} C{1.5cm}}
        \toprule
        {Battery Technology} & {Energy density} [\si{Wh/kg}] & {Unit power per cost} [\si{Wh/\$}] & {Number of cycles} \\
        \hline 
         NiMH & 100.0 & 3.41 & 500 \\
         NiH2 & 45.0 & 10.50 & 20,000 \\
         LCO & 195.0 & 2.84 & 750 \\
         LMO & 150.0 & 2.84 & 500 \\
         NiCad & 30.0 & 7.50 & 500 \\
         SLA & 30.0 & 7.00 & 500 \\
         LiPo & 150.0 & 2.50 & 600 \\
         LFP & 90.0 & 1.50 & 1,500 \\
         \bottomrule
    \end{tabular}
    \caption{Deterministic parameters of component choices~\cite{censi2017uncertainty}.}
    \label{tab:param-actuator-batteries}
\end{table}

\subsection{Deterministic and interval uncertainty models}
\label{subsec:deterministic-interval-uav}
We begin with deterministic parameters and free choice over all \I{implementations}.
Fixing the \FI{task profile}, we solve the \fixfunminres query for the co-design model in \cref{fig:uav-dp-composition} and project the resulting \glsdp to \FI{payload} versus \RI{lifetime cost} for visualization \rebuttal{(\FI{payload} is \emph{not} in the \F{functionality} requirement but part of the resulting trade-off, as noted in \cref{rmk:queries-fix-any-opt-others})}.
\rebuttal{\Cref{fig:deterministic-interval-trade-off} reports the trade-off: larger payloads force higher lifetime cost, illustrating monotonicity of the feasible relation.}
\rebuttal{For the interval baseline we perturb the nominal parameters of \cref{tab:param-actuator-batteries} by \SI{\pm 5}{\percent} into optimistic and pessimistic variants and solve the same query for each.
The envelope must be read with care: as the two problems are solved \emph{independently}, it is a \emph{wait-and-see} interpretation in which the \I{implementation} optimal for each realized instance may be chosen, so the optimal actuator\slash battery may differ between the curves.
If the designer must instead commit \emph{before} uncertainty is realized (here-and-now), the envelope does not characterize a fixed commitment: picking \II{actuator $a_1$} and \II{battery LCO} at payload \SI{2500}{g} from the optimistic curve guarantees neither feasibility nor cost under the pessimistic realization, which may require \II{actuator $a_2$}.
This motivates explicit distributions and staged policies.}

\begin{figure}[tb]
    \centering
    \includegraphics[width=\columnwidth]{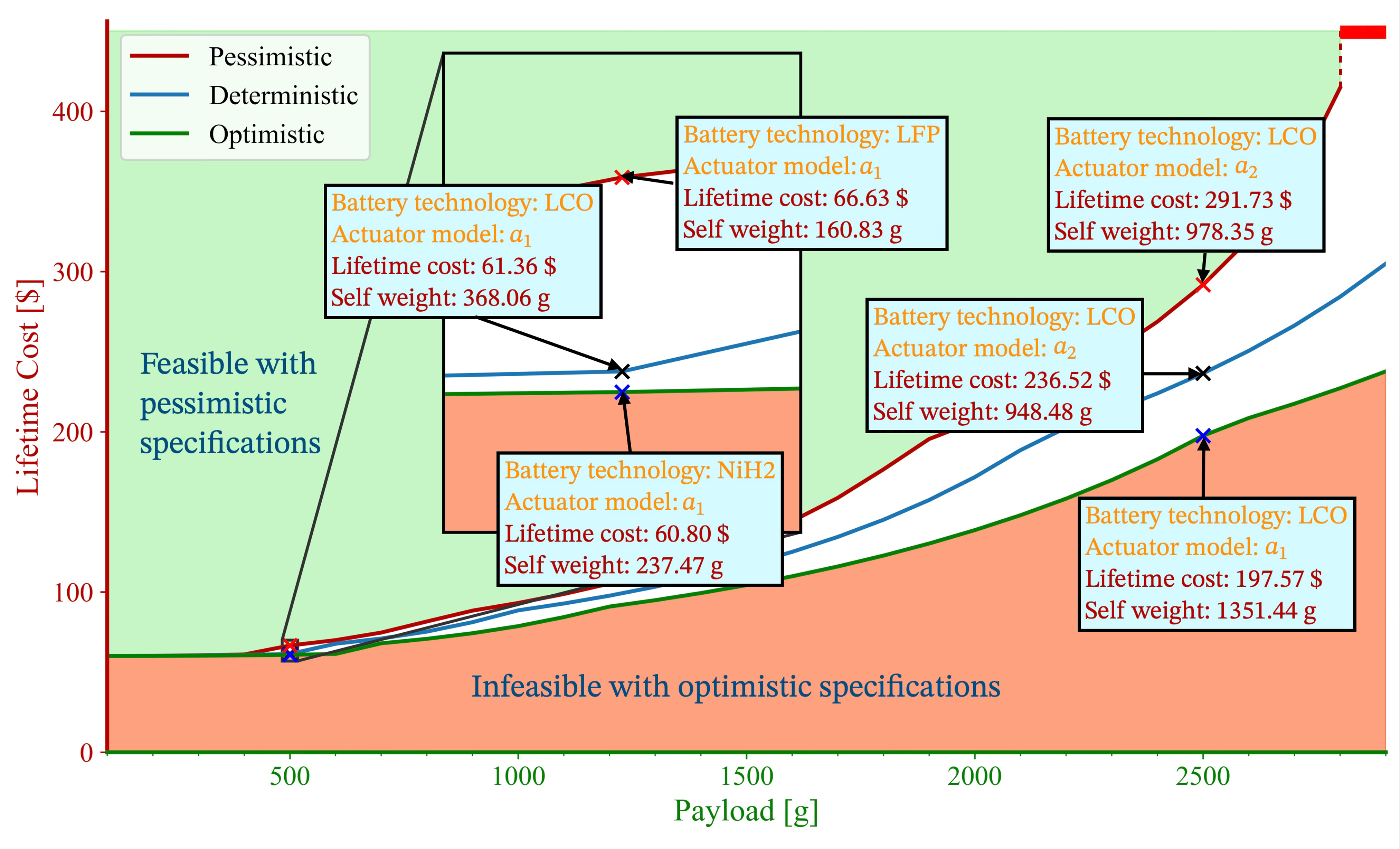}
\caption{
Required \FI{payload} versus minimal \RI{lifetime cost} for a fixed \FI{task profile}, deterministic and interval-uncertain.
Pairs above the pessimistic curve (green) are robustly feasible, below the optimistic one (orange) infeasible even optimistically; between them, feasible for some realizations only.
Labels give optimal \I{implementations} at some payloads; for large payloads the pessimistic model is infeasible, with minimal cost $+\infty$.}
    \label{fig:deterministic-interval-trade-off}
\end{figure}

\subsection{Distributional models and confidence bounds with all decisions after observing outcomes}\label{subsec:distributional-uncertainty-uav}
\label{subsec:distri-uav-max-adaptive}

We now add distributional uncertainty in (i) the task profile and (ii) the actuation and battery parameters.
\rebuttal{We first take a \emph{maximally adaptive} benchmark in which all discrete \I{implementation} choices follow the observation of the realized parameters: in the language of \cref{sec:distributional-uncertainty-co-design}, it is exactly the lifted union as \emph{free choice after observation}.
\Cref{subsec:design-results-with-adaptive-levels-uav} then restricts it to realistic adaptivity levels.}

\paragraph{Uncertain task profile}
\rebuttal{\FI{Mission frequency} and \FI{distance coverage} are deterministic; the \FI{number of delivery missions} is random, modeled with an extra \glsdp (\cref{rmk:dist-f-r-as-dps}, Task in \cref{fig:uav-adaptive-dp-battery-tech-after-actuation}).}

\paragraph{Sample space and distributions}
Let $\Omega$ be the product of the random parameters defining the task and all available choices:
\begin{equation*}
    \Omega \;=\; \Omega_T \times \Omega_\batteryset \times \Omega_\actuatorset
    \;=\; \rebuttal{\left(\Omega_N\right)} \times \Big(\prod_{b\in\mathcal{B}} \Omega_b\Big)\times \Big(\prod_{i\in\{1,2,3\}} \Omega_{a_i}\Big),
\end{equation*}
\rebuttal{with $\Omega_N = \Nats$ the number of missions, $\Omega_b$ the uncertain parameters of battery technology $b\in\mathcal{B}$ (\cref{tab:param-actuator-batteries}), and $\Omega_{a_i}=\Reals_{\ge 0}^2$, $\omega_{a_i}=\tup{p_{0,i},p_{1,i}}$, those of actuator $a_i$.}

\rebuttal{Each uncertain scalar has an independent (truncated) Gaussian marginal, with mean the value in \cref{tab:param-actuator-batteries} and variance calibrated so its \SI{90}{\percent} confidence interval is \SI{\pm 5}{\percent}; $\mu$ is the product of all marginals.}

\paragraph{Random DPs for components}
To lighten notation, write $\omega=\tup{\omega_T,\omega_\batteryset,\omega_\actuatorset}$ and display only the coordinates relevant to each component.
\rebuttal{With $\tup{v,\lift}$ the \FI{velocity} and \FI{lift} functionalities and $\tup{p,c,w}$ the \RI{power}, \RI{cost}, and \RI{weight} resources, the random performance of actuator model $a_i$ is
\begin{multline*}
    \alpha_{a_i}\colon\Omega\to \qusdpOf{\F{\RealsNonNeg^2}}{\R{\RealsNonNeg^3}},\ \omega \mapsto
    \Big\{\tup{\tup{v,\lift},\tup{\powerp,c,w}} \,\Big|\,\\
    v\le v_{a_i},\; c_{a_i}\le c,\; w_{a_i}\le w,\;
    p_{0,i}+p_{1,i}\lift^2 \le \powerp
    \Big\}.
\end{multline*}}
Since $\alpha_{a_i}$ depends on $\omega$ only through $\omega_{a_i}=\tup{p_{0,i},p_{1,i}}$, the sample-space decomposition required for constructive composition (\cref{rmk:composing-distri-dps}) applies.

As the benchmark lets the designer choose the actuator \emph{after} observing $\omega$, actuation \glsdp is the \rebuttal{pointwise deterministic} union
\begin{equation*}
    \randomVarAlpha_\Actuation = \union \mapFollowing \tup{\randomVarAlpha_\actuatorOne, \randomVarAlpha_\actuatorTwo, \randomVarAlpha_\actuatorThree}
    \colon
    \sampleSpaceOmega
    \mapArr
    \qusdpOf{\F{\RealsNonNeg^2}}{\R{\RealsNonNeg^3}}.
\end{equation*}
The battery is analogous, with $\randomVarAlpha_\Battery$ the union over technologies after observing $\outcomeomega_\batteryset$.
\rebuttal{Following \cref{rmk:dist-f-r-as-dps}, stochasticity in the \RI{task profile} is modeled with a special \glsdp (Task in \cref{fig:uav-adaptive-dp-battery-tech-after-actuation}) with one \R{resource} interface and trivial \F{functionality}: the singleton \glsposet \cite{censi2022}.
With $d_\text{r}, \nu_\text{r}$ the required distance coverage and mission frequency and $\outcomeomega_N \in \sampleSpaceOmega_N = \Nats$ the random number of deliveries,
\begin{multline*}
    \randomVarAlpha_\Task \colon \sampleSpaceOmega \to \qusdpOf{\F{\set{\singletonEl}}}{\R{\Reals_{\geq 0}^2 \setproduct \Nats}},\
    \outcomeomega \mapsto \\
    \setWithArg{\tup{\singletonEl, \tup{d, \nu, n}}}{d \geq d_\text{r}, \nu \geq \nu_\text{r}, n \geq \outcomeomega_N}.
\end{multline*}
}

\paragraph{\rebuttal{System-level random DP}}
\rebuttal{All remaining blocks of \cref{fig:uav-dp-composition} are deterministic, embedded as constant random variables.
As the decomposition of \cref{rmk:composing-distri-dps} applies, the system random variable is explicit: with $\mapf$ composing the deterministic component \glspldp of \cref{fig:uav-dp-composition} and the Task \glsdp of \cref{fig:uav-adaptive-dp-battery-tech-after-actuation},
\begin{equation}\label{eq:uav-system-rv-choosing-after-all}
    \randomVarAlpha_{\mathrm{UAV}}^\text{post} = \mapf \mapFollowing \tup{\randomVarAlpha_T, \randomVarAlpha_\Battery, \randomVarAlpha_\Actuation, \cdots}.
\end{equation}
Fixing the distribution $\distmu$ on $\sampleSpaceOmega$ yields the distributional uncertain design result $\tup{\distmu, \alpha_{\mathrm{UAV}}^{\text{post}}\colon\Omega \to \qusdpOf{\F{\Reals_{\ge 0}}}{\R{\Reals_{\ge 0}}}}$,}
with \FI{payload} the functionality and \RI{lifetime cost} the resource (projected for visualization).

\paragraph{Inner confidence bounds}
\rebuttal{Inner confidence bounds follow from monotonicity of the component models in their parameters.
Endow $\Omega$ with the product order: actuation is antitone in $\tup{p_{0,i},p_{1,i}}$ (larger values worsen feasibility) while key battery parameters (energy density, energy-per-cost) are monotone, so after taking opposites where needed the dependence is monotone and rectangular intervals $[\omega_L,\omega_U]\subseteq \Omega$ induce \glsdp intervals $[\alpha(\omega_L),\alpha(\omega_U)]$ as inner bounds.
By independence, each scalar interval at level $\rho$ gives rectangle probability $\rho^{K}$ for $K$ uncertain scalars, so an overall level $p$ needs $\rho=p^{1/K}$.}

\rebuttal{
\paragraph{Numerical design results}
To visualize results we apply the observation of \cref{ex:distri-minimal-resource}, giving distributions on one \R{resource} with all other \F{functionalities} and \R{resources} fixed, estimated by Monte Carlo sampling of $\omega$ from $\mu$; bounds are curves.

\Cref{fig:uav-battery-results-all-after} shows the \R{battery weight}\slash \R{battery cost} trade-off when requesting \FI{\SI{150}{\watt\hour}} and \FI{1000 cycles}.
The two bounds come from discrete choices, whereas the distributional result also captures (a) rare cases where an \I{implementation} off the bounds is optimal (\II{LMO}, first annotated distribution); (b) mixtures of the bounding \I{implementations} when these differ (\II{LFP} and \II{NiMH}, second); and (c) the gradual take-over of a new \I{implementation} as \R{weight} increases (third), in contrast to the abrupt jumps and large gaps of the bounds.
\Cref{fig:uav-actuator-results-all-after} shows the actuator trade-off projected onto \FI{lift} and \R{power}, \cref{fig:uav-design-results-choosing-all-after} the composed \glsuav system: components and system share \emph{the same} language and visualization (framework's modularity).
The \SI{\pm 5}{\percent} per-parameter bounds give $\rho=0.9$, hence for $K=30$ scalars a composed level of only $p=\rho^{30}\approx 0.042$.
The product construction is thus conservative while the distributional composition is \emph{exact}: empirically over \SI{90}{\percent} of realized minimal \RI{lifetime costs} fall inside the composed bound, above the \SI{4.2}{\percent} guarantee (\cref{fig:uav-design-results-choosing-all-after}).}

\begin{figure}[tb]

\centering
{
\newcommand{\subfigurescale}{0.12}
\begin{subfigure}{0.48\columnwidth}
    \includegraphics[scale=\subfigurescale]{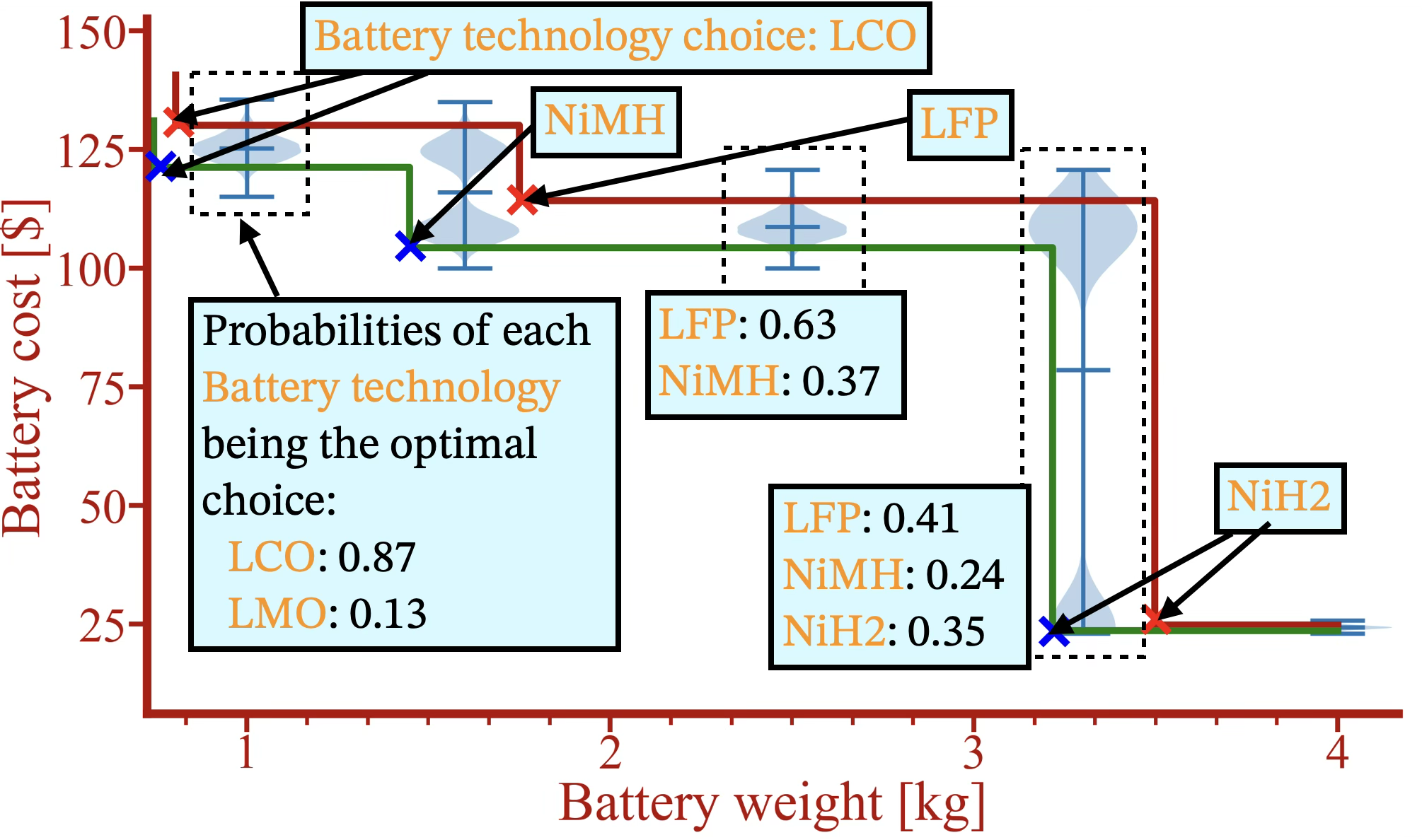}
    \subcaption{\rebuttal{Battery component, $\tup{\mu,\randomVarAlpha_\Battery}$.}}
    \label{fig:uav-battery-results-all-after}
\end{subfigure}
\begin{subfigure}{0.5\columnwidth}
    \includegraphics[scale=\subfigurescale]{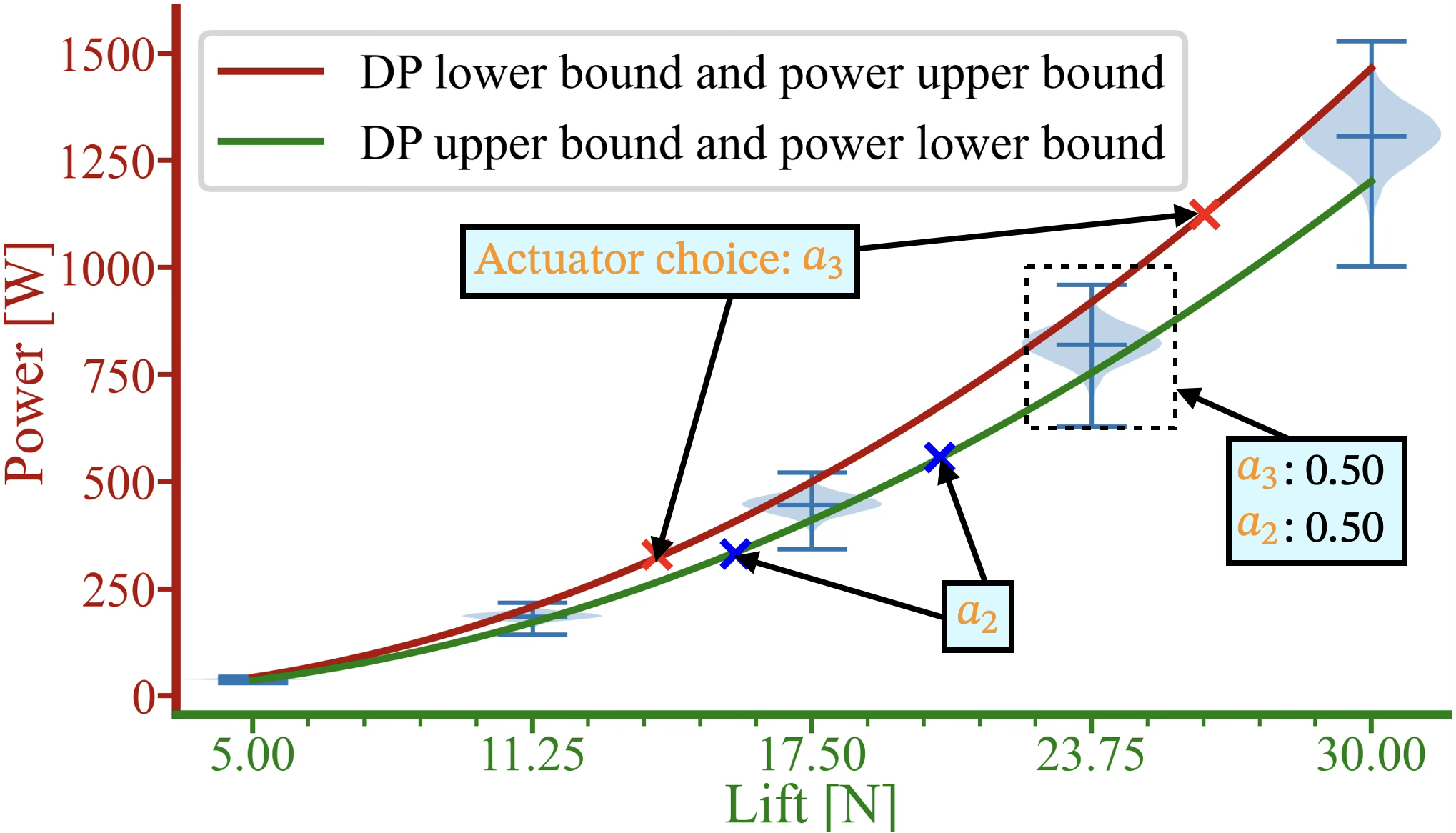}
    \subcaption{\rebuttal{Actuator component, $\tup{\mu,\randomVarAlpha_\Actuation}$.}}
    \label{fig:uav-actuator-results-all-after}
\end{subfigure}
}

\begin{subfigure}{\columnwidth}
    \includegraphics[width=\linewidth]{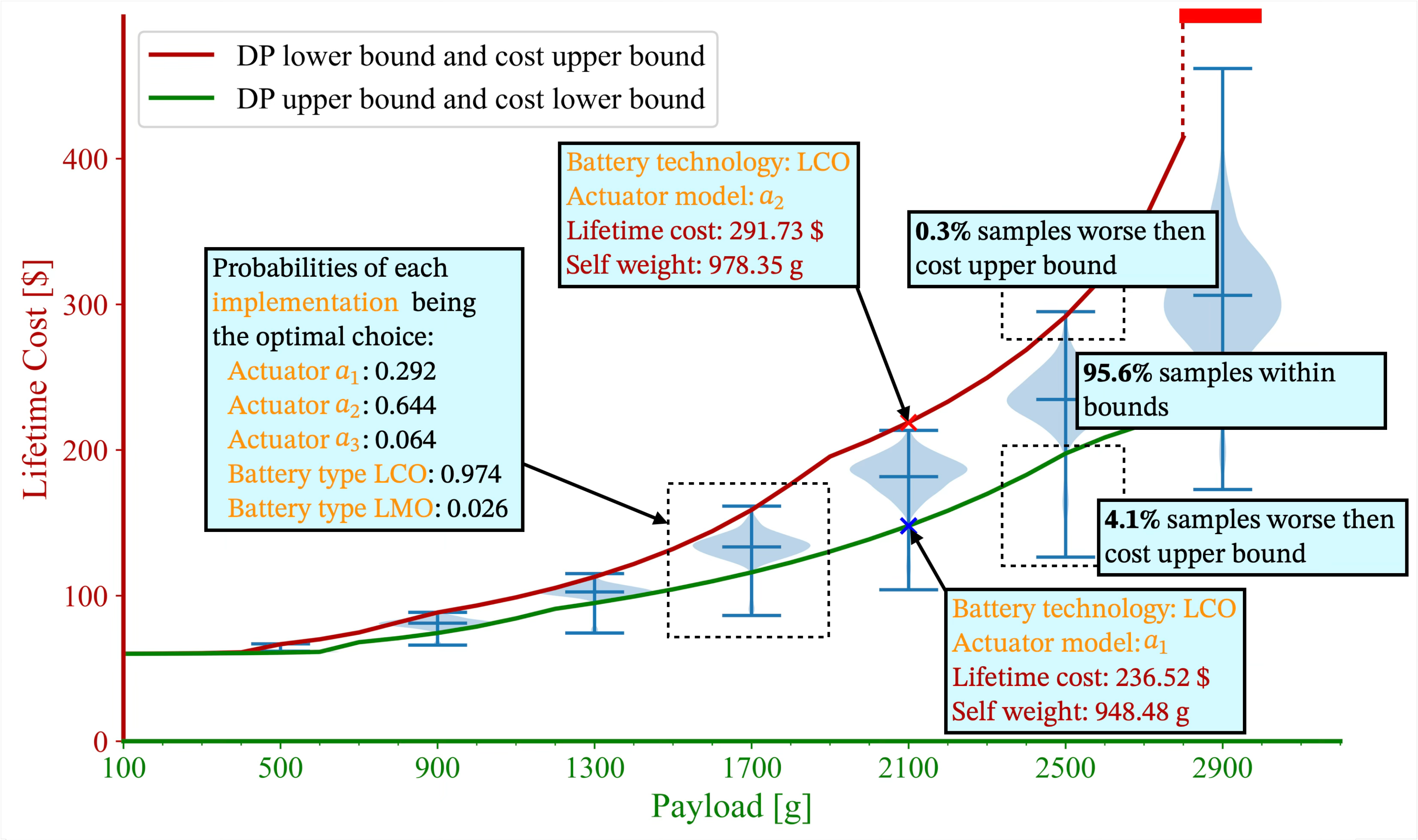}
    \subcaption{
        \rebuttal{\Glsuav system, $\tup{\mu,\randomVarAlpha_{\mathrm{UAV}}^{\text{post}}}$.}
    }
    \label{fig:uav-design-results-choosing-all-after}
\end{subfigure}

\caption{
\rebuttal{
Uncertain design results for the task-driven \glsuav and selected components, maximally adaptive benchmark.
Upper and lower curves are inner confidence bounds; violin plots are distributions of minimal \R{resource} via \cref{ex:distri-minimal-resource}.
Some distributions are annotated with the probability of each \I{implementation} being optimal and share of samples inside\slash outside the bound.
}
}
    \label{fig:uav-design-results-choosing-all-after-with-components}
\end{figure}

\subsection{Adaptive decision processes and design results}\label{subsec:design-results-with-adaptive-levels-uav}
\rebuttal{The maximally adaptive benchmark postpones every discrete \I{implementation} choice until the parameters are observed; practical processes restrict what may be postponed and what must be made first.
We compare three adaptivity levels, all allowing fully adaptive \emph{recourse} allocations (actuator operating point, battery weight) within components, differing only in how the discrete actuator and battery choices may depend on observations.}

\paragraph{Non-adaptive} 
\rebuttal{\II{Actuator} and \II{battery technology} are chosen before any observation, while actuator \II{operation} and battery \II{weight} are optimized after uncertainty is realized.
In \cref{fig:non-adaptive-uav-design} the policy $\measureKernelPolicy_{A,B}$ is a kernel with no input returning a distribution over $\actuatorset \times \batteryset$: a possibly randomized commitment to an actuator--battery pair.}
\rebuttal{Observation kernels then give distributions over the realized specifications of the chosen components, e.g.\ $\measureKernelObs_\Actuation \colon \actuatorset \to \QUSDistriOf{\Reals^5_{\geq}}$, $\actuator_i \mapsto \big\langle \distmu, \randomVarAlpha_{\actuator_i} \colon \outcomeomega \mapsto \tup{m_{\actuator_i}, c_{\actuator_i}, v_{\actuator_i}, p_{0,i}, p_{1,i}} \big\rangle$, with $p_{0,i}, p_{1,i}$ coordinates of $\outcomeomega$; $\measureKernelObs_\Battery, \measureKernelObs_\Task$ are analogous.}
The system model $\measureKernelUAV \colon \specSetS_\Actuation \measureProduct \specSetS_\Battery \measureProduct \specSetS_\Task \mapArr \QUSDistriOf{\qusdpOf{\funPosetF}{\resPosetR}}$ then returns a distribution over system-level \glspldp.
\rebuttal{
Here $\measureKernelUAV$ returns a delta distribution (a constant random variable) over $\qusdpOf{\funPosetF}{\resPosetR}$ per \Sp{specification}: with $\mapf$ as in \cref{eq:uav-system-rv-choosing-after-all} and $\mapg_T, \mapg_\Battery, \mapg_\Actuation$ returning deterministic component \glspldp for the corresponding \Sp{specifications} (notation-wise omitting other deterministic components),
$
    \measureKernelUAV
    (\specs_\Actuation, \specs_\Battery, \specs_T)
    =
    \tup{\distmu, \outcomeomega \mapsto \mapf(\mapg_T(\specs_T), \mapg_\Battery(\specs_\Battery), \mapg_\Actuation(\specs_\Actuation), \cdots)}.
$
}
Decisions inside the \glsdp (intermediate \R{resource}/\F{functionality} pairs, actuator operation, battery weight) are made with all outcomes observed, and we read off the random minimal \RI{lifetime cost} per \FI{payload} by post-composing $\mapr_{\min}$ as a deterministic kernel.

\paragraph{Partly adaptive} 
The \II{actuator} is chosen before observing its parameters, the \II{battery technology} after observing the realized actuator \Sp{specifications}.
In \cref{fig:partly-adaptive-uav-design}, $\measureKernelPolicy_\Actuation$ commits a (randomized) choice with no observation while $\measureKernelPolicy_\Battery$ maps observed actuator specifications to a distribution over technologies.
This is a three-stage re-parameterization: the third-stage choice depends on a second-stage observation, which in turn depends on the first-stage choice (\cref{fig:uav-adaptive-dp-battery-tech-after-actuation}).

\paragraph{Fully adaptive}
The parameters of \emph{all} actuator models and battery technologies are observed before selecting.
In \cref{fig:fully-adaptive-uav-design}, $\measureKernelObs_{\actuatorset,\batteryset}$ outputs a distribution over all parameters of \cref{tab:param-actuator-batteries}, from which $\measureKernelPolicy_{\Actuation,\Battery}$ picks a pair given the observed task number and \FI{payload}.
This is the policy-based counterpart of the lifted-union benchmark: always selecting the optimal pair per realized parameter and \FI{payload} reproduces the post-observation semantics.

\begin{figure}[tb]
\newcommand{\subfigscale}{0.18}

\begin{subfigure}{\columnwidth}
    \centering
    \includegraphics[scale=\subfigscale]{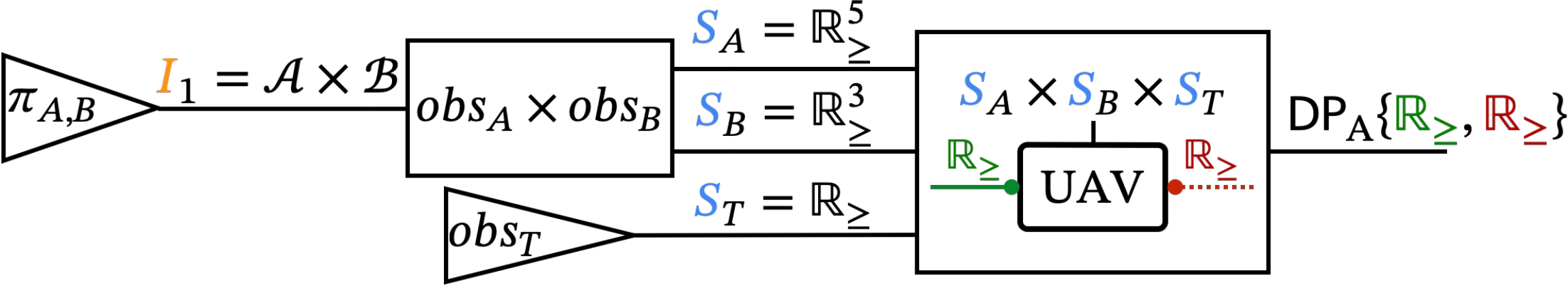}
    \subcaption{\rebuttal{Non-adaptive: \II{actuator} and \II{battery technology} chosen before any observation.}}
    \label{fig:non-adaptive-uav-design}

\end{subfigure}

\begin{subfigure}{\columnwidth}
    \centering
    \includegraphics[scale=\subfigscale]{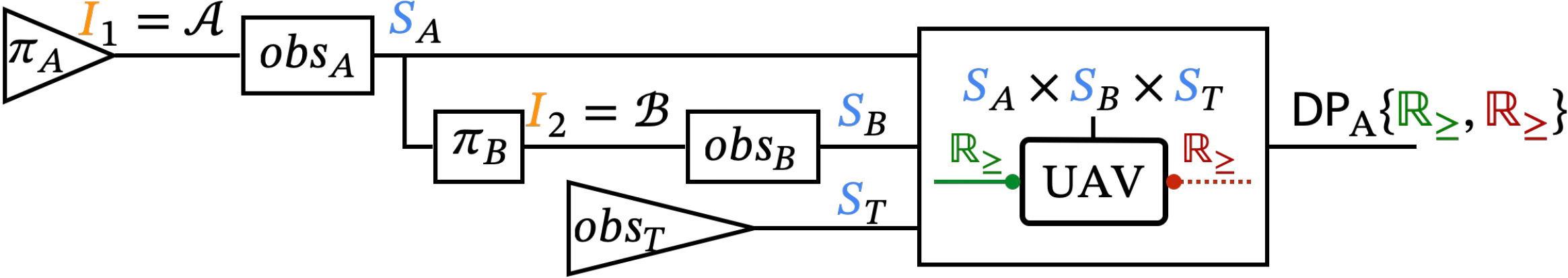}
    \subcaption{\rebuttal{Partly adaptive: \II{actuator} first, \II{battery technology} after observing its \Sp{specifications}.}}
    \label{fig:partly-adaptive-uav-design}

\end{subfigure}

\begin{subfigure}{\columnwidth}
    \centering
    \includegraphics[scale=\subfigscale]{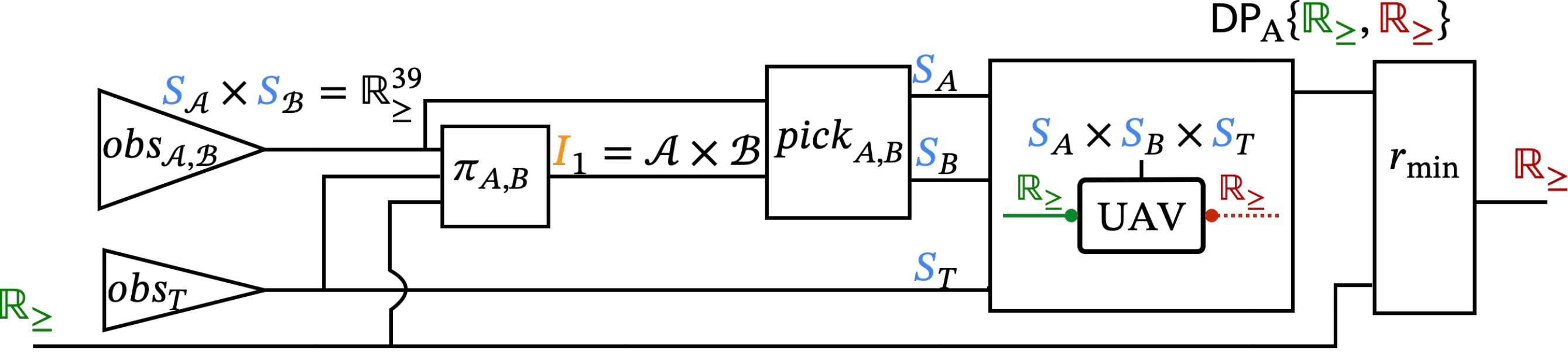}
    \subcaption{\rebuttal{Fully adaptive: all choices made after observing \Sp{specifications} and the \F{functionality} requirement.}}
    \label{fig:fully-adaptive-uav-design}

\end{subfigure}

\caption{Design processes for the \glsuav with different adaptivity levels.}
\label{fig:uav-diagrams-different-adaptive-levels}

\end{figure}

\begin{figure}[tb]
    \centering
    \includegraphics[width=\linewidth]{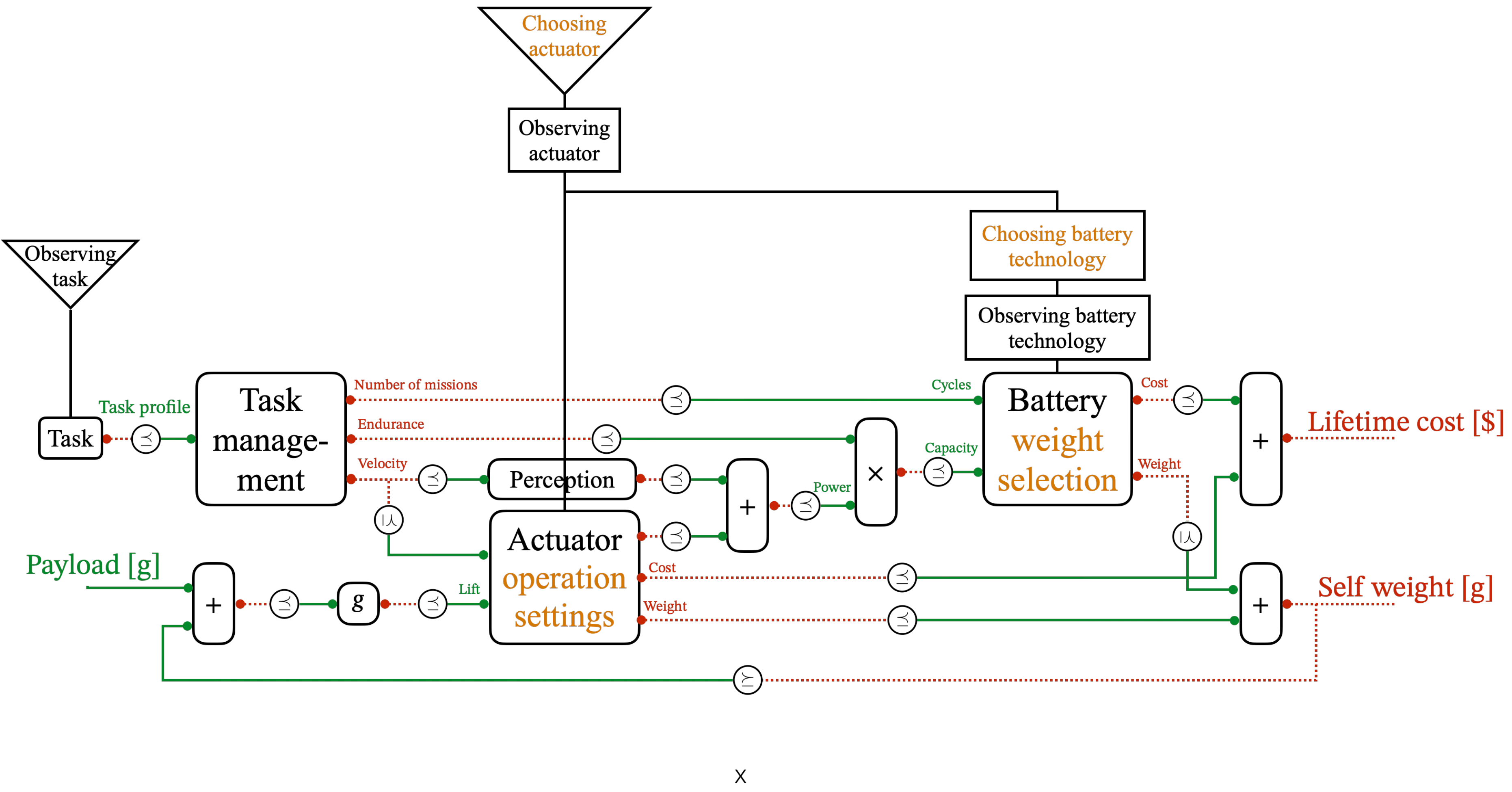}
    \caption{Partly adaptive decision process combined with the co-design decomposition of the task-driven \glsuav.
    }
    \label{fig:uav-adaptive-dp-battery-tech-after-actuation}
\end{figure}

\paragraph{Policies used for comparison}
\rebuttal{The framework allows general, including randomized, policies, which matter under chance constraints or explicit risk measures~\cite{shenChanceConstrainedProbability2023}.
We let the policies in \cref{fig:non-adaptive-uav-design,fig:partly-adaptive-uav-design} depend on \FI{payload} and use a simple deterministic template: \emph{given the observations, choose the \I{implementation} most likely to be optimal for minimal \RI{lifetime cost}}, estimating from Monte Carlo samples the probability that each implementation attains the minimum and taking the maximum a posteriori option.
This approximately optimizes expected minimal cost under each information structure.}

\paragraph{Results and interpretation}
\Cref{fig:uav-design-results-adaptive-levels} reports the distributions of minimal \RI{lifetime cost} for several \FI{payloads}.
Improved adaptivity substantially reduces expected \RI{lifetime cost} in some \FI{payload} regimes (near feasibility boundaries, where technology switching matters most), while the upper tails may not improve and can even worsen.
This is the distinction between ``average'' performance and low-probability, high-impact outcomes: exactly the trade-off interval bounds cannot quantify and distributional co-design makes explicit.

\begin{figure}[t]
    \centering
    \includegraphics[width=0.9\linewidth]{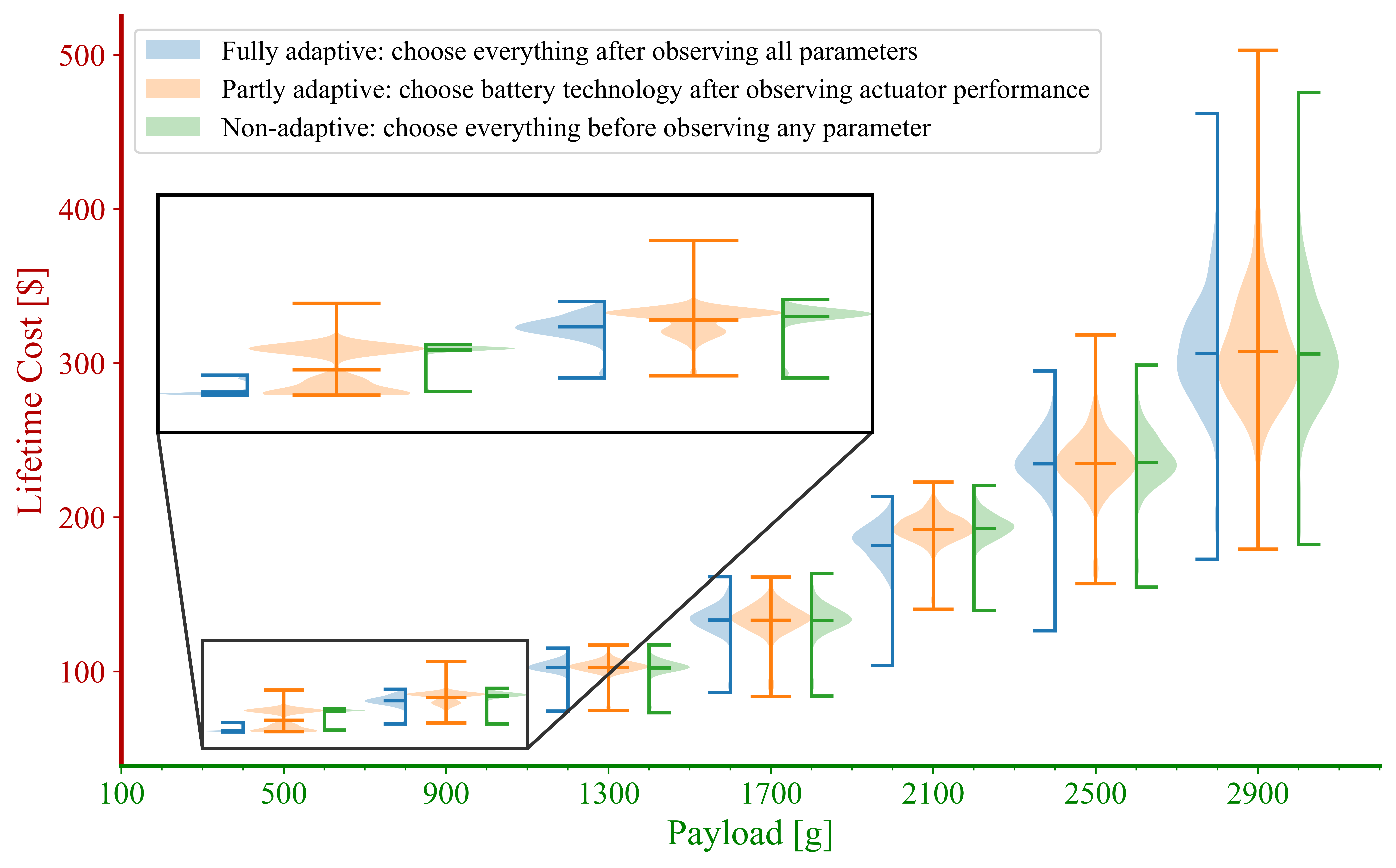}
    \caption{Design results for three adaptivity levels: distributions of minimal \RI{lifetime cost} induced by each design process at the given \F{payloads}.}
    \label{fig:uav-design-results-adaptive-levels}
\end{figure}

\section{Conclusions and future work}\label{sec:conclusion}
\rebuttal{We developed a compositional framework for co-design under \emph{distributional} uncertainty with explicit \emph{adaptive} multi-stage decision-making.
Building on monotone co-design, we used \glsxtrlong{abk:qms}\slash \glsxtrlong{abk:qus} to define distributions over \glsxtrlongpl{abk:dp} that avoid measurability pathologies: \glsdp compositions (series, parallel, feedback, union, intersection) remain measurable and lift to those distributions, feasibility probabilities are computable, stochasticity in \F{functionality}\slash \R{resources} are representable, and deterministic theory becomes a special case.
We introduced adaptive design processes via Markov-kernel re-parameterizations, supporting uncountable specification spaces, dependence between specifications and choices, and risk-aware objectives; and formalized \emph{queries} and \emph{observations} as the interface between model and engineering questions, including probability-of-feasibility, confidence bounds, and minimal-resource observations, all illustrated on an uncertain task-driven \glsuav study.}

\rebuttal{Several directions remain.
Algorithmically, we rely on Monte Carlo estimation; a natural next step is \emph{compositional} estimators and probability bounds exploiting the co-design graph, including variance-reduced sampling, concentration certificates, and deterministic bounds under monotonicity.
On the modeling side, risk measures and constraints (chance constraints, CVaR-type objectives, distributionally robust variants) can enter the query layer directly, alongside policy synthesis under partial observations, and the framework suggests an end-to-end learning loop: infer and update specification distributions from data, propagate uncertainty compositionally, and optimize adaptive policies at the system level.
Finally, given \cite{furter2025composableuncertaintysymmetricmonoidal}, an end-to-end formalization of further uncertainty layers (e.g.\ ambiguity sets \cite{duttaDistributionalRobustVerification25}) is of great interest.}

\bibliographystyle{IEEEtran}
\bibliography{bibliography}

\end{document}